\documentclass[a4paper,twoside,reqno]{amsart}
\newcommand{\Ueberschrift}{Galois sections and $p$-adic period mappings}
\newcommand{\Kurztitel}{Galois sections and $p$-adic period mappings}
\usepackage[headings]{fullpage}
\usepackage{amsmath} 
\usepackage{amssymb} 
\usepackage{amsopn}
\usepackage{amsthm}
\usepackage{amsfonts} 
\usepackage{mathrsfs}
\usepackage{mathtools}
\usepackage{stmaryrd}
\usepackage{manfnt}
\usepackage{mathdots}
\usepackage[OT2, T1]{fontenc}
\usepackage[dvips]{graphicx} \usepackage[usenames]{color}
\usepackage{xspace}
\usepackage[all]{xy}
\usepackage{dsfont}
\usepackage{ifthen}
\usepackage{marvosym}
\usepackage{wasysym}
\usepackage{marginnote}
\usepackage{float}
\usepackage{float}
\usepackage{booktabs}
\usepackage{soul}
\usepackage{colonequals}

\usepackage{subfigure}
\usepackage{tikz}
\usepackage{tikz-cd}
\usetikzlibrary{through,intersections,calc,angles,arrows}
\usetikzlibrary{shapes.geometric}
\usetikzlibrary{decorations.markings,calc}
\usepackage{pgfplots}
\pgfplotsset{compat=1.12}
\usepackage[pdfpagelabels, pdftex]{hyperref}
\hypersetup{
  pdftitle={},
  pdfauthor={},
  pdfsubject={},
  pdfkeywords={},
  colorlinks=true,    % false: boxed links; true: colored links
  linkcolor=red,     % color of internal links
  citecolor=blue,     % color of links to bibliography
  filecolor=blue,      % color of file links
  urlcolor=blue,       % color of external links
  breaklinks=true,
  bookmarksopen=true,
  bookmarksnumbered=true,
  pdfpagemode=UseOutlines,
  plainpages=false}
\DeclareRobustCommand{\gobblefive}[5]{}

\DeclareMathOperator{\rB}{B}

\DeclareMathOperator{\rF}{F}

\DeclareMathOperator{\rH}{H}

\DeclareMathOperator{\rN}{N}

\DeclareMathOperator{\rP}{P}

\DeclareMathOperator{\rR}{R}

\DeclareMathOperator{\rd}{d}

\newcommand{\bA}{{\mathbb A}}
\newcommand{\bB}{{\mathbb B}}
\newcommand{\bC}{{\mathbb C}}

\newcommand{\bE}{{\mathbb E}}
\newcommand{\bF}{{\mathbb F}}
\newcommand{\bG}{{\mathbb G}}

\newcommand{\bM}{{\mathbb M}}

\newcommand{\bP}{{\mathbb P}}
\newcommand{\bQ}{{\mathbb Q}}

\newcommand{\bS}{{\mathbb S}}
\newcommand{\bT}{{\mathbb T}}

\newcommand{\bZ}{{\mathbb Z}}

\newcommand{\ubC}{{\underline{\bC}}}

\newcommand{\cH}{{\mathscr H}}

\newcommand{\cO}{{\mathscr O}}

\newcommand{\dA}{{\mathcal A}}

\newcommand{\dE}{{\mathcal E}}
\newcommand{\dF}{{\mathcal F}}
\newcommand{\dG}{{\mathcal G}}
\newcommand{\dH}{{\mathcal H}}
\newcommand{\dI}{{\mathcal I}}

\newcommand{\dL}{{\mathcal L}}
\newcommand{\dM}{{\mathcal M}}

\newcommand{\dP}{{\mathcal P}}

\newcommand{\fX}{{\mathfrak X}}

\newcommand{\sB}{{\mathsf B}}

\newcommand{\sD}{{\mathsf D}}

\DeclareSymbolFont{cyrletters}{OT2}{wncyr}{m}{n}
\DeclareMathSymbol{\Sha}{\mathalpha}{cyrletters}{"58}
\DeclareMathOperator{\Aut}{Aut}

\DeclareMathOperator{\Cov}{Cov}
\DeclareMathOperator{\End}{End}

\DeclareMathOperator{\Hom}{Hom}
\DeclareMathOperator{\Map}{Map}

\DeclareMathOperator{\pr}{pr}

\newcommand{\Rep}{{\sf Rep}}
\newcommand{\MF}{{\sf MF}}
\newcommand{\surj}{\twoheadrightarrow} 
\DeclareMathOperator{\uAut}{\underline{Aut}}

\newcommand{\xyinj}{\ar@{^(->}}

\DeclareMathOperator{\GL}{GL}
\DeclareMathOperator{\gr}{gr}

\DeclareMathOperator{\Sp}{Sp}
\newcommand{\tr}{{\rm tr}} 
\DeclareMathOperator{\rank}{rk}

\DeclareMathOperator*{\bigwedgesquare}{{\bigwedge\nolimits^{\!\!2}}}
\DeclareMathOperator*{\bigwedgek}{{\mathit{\bigwedge\nolimits^{\!\!k}}}}
\DeclareMathOperator*{\bigwedgebullet}{{\bigwedge\nolimits^{\!\!\bullet}}}

\DeclareMathOperator{\Aff}{Aff}

\DeclareMathOperator{\Frac}{Frac}

\DeclareMathOperator{\Pic}{Pic}

\DeclareMathOperator{\Spec}{Spec}

\DeclareMathOperator{\bc}{{\rm bc}}

\DeclareMathOperator{\cHom}{\cH\!\mathit{om}}

\DeclareMathOperator{\Gal}{Gal}

\DeclareMathOperator{\Ind}{Ind}

\def\10{{\overrightarrow{10}}}
\def\01{{\overrightarrow{01}}}
\newcommand{\ab}{{\rm ab}} 
\newcommand{\alg}{{\rm alg}}
\newcommand{\an}{{\rm an}}

\DeclareMathOperator{\cl}{cl}

\newcommand{\cris}{{\rm cris}}

\newcommand{\pst}{{\rm pst}}
\newcommand{\deR}{{\rm dR}}

\newcommand{\cyc}{{\rm cyc}}

\newcommand{\et}{\mathrm{\acute{e}t}}
\newcommand{\proet}{\mathrm{pro\acute{e}t}}

\newcommand{\loc}{{\rm loc}}

\newcommand{\nr}{{\rm nr}}

\newcommand{\ph}{\varphi}

\DeclareMathOperator{\Res}{{\rm Res}}
\newcommand{\Sel}{{\rm Sel}}

\newcommand{\ov}[1]{\mbox{${\overline{#1}}$}} 
\DeclarePairedDelimiter\abs{\lvert}{\rvert}

\newtheorem{thm}{Theorem}[section]

\newtheorem{prop}[thm]{Proposition}
\newtheorem{lem}[thm]{Lemma}

\newtheorem{sublem}[thm]{Sublemma}

\newtheorem{cor}[thm]{Corollary}

\newtheorem{thmABC}{Theorem}

\theoremstyle{definition}
\newtheorem{defi}[thm]{Definition}

\theoremstyle{remark}
\newtheorem{rmk}[thm]{Remark}

\newtheorem{ex}[thm]{Example}

\newenvironment{pro*}[1][\proofname]{{\it{#1:}} }{}
\newenvironment{pro**}[1][]{{\it{#1}} }{\hfill $\square$}

\numberwithin{equation}{section}
\usepackage{enum}
\newlist{enumer}{enumerate}{2}
\setlist[enumer]{label=(\roman*),align=left,labelindent=0pt,leftmargin=*,widest = (iii)}
\newlist{enumerar}{enumerate}{1}
\setlist[enumerar]{label=\arabic*.,align=left,labelindent=0pt,leftmargin=*,widest = 8.}
\newlist{enumera}{enumerate}{2}
\setlist[enumera]{label=(\arabic*),align=left,labelindent=0pt,leftmargin=*,widest = (8)}
\newlist{enumeral}{enumerate}{2}
\setlist[enumeral]{label=(\alph*),align=left,labelindent=0pt,leftmargin=*,widest = (m)}
\definecolor{shadecolor}{RGB}{186,238,186}
\definecolor{softlimegreen}{RGB}{186,238,186}
\definecolor{limegreen}{RGB}{208,243,208}
\definecolor{questioncolor}{RGB}{135, 173, 241} %{100, 149, 237}
\definecolor{warningcolor}{RGB}{240,120,134}
\definecolor{verypaleyellow}{RGB}{255,255,194}
\usepackage{comment,xcolor}
\newcommand\Sec{\mathrm{Sec}}

\newcommand\fcov{\mathrm{f{-}cov}}

\renewcommand\st{\mathrm{st}}

\DeclareMathOperator{\Surj}{Surj}
\newcommand{\out}{\mathrm{out}} %%outer surjections
\DeclareMathOperator{\LGrass}{LGrass} %%Lagrangian Grassmannian
\DeclareMathOperator{\IGrass}{IGrass} %%Isotropic Grassmannian
\newcommand{\acase}{\ensuremath{\mathrm{(a)}}\xspace}
\newcommand{\bcase}{\ensuremath{\mathrm{(b)}}\xspace}
\newcommand{\ccase}{\ensuremath{\mathrm{(c)}}\xspace}
\newcommand{\Betti}{{\rB}}

\newcommand{\Kbar}{\overline K}
\newcommand{\Lbar}{\overline L}

\newcommand{\Rbar}{\overline R}
\newcommand{\Ubar}{\overline U}

\newcommand{\DeR}{\mathrm{DR}}

\newcommand{\llbrack}{\lbrack\!\lbrack}
\newcommand{\rrbrack}{\rbrack\!\rbrack}

\newcommand{\ssimp}{\mathrm{ss}}
\newcommand{\MIC}{\mathbf{MIC}}
\newcommand{\FMIC}{\mathbf{FMIC}}
\newcommand{\Loc}{\mathbf{Loc}}

\newcommand{\Kom}{\mathbf{Kom}}
\newcommand{\SKom}{\mathbf{SKom}}
\newcommand{\FKom}{\mathbf{FKom}}
\newcommand{\Derv}{\mathbf{D}}
\newcommand{\FDerv}{\mathbf{FD}}

\newcommand{\Seq}{\mathbf{Seq}}

\newcommand{\SPair}{\mathsf{SP}}
\DeclareMathOperator{\SMod}{\mathbf{SMod}}
\DeclareMathOperator{\FMod}{\mathbf{FMod}}
\DeclareMathOperator{\Alg}{\mathbf{Alg}}
\newcommand\BO{\mathrm{BO}}
\newcommand\size{\mathrm{size}}
\newcommand\DpH{\sD_\pH}
\newcommand\GSp{\mathrm{GSp}}

\newcommand{\Nbd}{U}
\newcommand{\pH}{{p\!\rH}}
\newcommand{\LV}{\mathrm{LV}} %%the Lawrence--Venkatesh locus
\newcommand{\relH}{\underline{\rH}} %%relative cohomology
\newcommand{\bt}{\mathbf{t}}
\newcommand{\dt}{\rd\!t}
\newcommand{\pia}{\pi^a}
\newcommand{\pif}{\pi^f}
\newcommand{\SpAut}{\uAut_{\GSp}}
\newcommand{\MM}{M}
\newcommand{\unit}{\mathbf{1}}

\begin{document}

\hrule width\hsize

\vspace{0.5cm}

\title[\Kurztitel]{\Large \Ueberschrift}
\author{L.\ Alexander Betts}
\address{L.\ Alexander Betts, Department of Mathematics, Harvard University, 1 Oxford Street, Cambridge MA 02138, USA}
\email{abetts@math.harvard.edu}
%\urladdr{https://lalexanderbetts.net}

\author{Jakob Stix}
\address{Jakob Stix, Institut f\"ur Mathematik, Goethe--Universit\"at Frankfurt, Ro\-bert-Mayer-Stra{\ss}e~{6--8},
60325 Frankfurt am Main, Germany}
\email{stix@math.uni-frankfurt.de}
%\urladdr{http://www.math.uni-frankfurt.de/~stix/}

\thanks{
The first author (LAB) was supported by the Simons Collaboration on Arithmetic Geometry, Number Theory, and Computation under grant number 550031. The second author (JS) was supported in part by the LOEWE research unit USAG and acknowledges funding by the Deutsche Forschungsgemeinschaft  (DFG, German Research Foundation) through the Collaborative Research Centre TRR 326 GAUS, project number 444845124.
}
%\subjclass[2010]{}
%\keywords{}
\date{\today}

\maketitle

\vspace{-0.4cm}

\begin{quotation} 
\noindent \small {\bf Abstract} Let $K$ be a number field not containing a CM subfield. For any smooth projective curve $Y/K$ of genus~$\geq2$, we prove that the image of the ``Selmer'' part of Grothendieck's section set inside the $K_v$-rational points~$Y(K_v)$ is finite for every finite place $v$. This gives an unconditional verification of a prediction of Grothendieck's section conjecture. In the process of proving our main result, we also refine and extend the method of Lawrence and Venkatesh, with potential consequences for explicit computations.
\end{quotation}

%%%%%%%%%%%%%%%%%%%%%%%%%%%%%%%%%%%%%%%%%%%
%%%%%% Table of contents %%%%%%%%%%%%%%%%%%
%%%%%%%%%%%%%%%%%%%%%%%%%%%%%%%%%%%%%%%%%%%
\setcounter{tocdepth}{1} {\scriptsize \tableofcontents}
%%%%%%%%%%%%%%%%%%%%%%%%%%%%%%%%%%%%%%%%%%%
%%%%%% Main body %%%%%%%%%%%%%%%%%%%%%%%%%%
%%%%%%%%%%%%%%%%%%%%%%%%%%%%%%%%%%%%%%%%%%%
\vspace{-1.2cm}
\section{Introduction}

Let~$K$ be a number field and~$Y/K$ a smooth projective (geometrically connected) curve of genus~$\geq2$. Associated to~$Y$ one has the \emph{fundamental exact sequence}
\vspace{-0.1cm}
\begin{equation}\label{eq:fes}
1 \to \pi_1^\et(Y_{\Kbar}) \to \pi_1^\et(Y) \to G_K \to 1
\end{equation}
on \'etale fundamental groups (at appropriate basepoints), where~$G_K$ is the absolute Galois group of~$K$. Every $K$-rational point~$y\in Y(K)$ gives rise to a $\pi_1^\et(Y_{\Kbar})$-conjugacy class of splittings of~\eqref{eq:fes}, and Grothendieck's \emph{Section Conjecture} predicts that every splitting of~\eqref{eq:fes} arises in this way for a unique~$y$.

In conjunction with the Mordell Conjecture, the Section Conjecture predicts that the set~$\Sec(Y/K)$ of 
conjugacy classes of
splittings of~\eqref{eq:fes} should be finite, but this remains unknown outside a handful of particular examples where it can be shown that~\eqref{eq:fes} has no splittings at all~\cite{harari-szamuely:no_abelian_sections,li-litt-salter-srinivasan:no_sections,stix:leereSC,stix:brauer-manin_for_sections}. Our aim in this paper is to prove a partial finiteness result for splittings of~\eqref{eq:fes}, unconditionally and for an arbitrary curve~$Y$, provided that $K$ contains no CM subfield. In order to state our main result, we introduce a subset of~$\Sec(Y/K)$ of local-to-global nature.

\begin{defi}
	A section~$s$ of~\eqref{eq:fes}, when restricted to a decomposition group~$G_u$ at a place~$u$ of~$K$, yields a splitting of the local fundamental exact sequence
	\[
	1 \to \pi_1^\et(Y_{\Kbar_u}) \to \pi_1^\et(Y_{K_u}) \to G_u \to 1 \,.
	\]
	We say that~$s$ is \emph{Selmer} just when $s|_{G_u}$ is the section arising from a $K_u$-rational point~$y_u\in Y(K_u)$ for each place~$u$. We write~$\Sec^\Sel(Y/K)$ for the set of Selmer sections, and for a finite place~$u$ we write
	\[
	\loc_u\colon\Sec^\Sel(Y/K) \to Y(K_u)
	\]
	for the map taking a Selmer section~$s$ to the unique $K_u$-point $y_u$ giving rise to the restricted section~$s|_{G_u}$.	
\end{defi}

Our main theorem is as follows.

\begin{thmABC}\label{thm:finite_selmer_image}
	Let~$K$ be a number field containing no CM subfield, and let~$Y/K$ be a smooth projective curve of genus~$\geq2$. Then for every finite place~$v$ of~$K$, the image of the localisation map 
	\[
	\loc_v\colon\Sec^\Sel(Y/K) \to Y(K_v)
	\]
	is finite.
\end{thmABC}

Theorem~\ref{thm:finite_selmer_image} can be rephrased in terms of the finite descent obstruction. Recall that the set of modified adelic points $Y(\bA_K)_\bullet$ of $Y$ is the product of the finite adelic points $\prod_{v \nmid \infty} Y(K_v)$ with the connected components of the infinite adelic points $\prod_{v \mid \infty} \pi_0(Y(K_v))$, see \cite[\S2]{stoll:finite_descent}. 
The finite descent locus  
\[
Y(\bA_K)_\bullet^\fcov \subseteq Y(\bA_K)_\bullet
\]
is a subset of the modified adelic points, which contains the $K$-rational points~$Y(K)$ and is conjecturally equal to them \cite[Conjecture~9.1]{stoll:finite_descent}. The relationship between the finite descent locus and the Selmer section set is that~$Y(\bA_K)_\bullet^\fcov$ is exactly the image of~$\Sec^\Sel(Y/K)$ under the total localisation map $\loc\colon\Sec^\Sel(Y/K)\to Y(\bA_K)_\bullet$ (whose~$v$th component is the map~$\loc_v$ above) \cite[Theorem~11]{harari-stix:descent_and_fundamental_sequence}. Hence Theorem~\ref{thm:finite_selmer_image} implies the following shadow of the expected finiteness of~$Y(\bA_K)_\bullet^\fcov$.

\begin{thmABC}\label{thm:finite_finite_descent}
	Let~$K$ be a number field containing no CM subfield, and let~$Y/K$ be a smooth projective curve of genus~$\geq2$. Then for every finite place~$v$ of~$K$, the projection of the finite descent locus~$Y(\bA_K)_\bullet^\fcov$ on~$Y(K_v)$ is finite.
\end{thmABC}

The history of Grothendieck's Section Conjecture is inextricably bound up with that of the Mordell Conjecture, as it was Grothendieck's original hope that a proof of the Section Conjecture would lead to a new arithmetic-homotopical proof of the Mordell Conjecture. This paper rather reverses the relationship between these two conjectures, in that our concern is to adapt techniques used to prove the Mordell Conjecture to study the otherwise mysterious section set. In this way, Theorem~\ref{thm:finite_selmer_image} demonstrates that the theory of the \'etale fundamental group is strong enough to support Mordell-like finiteness theorems, 
giving a non-trivial finiteness constraint on the section set that applies to every curve~$Y$, at least over base fields containing no CM subfield.

\subsection{Method of proof}

The method we use in the proof of Theorem~\ref{thm:finite_selmer_image} is an adaptation of the method developed by Brian Lawrence and Akshay Venkatesh in their recent new proof of the Mordell Conjecture \cite{LVinventiones}, albeit with some significant modifications necessary for our application. Let us describe our argument in broad strokes, remarking on the differences from \cite{LVinventiones} as they arise.
\smallskip

We will take an explicitly obstruction-theoretic perspective on the strategy of~\cite{LVinventiones}. Consider an \emph{abelian-by-finite family} over~$Y$, meaning a sequence
\[
X \to Y' \to Y
\]
where~$Y'\to Y$ is a finite \'etale covering and~$X\to Y'$ is a polarised abelian scheme. If~$y\in Y(K)$ is a $K$-rational point, then the \'etale cohomology groups~$\rH^i_\et(X_{y,\Kbar},\bQ_p)$ of the geometric fibre $X_{y,\Kbar}$ are Galois representations carrying extra structures: $\rH^0_\et(X_{y,\Kbar},\bQ_p)$ is an algebra and $\rH^1_\et(X_{y,\Kbar},\bQ_p)$ is a symplectic $\rH^0_\et(X_{y,\Kbar},\bQ_p)$-module in the category of $G_K$-representations. This allows us to cut out an obstruction locus as follows.

\begin{defi}[recalled as Definition~\ref{def:s-good}]\label{def:intro_locus}
	Let~$S$ be a finite set of places of~$K$, and~$p$ a prime number. A pair~$(A,V)$ of a (commutative) algebra~$A$ and a symplectic $A$-module~$V$ in the category of $\bQ_p$-linear $G_K$-representations is called \emph{$S$-good} just when:
	\begin{itemize}
		\item $A$ is unramified outside~$S$;
		\item $V$ is unramified, pure and integral of weight~$1$ outside~$S$ (see Definition~\ref{def:purity}); and
		\item $V$ is de Rham at all places over~$p$, with Hodge--Tate weights in~$\{0,1\}$.
	\end{itemize}
	
	Let $X\to Y'\to Y$ be an abelian-by-finite family, and let~$S$ be a finite set of places of~$K$ containing all places dividing~$p\infty$ and all places of bad reduction for~$X\to Y$. Let~$v$ be a finite place of~$K$, lying over the rational prime~$p$. We define the \emph{Lawrence--Venkatesh locus}
	\[
	Y(K_v)_{X,S}^\LV \subseteq Y(K_v)
	\]
	to be the set of points~$y_v\in Y(K_v)$ for which there exists an $S$-good pair~$(A,V)$ together with $G_v$-equivariant isomorphisms
	\[
	\rH^0_\et(X_{y_v,\Kbar_v},\bQ_p) \cong A|_{G_v} \hspace{0.4cm}\text{and}\hspace{0.4cm} \rH^1_\et(X_{y_v,\Kbar_v},\bQ_p) \cong V|_{G_v}
	\]
	compatible with algebra and symplectic module structures.
\end{defi}

In~\cite[Definition~7.3]{LVinventiones}, Lawrence and Venkatesh construct an abelian-by-finite family
\[
X_q \to Y'_q \to Y
\]
called the \emph{Kodaira--Parshin family}, where the parameter~$q$ is an odd prime number. Our proof of Theorem~\ref{thm:finite_selmer_image} consists of proving two statements.
\begin{enumerate}[label = (\arabic*), ref = (\arabic*)]
	\item\label{mainpart:containment} For any abelian-by-finite family~$X\to Y'\to Y$, any suitable set~$S$ and any finite place~$v\mid p$ of~$K$, the image of the localisation map $\loc_v\colon \Sec(Y/K)^\Sel\to Y(K_v)$ is contained in~$Y(K_v)_{X,S}^\LV$.
	\item\label{mainpart:finiteness} For every self-conjugate finite place~$v$ of~$K$, there exists an odd prime~$q$ such that~$Y(K_v)_{X_q,S}^\LV$ is finite for all suitable~$S$.
\end{enumerate}
In the second point, \emph{self-conjugacy} is a technical condition on the place~$v$, slightly weaker than the condition of \emph{friendliness} in \cite[Definition~2.7]{LVinventiones}; see Definition~\ref{def:friendly}. If~$K$ has no CM subfield then every finite place is self-conjugate, so points~\ref{mainpart:containment} and~\ref{mainpart:finiteness} together imply Theorem~\ref{thm:finite_selmer_image}.

\begin{rmk}
	Although not explicitly couched in obstruction-theoretic language, \cite{LVinventiones} essentially\footnote{Strictly speaking, the argument in \cite{LVinventiones} is not quite sufficient to prove~\ref{oldmainpart:finiteness}, since the analysis in \cite[\S3.4]{LVinventiones} is only valid for residue discs centred on $K$-rational points~$y_0\in Y(K)$. However, it is easy to adapt this part of the argument to cover the case when~$y_0\in Y(K_v)$ is merely $K_v$-rational, as we shall do in \S\ref{s:big_monodromy}.} proves the following two statements.
	\begin{enumerate}[label = (\arabic*${}^\circ$), ref = (\arabic*${}^\circ$)]
		\item\label{oldmainpart:containment} For any abelian-by-finite family~$X\to Y'\to Y$, any suitable set~$S$ and any finite place~$v\mid p$ of~$K$, we have $Y(K)\subseteq Y(K_v)_{X,S}^\LV$.
		\item\label{oldmainpart:finiteness} There exists a finite place~$v\mid p$ of~$K$ and an odd prime~$q$ such that $Y(K_v)_{X_q,S}^\LV$ is finite for all suitable~$S$.
	\end{enumerate}
	In particular,~\ref{mainpart:containment} (whose proof is easy) and~\ref{oldmainpart:finiteness} together are enough to prove Theorem~\ref{thm:finite_selmer_image} for \emph{one} place~$v$. In fact, with some careful bookkeeping, one can upgrade this to a proof of Theorem~\ref{thm:finite_selmer_image} for 100\% of places~$v$ when~$K$ has no CM subfield, see Remark~\ref{rmk:compare_v,q}\ref{rmkpart:100_percent}. On the other hand, there are always some places~$v$ which the original methods of~\cite{LVinventiones} do not cover, e.g.\ any~$v$ which is ramified over~$\bQ$, which divides~$2$, or which is of bad reduction for~$Y$. It is extending the methods of~\cite{LVinventiones} to cover also these places which takes most of the work in this paper.
\end{rmk}

\begin{rmk}
	We also remark that the argument in~\cite{LVinventiones} actually constrains a slightly larger obstruction locus to the one described in Definition~\ref{def:intro_locus}, namely the locus $Y(K_v)_{X,S}^{\LV^\circ}\subseteq Y(K_v)$ defined by omitting the word ``symplectic'' throughout Definition~\ref{def:intro_locus}. We always have the containment $Y(K_v)_{X,S}^\LV\subseteq Y(K_v)_{X,S}^{\LV^\circ}$, so the Lawrence--Venkatesh method as we formulate it here imposes stronger conditions on $K$-rational points (and Selmer sections).
	
	This extra efficiency is relevant if one wants to use the Lawrence--Venkatesh method to compute $Y(K)$ in practice. It seems likely that the main contributor to the running time of any implementation of the method is going to be the relative dimension of $X\to Y'$, and our formulation here typically requires abelian-by-finite families of smaller relative dimension than in~\cite{LVinventiones}. For example, if~$Y$ is a curve of genus~$3$ over~$\bQ$, then the original argument of~\cite{LVinventiones} can only prove finiteness of~$Y(\bQ)$ using the Kodaira--Parshin family $X_q\to Y$ for~$q\geq23$ (relative dimension~$\geq55$), whereas our formulation proves finiteness already when~$q=11$ (relative dimension~$25$). This is surely still out of the range of practical computation, but indicates that systematically keeping track of symplectic pairings is likely to reduce the complexity of any explicit Lawrence--Venkatesh computations.
	
	For a detailed comparison of which pairs~$(v,q)$ we prove finiteness of~$Y(K_v)_{X_q,S}^\LV$ for, compared with the original arguments of~\cite{LVinventiones}, see Remark~\ref{rmk:compare_v,q}.
\end{rmk}

\subsubsection{Assigning representations to sections}\label{ss:intro_containment}

Of the two halves of our proof of Theorem~\ref{thm:finite_selmer_image}, the proof of \ref{mainpart:containment} is relatively easy. The relative \'etale cohomology $\rR^i\!\pi_{\et*}\bQ_p$ of the abelian-by-finite family $\pi\colon X\to Y$ is a $\bQ_p$-local system on~$Y$ for the \'etale topology, so corresponds to a continuous representation~$V^i$ of~$\pi_1^\et(Y)$. Given a section~$s\colon G_K\to\pi_1^\et(Y)$, one can restrict the representations~$V^i$ along~$s$ to obtain $G_K$-representations~$V^i_s$, where~$A_s\colonequals V^0_s$ is an algebra and~$V_s\colonequals V^1_s$ is a symplectic $A_s$-module. When the section~$s$ is Selmer, one checks that the pair~$(A_s,V_s)$ is~$S$-good and witnesses that $\loc_v(s)\in Y(K_v)_{X,S}^\LV$. This argument will be spelled out carefully in \S\ref{s:proof}.

\begin{rmk}\label{rmk:other_possibilities}
	This construction is really quite general, and shows that any $v$-adic obstruction to $K$-rational points coming from a $\bQ_p$-local system on~$Y$ in fact constrains the image of the localisation map~$\loc_v\colon \Sec^\Sel(Y/K)\to Y(K_v)$. The consequences of this observation in the case of the Chabauty--Kim method are being worked out in work of the first author, Theresa Kumpitsch and Martin L\"udtke.
\end{rmk}

\subsubsection{Period maps}

It is the proof of~\ref{mainpart:finiteness} which is much more difficult, and takes up the majority of this paper. The principal extra difficulty compared with \cite{LVinventiones} is that we need to prove a finiteness result for \emph{all} self-conjugate~$v$, rather than just for \emph{some}~$v$. This means that we are not free to assume, for example, that~$v$ is a place of good reduction for $X_q\to Y'_q\to Y$.

As in \cite{LVinventiones}, the strategy relies on the machinery of \emph{period maps}. If~$X\to Y'\to Y$ is an abelian-by-finite family and~$v$ is a finite place of~$K$, then one has a period map
\[
\Phi_{y_0}\colon U_{y_0} \to \dH_{y_0}^\an
\]
where~$U_{y_0}\subseteq Y_{K_v}^\an$ is a small admissible neighbourhood of~$y_0$ in the rigid analytification of~$Y_{K_v}$ and~$\dH_{y_0}$ is the Lagrangian Grassmannian parametrising Lagrangian $\rH^0_\deR(X_{y_0}/K_v)$-submodules of~$\rH^1_\deR(X_{y_0}/K_v)$. The map $\Phi_{y_0}$ is constructed by parallel transporting the Hodge filtrations on the de Rham cohomology groups $\rH^1_\deR(X_y/K_v)$ along the Gau\ss--Manin connection (see Definition~\ref{def:period_maps} and \S\ref{ss:abf_period_map}).

The point of these period maps is that they control the variation of the local Galois representation associated to the point~$y\in U_{y_0}(K_v)$, namely
\[
\rH^1_\et(X_{y,\Kbar_v},\bQ_p),
\]
where~$p$ is the rational prime below~$v$. 
More precisely, they control the associated filtered discrete $(\varphi,N,G_v)$-module\footnote{This is also denoted~$\sD_\pst$ in some sources. In this paper, we will reserve~$\sD_\pst(V)$ to denote the (non-filtered) discrete $(\varphi,N,G_v)$-module attached to a de Rham representation~$V$, and write~$\sD_\pH(V)$ when we want to regard this as a filtered object via comparison with~$\sD_\deR(V)$. The subscript~$\pH$ stands for ``$p$-adic Hodge''. See \S\ref{sss:DpH} for precise definitions.} 
\[
\sD_\pH\big(\rH^1_\et(X_{y,\Kbar_v},\bQ_p)\big).
\]
This goes as follows. To a Lagrangian submodule~$\Phi\in\dH_{y_0}(K_v)$, one can associate a filtered discrete $(\varphi,N,G_v)$-module $\MM^1(\Phi)$, whose underlying $(\varphi,N,G_v)$-module is $\sD_\pst(\rH^1_\et(X_{y_0,\Kbar_v},\bQ_p))$, but whose filtration is the filtration on $\sD_\deR(\rH^1_\et(X_{y_0,\Kbar_v}/K_v))\cong\rH^1_\deR(X_{y_0}/K_v)$ given by~$\rF^1=\Phi$ rather than the usual Hodge filtration. This makes~$\MM^1(\Phi)$ into a symplectic module over~$\MM^0\colonequals\sD_\pH(\rH^0_\et(X_{y_0,\Kbar_v},\bQ_p))$ in the category of filtered~$(\varphi,N,G_v)$-modules. The main technical result we need is that there are isomorphisms
\[
\sD_\pH\big(\rH^0_\et(X_{y,\Kbar_v},\bQ_p)\big) \cong \MM^0 \hspace{0.4cm}\text{and}\hspace{0.4cm} \sD_\pH\big(\rH^1_\et(X_{y,\Kbar_v},\bQ_p)\big) \cong \MM^1\big(\Phi_{y_0}(y)\big)
\]
of filtered $(\varphi,N,G_v)$-modules for every~$y\in U_{y_0}(K_v)$, compatible with algebra and symplectic module structures. We find it elucidates matters to represent this diagrammatically, as the commutativity of the diagram
\begin{equation}\label{diag:intro_period_square}
\begin{tikzcd}
	Y(K_v) \arrow[r,phantom,"\supset"]\arrow[d] &[-8ex] \Nbd_{y_0}(K_v) \arrow[r,"\Phi_{y_0}"] & 
	\dH_{y_0}(K_v) \arrow[d,"{\big(\MM^0,\MM^1(-)\big)}"] \\
	\pi_0\SPair\big(\Rep_{\bQ_p}^\deR(G_v)\big) \arrow[rr,"\DpH",hook]  & & \pi_0\SPair\big(\MF(\varphi,N,G_v)\big) \,.
\end{tikzcd}
\end{equation}
where~$\pi_0\SPair$ denotes the set of isomorphism classes of pairs of an algebra~$A$ and symplectic $A$-module~$V$ in a suitable category.

\begin{rmk}
	In~\cite{LVinventiones}, it was sufficient to consider only the case that the abelian-by-finite family $X\to Y'\to Y$ has good reduction (in the strong sense of \cite[Definition~5.1]{LVinventiones}), in which case the above setup simplifies significantly. One may take~$U_{y_0}$ to be the residue disc of~$y_0$, and may work with filtered $\varphi$-modules instead of filtered discrete $(\varphi,N,G_v)$-modules. The commutativity of~\eqref{diag:intro_period_square} is then a consequence of standard facts about crystalline cohomology in families. However, outside the good reduction case, it is significantly more complicated to show commutativity of~\eqref{diag:intro_period_square}; the approach we explain in \S\ref{s:reps_in_families} uses relative $p$-adic Hodge theory as developed by Scholze \cite{scholze:relative_p-adic_hodge_theory}, combined with the potential horizontal semistability theorem of Shimizu \cite{shimizu:p-adic_monodromy}. It is possible that one could also prove commutativity of~\eqref{diag:intro_period_square} using Hyodo--Kato cohomology and alterations, but we do not know how to do this.
\end{rmk}

We can use the diagram~\eqref{diag:intro_period_square} to isolate a subset of~$\dH_{y_0}(K_v)$ corresponding to the Lawrence--Venkatesh locus as follows.

\begin{defi}\label{def:intro_period_locus}
	We define~$\dH_{y_0}(K_v)_{X,S}^\LV\subseteq\dH_{y_0}(K_v)$ to be the set of all Lagrangian $\rH^0_\deR(X_{y_0}/K_v)$-submodules $\Phi\leq\rH^1_\deR(X_{y_0}/K_v)$ for which there exists an $S$-good pair~$(A,V)$ together with isomorphisms
	\[
	M^0 \cong \sD_\pH(A|_{G_v}) \hspace{0.4cm}\text{and}\hspace{0.4cm} M^1(\Phi) \cong \sD_\pH(V|_{G_v})
	\]
	of filtered discrete $(\varphi,N,G_v)$-modules compatible with algebra and symplectic module structures.
\end{defi}

It follows from commutativity of~\eqref{diag:intro_period_square} that the image of~$U_{y_0}(K_v)\cap Y(K_v)_{X,S}^\LV$ under the period map~$\Phi_{y_0}$ is contained in~$\dH_{y_0}(K_v)_{X,S}^\LV$. So the Lawrence--Venkatesh locus $Y(K_v)_{X,S}^\LV$ will be finite as soon as the following conditions hold for all~$y_0\in Y(K_v)$:
\begin{enumerate}[label = \roman*), ref = (\roman*)]
	\item\label{condn:intro_big_monodromy} the period map $\Phi_{y_0}\colon U_{y_0}\to\dH_{y_0}^\an$ has Zariski-dense image; and
	\item\label{condn:intro_non-density} $\dH_{y_0}(K_v)_{X,S}^\LV$ is not Zariski-dense in~$\dH_{y_0}$.
\end{enumerate}
Indeed, the conjunction of these two conditions implies that~$U_{y_0}(K_v)\cap Y(K_v)_{X,S}^\LV$ is contained in the vanishing locus of a non-zero coherent sheaf of ideals on~$U_{y_0}$, so is finite for all~$y_0$. Since~$Y(K_v)$ can be covered by finitely many~$U_{y_0}(K_v)$ by compactness, we are done.
\smallskip

We now discuss the proofs of~\ref{condn:intro_big_monodromy} and~\ref{condn:intro_non-density}.

\subsubsection{Full monodromy}

Proving~\ref{condn:intro_big_monodromy} for an abelian-by-finite family~$X\to Y'\to Y$ reduces to a similar property for complex-analytic period maps, and eventually follows from a purely topological property of~$X\to Y'\to Y$ known as \emph{full monodromy}. In any case, full monodromy for Kodaira--Parshin families was already verified in~\cite{LVinventiones}, so we need say no more on this point here.

\subsubsection{The principal trichotomy}

It is~\ref{condn:intro_non-density} which is the delicate condition. As an illustration of the approach, let us consider the set $Y(K_v)_{X,S,\ssimp}^\LV$ defined as in Definition~\ref{def:intro_locus}, but with the additional requirement that~$V$ is semisimple as a representation of~$G_K$. We also write~$\dH_{y_0}(K_v)_{X,S,\ssimp}^\LV$ for the corresponding subset of~$\dH_{y_0}(K_v)$ as in Definition~\ref{def:intro_period_locus}. The key theoretical input is a lemma of Faltings, which implies that there are only finitely many $S$-good pairs~$(A,V)$ of prescribed dimensions, up to isomorphism. If we write~$H/\bQ_p$ for the $\bQ_p$-algebraic group of (non-filtered) symplectic $(\varphi,N,G_v)$-module automorphisms of~$\sD_\pst\big(\rH^1_\et(X_{y_0,\Kbar_v},\bQ_p)\big)$, see Definition~\ref{def:phi_automorphisms}, then there is a natural action of~$H$ on the Weil restriction $\Res^{K_v}_{\bQ_p}\dH_{y_0}$, whose orbits are exactly the fibres of the right-hand vertical map in~\eqref{diag:intro_period_square}. So Faltings' Lemma implies that~$\dH_{y_0}(K_v)_{X,S,\ssimp}^\LV$ is contained in a finite number of~$H(\bQ_p)$-orbits, and hence is not Zariski-dense in~$\dH_{y_0}$ as soon as the inequality
\[
\dim_{\bQ_p}H < \dim_{K_v}\dH_{y_0}
\]
holds. This condition is relatively easy to arrange in practice, since both dimensions can be controlled rather precisely. So if one could arrange that $X\to Y'\to Y$ had full monodromy and the above inequality held for every~$y_0\in Y(K_v)$, then one would have proved finiteness of $Y(K_v)_{X,S,\ssimp}^\LV$.
\smallskip

What is less clear \emph{a priori} is how to adapt this approach to also cover non-semisimple $S$-good pairs, and overcoming this is the key insight of \cite{LVinventiones}. Lawrence and Venkatesh proved that when the place~$v$ is chosen appropriately (``friendly'' in their terminology), then being $S$-good imposes strong restrictions on the isomorphism class of a pair~$(A,V)$: weaker than being semisimple, but still strong enough to make some version of the above argument work. The property of friendly places~$v$ which enables this is that the average Hodge--Tate weight at~$v$ of a $G_K$-representation~$V$ is determined by its weight as a global Galois representation. So if~$(A,V)$ is an $S$-good pair and $v$ is friendly, then one obtains numerical constraints on the Hodge--Tate weights of any subrepresentation of~$V$, which ultimately restricts the ways in which~$V$ can fail to be semisimple.

We will formalise this argument in this paper as what we call the \emph{principal trichotomy} for $S$-good pairs, which says that when~$v$ is self-conjugate, any~$S$-good pair must satisfy one of three rather technical properties, the first of which is a weak simplicity condition and the other two of which govern the possible failures of simplicity. The precise statement is as follows.

\begin{prop}[The Principal Trichotomy (=Proposition~\ref{prop:principal_trichotomy})]\label{prop:intro_principal_trichotomy}
	Let~$(A,V)$ be a pair of an algebra~$A$ and symplectic $A$-module~$V$ in the category of $G_K$-representations which is $S$-good for some~$S$, and where~$V$ has constant rank~$2d>0$ as an $A$-module. Let~$\Sigma = \Hom_{\Alg(\bQ_p)}(A,\bQ_p)$ be the finite $G_K$-set of $\bQ_p$-algebra homomorphisms $A\to\bQ_p$. For~$\psi\in\Sigma$, we write~$G_\psi\leq G_K$ for its stabiliser in~$G_K$ and~$G_{w_\psi}=G_\psi\cap G_v$ for its stabiliser in~$G_v$.
	
	Suppose that~$v$ is self-conjugate. Then at least one of the following occurs:
	\begin{enumerate}[label=\alph*),ref=(\alph*)]
		\item there is a $\psi \in \Sigma$ such that $[G_v:G_{w_\psi}]\geq4$ and $\bQ_p\otimes_{A,\psi}V$ has no non-zero isotropic $G_\psi$-subrepresentation; or
		\item there is a $\psi \in \Sigma$ such that $[G_v:G_{w_\psi}]\geq4$ and~$\bQ_p\otimes_{A,\psi}V|_{G_v}$ has a non-zero isotropic $G_{w_\psi}$-subrepresentation~$W$ whose average Hodge--Tate weight is $\geq1/2$; or
		\item the number of $\psi \in \Sigma$ satisfying $[G_v:G_{w_\psi}]<4$ is $\geq\frac1{d+1}\dim_{\bQ_p}(A)$.
	\end{enumerate}
\end{prop}
The principal trichotomy gives a decomposition (not necessarily disjoint)
\[
\dH_{y_0}(K_v)_{X,S}^\LV = \dH_{y_0}(K_v)_{X,S,\acase}^\LV\cup\dH_{y_0}(K_v)_{X,\bcase}^\LV\cup\dH_{y_0}(K_v)_{X,\ccase}^\LV
\]
according to which of the conditions~\acase, \bcase or \ccase is satisfied for the pair~$(A,V)$, and the proof of~\ref{condn:intro_non-density} then amounts to showing that, for suitably chosen~$X$, each of these three sets is not Zariski-dense in~$\dH_{y_0}(K_v)_{X,S}^\LV$. 

Using a similar argument to that sketched for~$Y(K_v)_{X,S,\ssimp}^\LV$ above, we show that $\dH_{y_0}(K_v)_{X,S,\acase}^\LV$ is \emph{never} Zariski-dense in~$\dH_{y_0}$ for any~$X$, and a more explicit version of the same argument shows the same for~$\dH_{y_0}(K_v)_{X,\bcase}^\LV$. For~$\dH_{y_0}(K_v)_{X,\ccase}^\LV$ the strategy is different: one shows that for suitably chosen~$q$, the set $\dH_{y_0}(K_v)_{X_q,\ccase}^\LV$ is actually empty. Put all together, this shows that when~$v$ is self-conjugate and~$q$ is chosen suitably, then 
\[
\dH_{y_0}(K_v)_{X_q,S}^\LV=\dH_{y_0}(K_v)_{X_q,S,\acase}^\LV\cup\dH_{y_0}(K_v)_{X_q,\bcase}^\LV
\]
is not Zariski-dense in~$\dH_{y_0}$, completing the proof of~\ref{condn:intro_non-density} for the Kodaira--Parshin family with these parameters~$q$. Since the Kodaira--Parshin family always has full monodromy, this finishes the proof of Theorem~\ref{thm:finite_selmer_image}.

\begin{rmk}
	The principal trichotomy is not stated explicitly in \cite{LVinventiones}, but a trichotomy of sorts appears implicitly in the structure of their proof. Roughly speaking, pairs of type~\acase correspond to those points dealt with in \cite[Lemma~6.2]{LVinventiones}, pairs of type~\bcase correspond to those dealt with in \cite[Lemma~6.1]{LVinventiones}, and pairs of type~\ccase correspond to those dealt with in the proof of \cite[Theorem~5.4]{LVinventiones}. This correspondence is not quite exact: we have adjusted the division slightly from \cite{LVinventiones} to make full use of the symplectic structure.
\end{rmk}

\subsection{Overview of sections}

The structure of the paper is as follows. We begin in Section \ref{s:reps} by recalling some basics on Galois representations and $p$-adic Hodge theory \`a la Fontaine. The key points in this section are the dimension bounds on automorphism groups of $(\varphi,N,G_v)$-modules (Proposition~\ref{prop:dim_of_aut}), as well as the definition of self-conjugate places (Definition~\ref{def:friendly}) and the consequences for numerics of Hodge--Tate weights (Corollary~\ref{cor:friendly_representations}). In Section \ref{s:reps_in_families} we recall the construction of $v$-adic period maps for smooth proper families $X\to Y$, and prove that these $v$-adic period maps control the variation of local Galois representations. This section presents the most significant departure from \cite{LVinventiones} since we need this machinery for a general smooth proper family, not just one with good reduction. Thus, rather than using crystalline cohomology as a bridge between \'etale and de Rham cohomology, we are forced instead to use analytic techniques and relative $p$-adic Hodge theory.

Section~\ref{s:abelian-by-finite} recalls the definition of abelian-by-finite families, and describes the symplectic module structure on the cohomology of their fibres. It also introduces the Lagrangian Grassmannian~$\dH_{y_0}$ above, and uses the results of~Section \ref{s:reps_in_families} to justify commutativity of the diagram~\eqref{diag:intro_period_square}. Section~\ref{s:big_monodromy} then gives the proof of~\ref{condn:intro_big_monodromy}, closely paralleling the discussion in \cite[\S3.4]{LVinventiones} except that the centre~$y_0\in Y(K_v)$ of the disc~$U_{y_0}$ need not be $K$-rational.

The main part of the argument comes in Section \ref{s:locus}, which proves the principal trichotomy (Proposition~\ref{prop:principal_trichotomy}), and then uses this to address~\ref{condn:intro_non-density} along the lines sketched above. The proof of~\ref{mainpart:finiteness} is assembled at the end of this section.

Section~\ref{s:proof} then tackles part~\ref{mainpart:containment}, following the construction outlined in \S\ref{ss:intro_containment} above. This section is largely independent of the rest of the paper, so may be read first if the reader wishes.

\subsection*{Acknowledgements}

We are grateful to Owen Gwilliam, Kiran Kedlaya, Minhyong Kim, Mark Kisin, Dmitri Pavlov, Koji Shimizu, Xavier Xarles, and Aled Walker for helpful discussions about various parts of this paper. We particularly thank Peter Scholze for taking the time to explain to us many technical aspects of the relative $p$-adic Hodge theory developed in \cite{scholze:relative_p-adic_hodge_theory}.

\section{Preliminaries: Galois representations and \texorpdfstring{$(\varphi,N,G_v)$-modules}{p-adic Hodge theory}}\label{s:reps}

We collect in this section several results on local and global $\bQ_p$-linear Galois representations which will be used in the sequel. This is essentially a rephrasing of the material in \cite[\S2]{LVinventiones}, except that we treat a few subjects, namely $(\varphi,N,G_v)$-modules and self-conjugate places, in greater generality. This is so that we can avoid good reduction hypotheses when we set up the machinery of Lawrence--Venkatesh, which will be important in our proof of the main theorem. The reader may wish to skip this section on a first reading, referring back to it as needed.

Throughout this section, and indeed the whole paper, we fix a number field~$K$ and an algebraic closure~$\Kbar$ of~$K$; we write~$G_K=\Gal(\Kbar|K)$ for the absolute Galois group of~$K$. For each place~$u$ of~$K$, we also fix an algebraic closure~$\Kbar_u$ of the completion~$K_u$ of~$K$ at~$u$ and a $K$-embedding~$\Kbar\hookrightarrow\Kbar_u$, allowing us to view the absolute Galois group~$G_u=\Gal(\Kbar_u|K_u)$ as a decomposition group~$G_u\subset G_K$. We write~$\cO_u$ for the ring of integers of~$K_u$, write $k_u$ for the residue field, and set $q_u\colonequals\# k_u$.

We also fix a prime number~$p$---all representations are to be~$\bQ_p$-linear---and reserve the letter~$v$ for a $p$-adic place of~$K$. We sometimes permit ourselves to write~$K_v$ for a finite extension of~$\bQ_p$, not necessarily arising as the completion of a specific number field~$K$. The absolute Galois group of a finite extension~$L_w/K_v$ contained in~$\Kbar_v$ is denoted~$G_w\leq G_v$.

\subsection{Global Galois representations, purity, and the symplectic Faltings' Lemma}

\begin{defi}[Purity]\label{def:purity}
Recall that a \textit{$q$-Weil number of weight $n$} is an algebraic number $\alpha$ such that for all complex embeddings $\iota$ the absolute value of~$\alpha$ is $\abs{\iota(\alpha)} = q^{n/2}$.

Let~$u$ be a place of the number field~$K$ not dividing~$p\infty$. We say that an unramified representation~$V$ of~$G_u$ is \emph{pure of weight~$n\in\bZ$} just when the eigenvalues of the geometric Frobenius are $q_u$-Weil numbers of weight $n$, where $q_u$ is the size of the residue field of $K_u$. We say moreover that~$V$ is \emph{integral} just when the monic characteristic polynomial of the geometric Frobenius has coefficients in~$\bZ$.

We say that a representation~$V$ of~$G_K$ is \emph{unramified and pure} (resp.\ \emph{unramified, pure and integral}) at~$u$ just when its restriction to the decomposition group~$G_u$ is.
\end{defi}

\begin{ex}\label{ex:cohomology_pure}
Let~$X/K$ be a smooth proper variety over the number field~$K$. Then there is a finite set~$S$ of places of~$K$ outside which~$X$ has good reduction, i.e.\ $X$ is the generic fibre of a smooth proper~$\cO_{K,S}$-scheme~$\fX$. Enlarging~$S$ if necessary, we may assume that it contains all places dividing~$p\infty$.

Then for all~$n\geq0$, the \'etale cohomology $\rH^n_\et(X_{\Kbar},\bQ_p)$ is unramified, pure and integral of weight~$n$ at all places outside~$S$. Unramifiedness follows from smooth proper base change for \'etale cohomology, and purity and integrality from Deligne's proof of the Weil Conjectures \cite[Th\'eor\`eme~1.6]{deligne:weil_i}.
\end{ex}

It turns out that purity imposes strong restrictions on a Galois representation~$V$. For example, it was proved by Faltings \cite[Satz~5]{faltings:endlichkeitssaetze}\cite[Lemma~2.3]{LVinventiones} that for a given dimension~$d$, weight~$n\geq0$ and finite set~$S$ of places of~$K$, there are only finitely many $d$-dimensional semisimple representations~$V$ of~$G_K$ which are unramified, pure and integral of weight~$n$ outside~$S$. The tricky word here is ``semisimple'': it is in general quite easy to produce Galois representations which are unramified, pure and integral outside a finite set of primes (see Example~\ref{ex:cohomology_pure}), but it is very hard to show that these representations are semisimple. Even in the case of an abelian variety~$A/K$, semisimplicity of~$\rH^n_\et(A_{\Kbar},\bQ_p)$ is part of the Tate Conjecture, and is only known thanks to work of Faltings~\cite[Resultat~(a)]{faltings:endlichkeitssaetze}. 

Instead of using Faltings' Lemma~\cite[Lemma~2.3]{LVinventiones}, we will use a variant thereof for symplectic representations. By a \emph{symplectic Galois representation} we mean a triple $(V,L,\omega)$ consisting of two Galois representations~$V$ and~$L$, the latter being of dimension~$1$, and a Galois-equivariant perfect alternating pairing $\omega\colon\bigwedge^2V\to L$. We usually have~$L=\bQ_p(-1)$, and we often abbreviate~$(V,L,\omega)$ to~$(V,\omega)$ or just~$V$ for short. If we pick bases of~$V$ and~$L$ for which~$\omega$ is the standard symplectic form on~$V=\bQ_p^{\oplus2d}$, then the Galois action on~$V$ is given by a continuous group homomorphism $\rho\colon G_K\to\GSp_{2d}(\bQ_p)$, where~$\GSp_{2d}$ is the subgroup of~$\GL_{2d}$ preserving the standard symplectic form up to scalar factors of similitude.

\begin{defi}[$\GSp$-irreducibility, {\cite[p905]{LVinventiones}}]\label{def:sp-irred}
A non-zero symplectic representation~$(V,L,\omega)$ of~$G_K$ is called \emph{$\GSp$-irreducible} just when~$V$ has no non-zero $G_K$-stable isotropic subspace (subspace on which~$\omega$ vanishes). This amounts to saying that the representation does not factor over a nontrivial parabolic subgroup of the general symplectic group~$\GSp$.
\end{defi}

\begin{lem}[Symplectic Faltings' Lemma {\cite[Lemma~2.6]{LVinventiones}}]\label{lem:faltings_symplectic}
Let~$S$ be a finite set of places of~$K$ containing all places dividing~$p\infty$, and let $n\geq0$ and $d>0$ be integers. Then there are, up to isomorphism, only finitely many $\GSp$-irreducible symplectic representations $(V,L,\omega)$ of~$G_K$ where~$V$ has rank~$2d$ and is unramified, pure and integral of weight~$n$ outside~$S$.
\end{lem}

\begin{rmk}
Our using Lemma~\ref{lem:faltings_symplectic} differs slightly from the argument in~\cite{LVinventiones}, which instead uses Faltings' Lemma~\cite[Lemma~2.3]{LVinventiones} in the original form. 
\end{rmk}

\subsection{Local Galois representations and $(\varphi,N,G_v)$-modules}

While the theory of global Galois representations is very complicated, the theory of local Galois representations is much better understood, thanks to the work of Fontaine as we now recall.

\subsubsection{De Rham representations}

Fontaine's theory identifies a certain class of $G_v$-representations, called the \emph{de Rham representations}, which turns out to contain all representations coming from geometry. Fontaine defines a certain $\Kbar_v$-algebra~$\sB_\deR=\sB_\deR(\Kbar_v)$, called the \emph{ring of de Rham periods}, which is endowed with an action of the Galois group~$G_v$ restricting to the tautological action on~$\Kbar_v$. For a representation~$V$ of~$G_v$, one has the $K_v$-vector space 
\[
\sD_\deR(V)\colonequals (\sB_\deR\otimes_{\bQ_p}V)^{G_v},
\]
whose $K_v$-dimension is at most $\dim_{\bQ_p}(V)$. One says that~$V$ is \emph{de Rham} just when equality holds: 
\[
\dim_{K_v}(\sD_\deR(V))=\dim_{\bQ_p}(V).
\]
The de Rham representations form a full $\otimes$-subcategory\footnote{In this paper, a  $\otimes$-category/$\otimes$-functor/$\otimes$-natural transformation means a symmetric monoidal category/functor/nat\-ural transformation.}
\[
\Rep_{\bQ_p}^\deR(G_v)\subseteq\Rep_{\bQ_p}(G_v)
\]
closed under subobjects \cite[Th\'eor\`eme~3.8(ii)]{fontaine:representations_semi-stables}, and the assignment $V\mapsto\sD_\deR(V)$ is $\otimes$-functorial in de Rham representations~$V$.

Moreover, the period ring $\sB_\deR$ is a complete discretely valued field and as such carries an exhaustive, separated $G_v$-stable filtration by fractional ideals of its valuation ring. This induces a $K_v$-linear filtration on $\sD_\deR(V)$ for every representation~$V$. This is called the \emph{Hodge filtration} and denoted~$\rF^\bullet = \rF^\bullet\!\sD_\deR(V)$. The Hodge filtration is $\otimes$-functorial in~$V$ for de Rham representations, see 
\cite[\S3.8]{fontaine:representations_semi-stables}.

\subsubsection{$(\varphi,N,G_v)$-modules} 
\label{sss:DpH}

One of the most fundamental results in $p$-adic representation theory is that the category of de Rham representations can be described  in terms of explicit semilinear-algebraic objects known as \emph{filtered discrete $(\varphi,N,G_v)$-modules}. We recall the definition from \cite{fontaine:representations_semi-stables}.

\begin{defi}[{\cite[\S4.2.1]{fontaine:representations_semi-stables}}]
\label{def:discretephiNGvmodule}
Let~$\bQ_p^\nr$ denote the maximal unramified extension of~$\bQ_p$ contained in~$\Kbar_v$. A \emph{$(\varphi,N,G_v)$-module} is a $\bQ_p^\nr$-vector space $D$ endowed with:
\begin{enumerate}[label=(\roman*)]
	\item a bijective semilinear \emph{crystalline Frobenius} $\varphi$ (acting as absolute Frobenius on scalars);
	\item a $\bQ_p^\nr$-linear \emph{monodromy operator} $N$; and
	\item a semilinear action of the Galois group $G_v$ (with respect to the natural action of $G_v$ on $\bQ_p^\nr \subset \Kbar_v$).
\end{enumerate}
These are required to satisfy:
\begin{enumerate}[label=(\roman*),resume]
	\item $N\circ\varphi=p\cdot\varphi\circ N$; and
	\item both $N$ and $\varphi$ commute with the action of $G_v$.
\end{enumerate}
The \emph{dimension} of $D$ is its dimension as $\bQ_p^\nr$-vector space, and $D$ is said to be \emph{discrete} just when the point-stabilisers of the $G_v$-action on~$D$ are open in~$G_v$.
\end{defi}

\begin{defi}[{\cite[\S4.3.2]{fontaine:representations_semi-stables}}]
\label{def:filtereddiscretephiNGvmodule}
A \emph{filtered discrete $(\varphi,N,G_v)$-module} is a tuple~$D=(D_\pst,D_\deR,c_\BO)$ consisting of 
\begin{enumerate}[label=(\roman*)]
	\item a $(\varphi,N,G_v)$-module~$D_\pst$, 
	\item a $K_v$-vector space~$D_\deR$ endowed with an exhaustive, separated, decreasing $K_v$-linear filtration~$\rF^\bullet$, and 
	\item a $\Kbar_v$-linear $G_v$-equivariant \emph{comparison isomorphism}
\[
c_\BO\colon \Kbar_v\otimes_{\bQ_p^\nr}D_\pst \xrightarrow\sim \Kbar_v\otimes_{K_v}D_\deR \,.
\]
\end{enumerate}
The comparison isomorphism~$c_\BO$ ensures that~$D_\pst$ is indeed discrete. The filtration~$\rF^\bullet$ is referred to as the \emph{Hodge filtration} on~$D$.
The collection of all filtered discrete~$(\varphi,N,G_v)$-modules forms a $\otimes$-category 
\[
\MF(\varphi,N,G_v)
\]
with respect to the obvious morphisms and tensor products.
\end{defi}

\begin{ex}
	The archetypal example of a filtered discrete $(\varphi,N,G_v)$-module comes from crystalline cohomology. Suppose that $\fX$ is a proper smooth $\cO_v$-scheme, with special fibre~$\fX_0$ and generic fibre~$X$. Write~$K_{v,0}$ for the maximal unramified subfield of~$K_v$, which is the same as the fraction field of the ring $W(k_v)$ of Witt vectors of the residue field~$k_v$. One then has the crystalline cohomology 
	\[
	\rH^i_\cris(\fX_0/K_{v,0})\colonequals K_{v,0}\otimes_{W(k_v)}\rH^i_\cris(\fX_0/W(k_v))
	\]
of the special fibre, which comes with a semilinear crystalline Frobenius~$\varphi$, and the de Rham cohomology $\rH^i_\deR(X/K_v)$ of the generic fibre, which comes with a Hodge filtration~$\rF^\bullet$. The two are related by the Berthelot--Ogus isomorphism \cite[Theorem~2.4]{berthelot-ogus:crystalline-de_rham}
	\[
	K_v\otimes_{K_{v,0}}\rH^i_\cris(\fX_0/K_{v,0}) \xrightarrow\sim \rH^i_\deR(X/K_v) \,.
	\]
	One thus obtains a filtered discrete $(\varphi,N,G_v)$-module by taking $D_\pst=\bQ_p^\nr\otimes_{K_{v,0}}\rH^i_\cris(\fX_0/K_{v,0})$ with induced $\varphi$- and $G_v$-action (and $N=0$), taking $D_\deR=\rH^i_\deR(X/K_v)$, and taking $c_\BO$ to be the isomorphism obtained from the Berthelot--Ogus isomorphism above by base change to $\Kbar_v$. More refined versions of this construction yield examples of filtered discrete $(\varphi,N,G_v)$-modules where $N\neq0$ (arising from~$X$ with bad semistable reduction), or where the action of $G_v$ is not just the natural action on $\bQ_p^\nr\otimes_{K_{v,0}}V$ for some $K_{v,0}$-vector space~$V$ (arising from~$X$ with unstable reduction).
\end{ex}

If~$V$ is a representation of~$G_v$, Fontaine defines
\[
\sD_\pst(V)\colonequals \varinjlim_w\left(\sB_\st\otimes_{\bQ_p}V\right)^{G_w} \,,
\]
where~$\sB_\st$ is Fontaine's ring of semistable periods \cite[\S3.1]{fontaine:corps_des_periodes} and the colimit is taken over open subgroups $G_w\leq G_v$ \cite[\S5.6.4]{fontaine:representations_semi-stables}. This is a $\bQ_p^\nr$-vector space of dimension at most $\dim_{\bQ_p}(V)$ \cite[Th\'eor\`eme~5.6.7(i)]{fontaine:representations_semi-stables}, and carries a natural action of~$G_v$ with open point-stabilisers, namely the restriction of the action on $\sB_\st\otimes_{\bQ_p}V$. Moreover, the natural Frobenius and monodromy operator on~$\sB_\st$ \cite[\S3.2.1--\S3.2.2]{fontaine:corps_des_periodes} induce on~$\sD_\pst(V)$ the structure of a discrete~$(\varphi,N,G_v)$-module, functorially in~$V$.

When~$V$ is de Rham, a theorem of Berger implies that $\dim_{\bQ_p^\nr}\sD_\pst(V) = \dim_{\bQ_p}(V)$ \cite[Th\'eor\`eme~0.7]{berger:representations_p-adiques}, and so the natural $G_v$-equivariant $\Kbar_v$-linear map
\[
c_\BO\colon\Kbar_v\otimes_{\bQ_p^\nr}\sD_\pst(V)\to\Kbar_v\otimes_{K_v}\sD_\deR(V) 
\]
induced from the inclusion\footnote{Usually, one adopts the point of view that the embedding $\sB_\st\hookrightarrow\sB_\deR$ is non-canonical, depending on a choice of $p$-adic logarithm. However, we shall adopt the point of view of \cite{shimizu:p-adic_monodromy}, identifying $\sB_\st$ with its image under this embedding, which is independent of the choice of $p$-adic logarithm.}~$\sB_\st\subset\sB_\deR$ is an isomorphism \cite[Th\'eor\`eme~5.6.7(ii)]{fontaine:representations_semi-stables}. In this way, the tuple
\[
\DpH(V) \colonequals  (\sD_\pst(V),\sD_\deR(V),c_\BO)
\]
is a filtered discrete~$(\varphi,N,G_v)$-module whenever~$V$ is de Rham.

One surprising aspect of Fontaine's theory is that this construction suffices to capture all the intricacies of the category of de Rham representations.

\begin{thm}[{\cite[Th\'eor\`eme~5.6.7(v)]{fontaine:representations_semi-stables}}, {\cite[Th\'eor\`eme~0.7]{berger:representations_p-adiques}}]
The assignment $V\mapsto\DpH(V)$ gives a fully faithful $\otimes$-functor
\[
\DpH\colon \Rep_{\bQ_p}^\deR(G_v)\hookrightarrow\MF(\varphi,N,G_v) \,.
\]
\end{thm}

The following examples of filtered discrete~$(\varphi,N,G_v)$-modules will appear throughout this paper.

\begin{ex}\label{ex:p-adic_hodge_cohomology}
	If~~$X/K_v$ is a variety, we adopt the shorthand
	\[
	\rH^\bullet_\pst(X/\bQ_p^\nr) \colonequals \sD_\pst(\rH^\bullet_\et(X_{\Kbar_v},\bQ_p)) \,,
	\]
	which is a discrete $(\varphi,N,G_v)$-module. The \'etale--de Rham comparison isomorphism gives an isomorphism
	\[
	c_\deR\colon\sD_\deR(\rH^\bullet_\et(X_{\Kbar_v},\bQ_p)) \xrightarrow\sim \rH^\bullet_\deR(X/K_v) \,,
	\]
	and we define the \emph{$p$-adic Hodge cohomology} of~$X$ to be the filtered discrete $(\varphi,N,G_v)$-module
	\[
	\rH^\bullet_\pH(X/K_v) \colonequals (\rH^\bullet_\pst(X/\bQ_p^\nr),\rH^\bullet_\deR(X/K_v),c_\deR\circ c_\BO) \,.
	\]
	This is, of course, isomorphic to
	\[
	\sD_\pH(\rH^\bullet_\et(X_{\Kbar_v},\bQ_p)) = (\sD_\pst(\rH^\bullet_\et(X_{\Kbar_v},\bQ_p)),\sD_\deR(\rH^\bullet_\et(X_{\Kbar_v},\bQ_p)),c_\BO) \,.
	\]
	The assignment~$X\mapsto\rH^\bullet_\pH(X/K_v)$ is contravariant functorial in~$X$ by Proposition~\ref{prop:compatibility_i}\eqref{proppart:compatibility_functoriality} later.
\end{ex}

\begin{rmk}
	We will later have to be careful about exactly which comparison isomorphism~$c_\deR$ we use in the definition of the $p$-adic Hodge cohomology; see Remark~\ref{rmk:which_comparison_iso}. This is why we have to refer to Proposition~\ref{prop:compatibility_i}\eqref{proppart:compatibility_functoriality} for functoriality of~$\rH^\bullet_\pH$, rather than the more usual reference~\cite[Theorem~(A1.1)]{tsuji:c_st_survey}.
\end{rmk}

\subsubsection{Change of base field}

If~$L_w/K_v$ is a finite extension of~$K_v$ contained in~$\Kbar_v$ with absolute Galois group~$G_w$, then one also has the category~$\MF(\varphi,N,G_w)$ of filtered discrete $(\varphi,N,G_w)$-modules, using~$L_w$ as the base field in place of~$K_v$. One can pass between $(\varphi,N,G_v)$- and $(\varphi,N,G_w)$-modules as follows.

If~$D$ is a $(\varphi,N,G_v)$-module then we obtain by restriction a $(\varphi,N,G_w)$-module $D|_{G_w}$. And if~$D=(D_\pst,D_\deR,c_\BO)$ is a filtered discrete $(\varphi,N,G_v)$-module,
 then we define a filtered discrete $(\varphi,N,G_w)$-module~$D|_{G_w}$ by
\[
D|_{G_w} \colonequals (D_\pst|_{G_w},L_w\otimes_{K_v}D_\deR,c_\BO)
\]
where~$c_\BO$ also denotes the composite isomorphism
\[
\Kbar_v\otimes_{\bQ_p^\nr}D_\pst|_{G_w} \xrightarrow[\sim]{c_\BO} \Kbar_v\otimes_{K_v}D_\deR = \Kbar_v\otimes_{L_w}(L_w\otimes_{K_v}D_\deR) \,.
\]
The assignment $D\mapsto D|_{G_w}$ is a $\otimes$-functor in a natural way.

In the other direction, if~$D$ is a $(\varphi,N,G_w)$-module, we make the induction\footnote{We follow the convention that the induction $\Ind_{G_w}^{G_v}\!W$ is the set~$\Map_{G_w}(G_v,W)$ of $G_w$-equivariant maps $\psi\colon G_v\to W$, with $G_v$-action given by $(g\cdot\psi)(h)=\psi(hg)$.} $\Ind_{G_w}^{G_v}\!D$ into a $(\varphi,N,G_v)$-module by giving it the induced Frobenius~$\varphi$ and monodromy~$N$, and making it into a $\bQ_p^\nr$-vector space with the twisted action
\[
(\lambda\cdot\psi)(g) \colonequals g(\lambda)\cdot\psi(g)
\]
for~$\lambda\in\bQ_p^\nr$, $\psi\in\Ind_{G_w}^{G_v}\!D$ and~$g\in G_v$. And if~$D=(D_\pst,D_\deR,c_\BO)$ is a filtered discrete $(\varphi,N,G_w)$-module, then we define
\[
\Ind_{G_w}^{G_v}\!D \colonequals (\Ind_{G_w}^{G_v}\!D_\pst,D_\deR,c_\BO')
\]
where we consider $D_\deR$ as a $K_v$-vector space, and where~$c_\BO'$ is the right-to-left composite in
\[
\Kbar_v\otimes_{\bQ_p^\nr}\Ind_{G_w}^{G_v}\!D_\pst \cong \Ind_{G_w}^{G_v}\!(\Kbar_v\otimes_{\bQ_p^\nr}D_\pst) \xrightarrow[\sim]{\Ind(c_\BO)} \Ind_{G_w}^{G_v}\!(\Kbar_v\otimes_{L_w}D_\deR) \cong \Kbar_v\otimes_{K_v}D_\deR \,.
\]
Explicitly, $c_\BO'$ sends $1\otimes\psi$ to
\[
\sum_i(g_i\otimes1)\left(e\cdot c_\BO(\psi(g_i^{-1}))\right) \in \Kbar_v\otimes_{K_v}D_\deR \,,
\]
where the sum is over left coset representatives~$(g_i)$ for~$G_w\leq G_v$, and $e\in L_w\otimes_{K_v}L_w$ is the idempotent given by $\sum_jx_j\otimes y_j$ where~$(x_j)$ and~$(y_j)$ are dual bases of~$L_w$ with respect to the trace form. The assignment $D\mapsto\Ind_{G_w}^{G_v}\!D$ is a lax $\otimes$-functor in a natural way.

We remark that the induction functor $\Ind_{G_w}^{G_v}\!(-)$ is right adjoint to the restriction functor~$(-)|_{G_w}$ in a natural way, and that the lax $\otimes$-structure on $\Ind_{G_w}^{G_v}\!(-)$ is the one induced from the $\otimes$-structure on~$(-)|_{G_w}$.

\begin{ex}\label{ex:restriction_and_coinduction_of_p-adic_hodge_coh}
	If~$X$ is a variety over~$K_v$, then there is a canonical isomorphism
	\begin{equation}\label{eq:p-adic_hodge_coh_of_base-change}
	\rH^\bullet_\pH(X_{L_w}/L_w) \cong \rH^\bullet_\pH(X/K_v)|_{G_w} \,,
	\end{equation}
	of filtered discrete~$(\varphi,N,G_w)$-modules induced by the isomorphisms
	\[
	\rH^\bullet_\et((X_{L_w})_{\Lbar_w},\bQ_p) = \rH^\bullet_\et(X_{\Kbar_v},\bQ_p) \hspace{0.4cm}\text{and}\hspace{0.4cm} \rH^\bullet_\deR(X_{L_w}/L_w) \cong L_w\otimes_{K_v}\rH^\bullet_\deR(X/K_v) \,.
	\]
	(Here we write~$(-)_{\Lbar_w}=\Spec(\Kbar_v)\times_{\Spec(L_w)}(-)$ to emphasise that the base-change is over~$L_w$.) Similarly, if~$X$ is a variety over~$L_w$, viewed also as a variety over~$K_v$ in the natural way, then there is a canonical isomorphism
	\begin{equation}\label{eq:p-adic_hodge_coh_of_restriction}
	\rH^\bullet_\pH(X/K_v) \cong \Ind_{G_w}^{G_v}\rH^\bullet_\pH(X/L_w) \,,
	\end{equation}
	of filtered discrete~$(\varphi,N,G_v)$-modules induced by the isomorphisms
	\[
	\rH^\bullet_\et(X_{\Kbar_v},\bQ_p) \cong \Ind_{G_w}^{G_v}\rH^\bullet_\et(X_{\Lbar_w},\bQ_p) \hspace{0.4cm}\text{and}\hspace{0.4cm} \rH^\bullet_\deR(X/K_v) = \rH^\bullet_\deR(X/L_w) \,.
	\]
	
	We remark that implicit in the assertion that~\eqref{eq:p-adic_hodge_coh_of_base-change} is an isomorphism of filtered discrete~$(\varphi,N,G_v)$-modules is the fact that the comparison isomorphism~$c_\deR$ is compatible with base-change, see Proposition~\ref{prop:compatibility_i}\eqref{proppart:compatibility_base_change} later. The corresponding compatibility in~\eqref{eq:p-adic_hodge_coh_of_restriction} follows formally from the same property and the restriction--induction adjunction.
\end{ex}

\subsubsection{Automorphisms of $(\varphi,N,G_v)$-modules}\label{sss:phi-N_automorphisms}

In our generalisation of the method of Lawrence--Venkatesh, it will be important at one point to put upper bounds on the dimensions of automorphism groups of $(\varphi,N,G_v)$-modules. More precisely, we will want to bound the dimension of the automorphism group of a \emph{symplectic $(\varphi,N,G_v)$-module}, by which we mean a triple $(D,L,\omega)$ consisting of two $(\varphi,N,G_v)$-modules~$D$ and~$L$, the latter having $\bQ_p^\nr$-dimension~$1$, and a $\bQ_p^\nr$-linear perfect pairing
\[
\omega\colon\bigwedge\nolimits_{\bQ_p^\nr}^{\!2}D \to L
\]
which is equivariant for the $\varphi$-, $N$- and $G_v$-actions.

\begin{defi}\label{def:phi_automorphisms}
\leavevmode
\begin{enumerate}
	\item Let $D$ be a $(\varphi,N,G_v)$-module. An \emph{automorphism} of $D$ is a $\bQ_p^\nr$-linear automorphism of $D$ which commutes with the actions of $\varphi$, $N$ and $G_v$. More generally, if $R$ is a $\bQ_p$-algebra, an \emph{$R$-linear automorphism} of $D$ is a $R\otimes_{\bQ_p}\bQ_p^\nr$-linear automorphism of $R\otimes_{\bQ_p}D$ which commutes with the actions of $1\otimes\varphi$, $1\otimes N$ and the action of~$G_v$ on~$D$. We will see shortly that the functor
\[
\uAut(D)\colon\{\text{$\bQ_p$-algebras}\}\rightarrow \{\text{groups}\}
\]
sending a $\bQ_p$-algebra $R$ to the group of $R$-linear automorphisms of $D$ is representable by an affine algebraic group over $\bQ_p$, which we also call $\uAut(D)$.
	\item If $D=(D,L,\omega)$ is a symplectic $(\varphi,N,G_v)$-module, then we define 
	$\SpAut(D)$ to be the closed algebraic subgroup of $\uAut(D)$ consisting of those automorphisms which preserve the pairing $\omega$ up to a scalar factor of similitude\footnote{Technically speaking, one should regard this factor of similitude as part of the data of a point of $\uAut_\GSp(D)$. This only makes a difference in the degenerate case~$D=0$, which we will never see.}.
\end{enumerate}
By the \emph{scalars} in $\uAut(D)$ (resp.\ $\uAut_\GSp(D)$), we mean those automorphisms which act on~$R\otimes_{\bQ_p}D$ by multiplication by some~$\lambda\in R^\times$. The inclusion of the scalars thus defines a central cocharacter $\bG_{m}\rightarrow\uAut(D)$ (resp.\ $\bG_{m}\rightarrow\uAut_\GSp(D)$) defined over $\bQ_p$.
\end{defi}

\begin{rmk}\label{rmk:induced_autos}
	If~$D=(D_\pst,D_\deR,c_\BO)$ is a filtered discrete $(\varphi,N,G_v)$-module, then the automorphism group $\uAut(D_\pst)$ acts on~$D_\deR$. Given some $\psi\in\uAut(D_\pst)(R)$, the induced $R\otimes_{\bQ_p}\Kbar_v$-linear automorphism of $R\otimes_{\bQ_p}\Kbar_v\otimes_{\bQ_p^\nr}D_\pst\cong R\otimes_{\bQ_p}\Kbar_v\otimes_{K_v}D_\deR$ is $G_v$-equivariant, so induces on taking $G_v$-invariants an $R\otimes_{\bQ_p}K_v$-linear automorphism of $R\otimes_{\bQ_p}D_\deR$, not necessarily preserving the filtration. This construction yields an action
	\[
	\uAut(D_\pst) \to \Res^{K_v}_{\bQ_p}\GL(D_\deR) \,,
	\]
	which one can even show to be a closed embedding (though we don't use this).
\end{rmk}

We will later use the following construction. Fix $D=(D_\pst,D_\deR,c_\BO)$ a filtered discrete $(\varphi,N,G_v)$-module and let~$\dG$ denote the flag variety parametrising $K_v$-linear filtrations on~$D_\deR$ with the same dimension data as the given filtration~$\rF^\bullet$. So there is an action of $\uAut(D_\pst)$ on~$\Res^{K_v}_{\bQ_p}\dG$ induced from the action described in Remark~\ref{rmk:induced_autos}. If~$\Phi\in\dG(K_v)$ is such a filtration on~$D_\deR$, we define a filtered discrete~$(\varphi,N,G_v)$-module $\MM(\Phi)$ by
\[
\MM(\Phi) \colonequals (D_\pst,D_\deR,c_\BO) \,,
\]
where the $(\varphi,N,G_v)$-module structure on~$D_\pst$ is the given one, but the filtration on~$D_\deR$ is given by~$\Phi$ instead of its original filtration.

\begin{lem}
\label{lem:describe orbits}
	In the above setup, let~$\Phi_1,\Phi_2\in\dG(K_v)$ be two filtrations on~$D_\deR$. Then~$\MM(\Phi_1)$ and~$\MM(\Phi_2)$ are isomorphic as filtered discrete $(\varphi,N,G_v)$-modules if and only if~$\Phi_1$ and~$\Phi_2$ lie in the same orbit under the action of~$\uAut(D_\pst)(\bQ_p)$.
	
	In particular, the set of all~$\Phi$ such that~$\MM(\Phi)$ lies in a fixed isomorphism class lies in a $\bQ_p$-subvariety of~$\Res^{K_v}_{\bQ_p}\dG$ of dimension at most~$\dim_{\bQ_p}\uAut(D_\pst)$.
	\begin{proof}
		An isomorphism $\MM(\Phi_1)\xrightarrow\sim\MM(\Phi_2)$ is determined by its $\pst$ component, which must be a $(\varphi,N,G_v)$-module automorphism~$\psi$ of~$D_\pst$. The condition that~$\psi$ is a filtered isomorphism is exactly that~$\psi$ takes~$\Phi_1$ to~$\Phi_2$ under the action described in Remark~\ref{rmk:induced_autos}, so we are done.
	\end{proof}
\end{lem}

Now let us give the promised proof of representability of $\uAut(D)$ and $\uAut_\GSp(D)$, and bound their dimensions.

\begin{prop}\label{prop:dim_of_aut}
\leavevmode
\begin{enumerate}
	\item\label{proppart:dim_of_aut_basic} Let $D$ be a $(\varphi,N,G_v)$-module of $\bQ_p^\nr$-dimension $d$. Then $\uAut(D)$ is a $\bQ_p$-algebraic group of $\bQ_p$-dimension $\leq d^2$.
	\item\label{proppart:dim_of_aut_symplectic} Let $D=(D,L,\omega)$ be a symplectic $(\varphi,N,G_v)$-module of $\bQ_p^\nr$-rank $2d$. Then $\SpAut(D)$ is a $\bQ_p$-algebraic group of $\bQ_p$-dimension $\leq d(2d+1)+1$.
\end{enumerate}
\begin{proof}
	\eqref{proppart:dim_of_aut_basic}. Let~$\End(D)$ denote the (possibly) non-commutative $\bQ_p$-algebra of endomorphisms of~$D$ as a $(\varphi,N,G_v)$-module, so that $\uAut(D)$ is isomorphic to the functor
	\[
	R\mapsto (R\otimes_{\bQ_p}\End(D))^\times \,.
	\]
	To show that $\uAut(D)$ is a $\bQ_p$-algebraic group, it suffices to show that $\dim_{\bQ_p}\End(D)<\infty$, for then $\uAut(D)$ is its group-scheme of units (which is a closed algebraic subvariety of the affine space corresponding to $\End(D) \times \End(D)$). For this, we claim that the map
	\begin{equation}\label{eq:endomorphism_embedding}\tag{$\ast$}
		\bQ_p^\nr\otimes_{\bQ_p}\End(D) \to \End_{\bQ_p^\nr}(D)
	\end{equation}
	is injective, where the right-hand side denotes the $\bQ_p^\nr$-linear endomorphisms of~$D$ (requiring no compatibility with the $(\varphi,N,G_v)$-action). So suppose that $\psi_1,\dots,\psi_k$ are $\bQ_p$-linearly independent elements of~$\End(D)$ and that $\lambda_1,\dots,\lambda_k\in\bQ_p^\nr$ are such that~$\sum_i\lambda_i\otimes\psi_i$ lies in the kernel of~\eqref{eq:endomorphism_embedding}. This says that $\sum_i\lambda_i\psi_i=0$ as a $\bQ_p^\nr$-linear endomorphism of~$D$. Since each~$\psi_i$ commutes with the action of~$G_v$ on~$D$, we thus have
	\[
	\sum_i\sigma(\alpha\lambda_i)\psi_i=0
	\]
	for all~$\alpha\in\bQ_p^\nr$ and all~$\sigma\in G_v$. Taking a suitable linear combination of this identity then shows that
	\[
	\sum_i\tr_{\bQ_p^\nr/\bQ_p}(\alpha\lambda_i)\psi_i=0
	\]
	for all~$\alpha\in\bQ_p^\nr$, where $\tr_{\bQ_p^\nr/\bQ_p}$ denotes the normalised trace (so that it is the identity on~$\bQ_p$). By $\bQ_p$-linear independence of the $\psi_i$ and non-degeneracy of the trace pairing, this implies that $\lambda_i=0$ for all~$i$. So $\sum_i\lambda_i\otimes\psi_i=0$ and hence~\eqref{eq:endomorphism_embedding} is injective.
	
	Thus we have shown that $\End(D)$ is finite-dimensional, and so its group-scheme of units $\uAut(D)$ is a $\bQ_p$-algebraic group. To bound its dimension, we observe that~\eqref{eq:endomorphism_embedding}, being an injective morphism of non-commutative $\bQ_p^\nr$-algebras induces a closed immersion
	\begin{equation}\label{eq:automorphism_embedding}\tag{$\ast\ast$}
		\uAut(D)_{\bQ_p^\nr}\hookrightarrow\GL_{\bQ_p^\nr}(D)
	\end{equation}
	on the $\bQ_p^\nr$-group-schemes of units. Since the right-hand side has $\bQ_p^\nr$-dimension~$d^2$, we find that $\dim_{\bQ_p}\uAut(D)\leq d^2$ as desired.
	
	\eqref{proppart:dim_of_aut_symplectic}. This follows from the first part. It is easy to check that preserving~$\omega$ up to a scalar factor of similitude is a closed condition on~$\uAut(D)$, so $\SpAut(D)$ is a closed $\bQ_p$-subgroup-scheme of~$\uAut(D)$. Moreover, the embedding~\eqref{eq:automorphism_embedding} takes $\SpAut(D)_{\bQ_p^\nr}$ into $\GSp_{\bQ_p^\nr}(D)$, whence
	\[
	\dim_{\bQ_p}(\SpAut(D)) \leq \dim_{\bQ_p^\nr}(\GSp_{\bQ_p^\nr}(D)) \leq d(2d+1)+1 \,. \qedhere
	\]
\end{proof}
\end{prop}

\subsection{Self-conjugate places and average Hodge--Tate weights}\label{ss:friendly}

In order to make certain numerics work in a key step of our generalisation of the Lawrence--Venkatesh method, it will be necessary for us to restrict attention to places~$v$ of~$K$ which satisfy a particular technical assumption, which we call \emph{self-conjugacy}. This is a mild generalisation of the notion of a \emph{friendly place} from  \cite[Definition~2.7]{LVinventiones}.

\begin{defi}[Self-conjugate places, cf.\ {\cite[Definition~2.7]{LVinventiones}}]\label{def:friendly}
We say that two finite places~$v$ and~$v'$ 
of~$K$ are \emph{conjugate} just when their restrictions $v_E$ and $v'_E$ to every CM subfield~$E$ of~$K$ are conjugate under the complex conjugation on~$E$, i.e. $v'_E = \ov{v_E}$. We say that a finite place  is \emph{self-conjugate} if it is conjugate to itself.
\end{defi}

\begin{rmk}\leavevmode
\begin{enumerate}
	\item If~$K$ contains no CM subfield (e.g.\ $K$ totally real or $K$ of odd degree), then every finite place is self-conjugate.
	\item If~$K$ contains a CM subfield, then it has a maximal CM subfield~$E$, and a finite place~$v$ of~$K$ is self-conjugate if and only if its restriction to~$E$ is invariant under the conjugation on~$E$.
	\item The \emph{friendly places} of \cite[Definition~2.7]{LVinventiones} are exactly the self-conjugate places which are unramified over~$\bQ$.
\end{enumerate}
\end{rmk}

The property of self-conjugate places~$v$ which is relevant in the Lawrence--Venkatesh argument is that if~$V$ is a representation of~$G_K$ which is pure of some weight, then the average Hodge--Tate weight of~$V|_{G_v}$ is determined by the weight of~$V$. To prove this, one first deals with the case of $1$-dimensional representations.

\begin{prop}[cf.\ {\cite[Lemma~2.8]{LVinventiones}}]\label{prop:friendly_characters}
	Let $\chi\colon G_K\rightarrow\bQ_p^\times$ be a $p$-adic character such that:
	\begin{enumerate}[label=(\roman*)]
		\item $\chi$ is unramified and pure of weight~$n\in\bZ$ outside a finite set~$S$ of places of~$K$, and
		\item $\chi$ is de Rham at all $p$-adic places of~$K$.
	\end{enumerate}
	For a $p$-adic place~$v$, write~$r_v$ for the Hodge--Tate weight\footnote{The \emph{Hodge--Tate weights} of a representation $V$ of $G_v$ are the integers $r$ such that $(\bC_p(r)\otimes_{\bQ_p}V)^{G_v}\neq0$, where the $G_v$ action on $\bC_p(r)$ is the $r$th Tate twist of the natural action on $\bC_p$.} of~$\chi|_{G_v}$. Then if~$v$ and~$v'$ are conjugate $p$-adic places of~$K$, then $r_v+r_{v'}=n$.
	
	In particular, if~$v$ is self-conjugate, then~$n$ is even and $r_v=n/2$.
	\begin{proof}
		We follow the proof of \cite[Lemma~2.8]{LVinventiones}. Let~$\eta\colon\bA_K^\times\to\bQ_p^\times$ be the idele class character corresponding to~$\chi$ by global class field theory\footnote{There are two opposite conventions for normalising the isomorphisms of class field theory, depending on whether the local Artin maps $K_u^\times\to G_u^\ab$ send uniformisers to arithmetic or geometric Frobenii. We adopt the convention of~\cite{milne:cft}, that a uniformiser corresponds to an arithmetic Frobenius.}. We write~$\eta_u\colon K_u^\times\to\bQ_p^\times$ for the component of~$\eta$ at a place~$u$ of~$K$. The conditions on~$\chi$ translate into the following conditions on the~$\eta_u$:
		\begin{enumerate}[label=\alph*),ref=(\alph*)]
			\item for all places~$u$ outside a finite set~$S$ (assumed to contain all places dividing~$p\infty$) $\eta_u(\cO_u^\times)=\{1\}$, and~$\eta_u(\varpi_u)$ is a $q_u$-Weil number of weight~$-n$, where~$\varpi_u$ is a uniformiser of~$K_u$ and~$q_u$ is the order of the residue field of~$K_u$; and
			\item for all $p$-adic places~$v$, the component $\eta_v$ agrees with the $r_v$th power of the norm character $\rN_{K_v|\bQ_p}$ on an open subgroup of~$\cO_v^\times$.
		\end{enumerate}
		Of these, the second deserves some explanation. Being Hodge--Tate of weight~$r_v$, the restricted character~$\chi|_{G_v}$ agrees with the~$-r_v$th power of the cyclotomic character~$\chi_\cyc$ on an open subgroup of the inertia group~$I_v$ \cite[Theorem~III.A5.2]{serre:reps_and_elliptic_curves}. It follows from Lubin--Tate theory that the composite of~$\chi_\cyc$ with the local Artin map $\cO_v^\times\to I_v^\ab$ is equal to the \emph{inverse} of the norm map $\cO_v^\times\to\bZ_p^\times$; for instance this holds when~$K_v=\bQ_p$ by the explicit description on \cite[p40]{milne:cft}, and then holds in general by norm-compatibility of local Artin maps. So~$\eta_v$ agrees with the $r_v$th power of the norm character on an open subgroup of~$\cO_v^\times$.
		
		Now let us write
		\[
		\eta_p^\alg\colon(K\otimes_\bQ\bQ_p)^\times=\prod_{v\mid p}K_v^\times\to\bQ_p^\times
		\]
		for the product of the maps $\rN_{K_v|\bQ_p}^{r_v}\colon K_v^\times\to\bQ_p^\times$. So, by construction, $\eta_p^\alg$ agrees with the product of the local characters~$\eta_v$ for $v\mid p$ on an open neighbourhood of~$1\in(K\otimes\bQ_p)^\times$. From the first condition above and the fact that~$\eta$ vanishes on~$K^\times\subset\bA_K^\times$, we deduce the following regarding~$\eta_p^\alg$:
		\begin{enumerate}[label=\roman*),ref=(\roman*)]
			\item $\eta_p^\alg(\alpha)=1$ for all~$\alpha$ in a finite-index subgroup of~$\cO_K^\times\subset(K\otimes\bQ_p)^\times$; and
			\item\label{condn:weight_for_eta} for~$\alpha\in\cO_K$, $\alpha \not= 0$, not divisible by any prime in~$S$, we have that~$\eta_p^\alg(\alpha)$ is algebraic over~$\bQ$, and $|\iota(\eta_p^\alg(\alpha))|=|\rN_{K|\bQ}(\alpha)|^{n/2}$ for every complex embedding~$\iota\colon\bQ_p\hookrightarrow\bC$.
		\end{enumerate}
		
		Now the character~$\eta_p^\alg$ above is algebraic, meaning that it is the map on~$\bQ_p$-points of a homomorphism
		\[
		\eta_p^\alg\colon\left(\Res^K_{\bQ}\bG_{m,K}\right)_{\bQ_p}=\prod_{v\mid p}\Res^{K_v}_{\bQ_p}\bG_{m,K_v}\to\bG_{m,\bQ_p}
		\]
		of tori over~$\bQ_p$. The fact that~$\eta_p^\alg$ vanishes on an open subgroup of~$\cO_K^\times$ ensures that it factors over~$\bS_{K,\bQ_p}$ where~$\bS_K$ is the Serre torus: the quotient of $\Res^K_\bQ\bG_{m,K}$ by the identity component of the Zariski-closure\footnote{Here we are implicitly using that Zariski-closures are stable under base extension. That is, suppose that~$F'/F$ is an extension of fields, $X$ is a variety over~$F$ and $X_0\subseteq X(F)$ is a subset of the~$F$-rational points of~$X$. Write~$Z\subseteq X$ and $Z'\subseteq X_{F'}$ for the Zariski-closures of~$X_0$ in~$X$ and~$X_{F'}$, respectively. Then~$Z'=Z_{F'}$.} of~$\cO_K^\times$.
		
		We conclude using results of Serre on the structure of~$\bS_K$. Suppose firstly that~$K$ has no CM subfield. Then the norm map~$\rN_{K|\bQ}\colon\bS_K\to\bS_\bQ=\bG_{m,\bQ}$ is an isogeny \cite[II-34]{serre:reps_and_elliptic_curves}. So, replacing~$\chi$ by a power if necessary we may assume that~$\eta_p^\alg=\rN_{K|\bQ}^m$ is a power of the global norm character. Now on the one hand, the fact that~$\eta_p^\alg$ is the product of the local norm characters means that $r_v=m$ for all $p$-adic places~$v$. On the other hand, condition~\ref{condn:weight_for_eta} above implies that~$m=n/2$. Hence we have~$r_v=n/2$ for all $p$-adic places~$v$ and we are done.
		
		Suppose instead that~$K$ has a CM subfield, and write~$E$ for the maximal CM subfield of~$K$ and $E^+$ for the maximal totally real subfield of~$E$. Then the norm map~$\rN_{K|E}\colon\bS_K\to\bS_E$ is an isogeny, as is the norm map~$\rN_{E^+|\bQ}\colon\bS_{E^+}\to\bS_\bQ=\bG_{m,\bQ}$ by \cite[II-34]{serre:reps_and_elliptic_curves}. So, replacing~$\chi$ by a power if necessary, we may assume that~$\eta_p^\alg$ factors as $\eta^\alg_0\circ\rN_{K|E}$ for some character~$\eta_0^\alg\colon\bS_{E,\bQ_p}\to\bG_{m,\bQ_p}$, and that $\eta_0^\alg\bar\eta_0^\alg=\rN_{E|\bQ}^m$ is a power of the norm character, where~$\bar\eta_0^\alg$ denotes the composite of~$\eta_0^\alg$ and the conjugation on~$E$. On the one hand, the definition of~$\eta_p^\alg$ ensures that~$r_v+r_{v'}=m$ whenever~$v$ and~$v'$ are conjugate $p$-adic places of~$K$. On the other hand, condition~\ref{condn:weight_for_eta} implies that~$m=n$, so $r_v+r_{v'}=n$ and we are done.
	\end{proof}
\end{prop}

\begin{rmk}
	The conditions of Proposition~\ref{prop:friendly_characters} impose very strong restrictions on the character~$\chi$: if~$v$ is self-conjugate, then we have $\chi|_{G_v}=\chi_0\cdot\chi_\cyc^{-n/2}$ where~$\chi_0$ is a finite-order character of~$G_v$ (cf.\ the statement of \cite[Lemma~2.8]{LVinventiones}). Although we will not need this finer statement in what follows, we indicate how to deduce this from Proposition~\ref{prop:friendly_characters} as stated.
	
	Replacing~$\chi$ by a twist by a power of the cyclotomic character~$\chi_\cyc$, it suffices to prove this in the case~$n=0$. If~$K$ contains no CM subfield, then we have~$r_v=0$ for all~$v\mid p$, so that the idele class character~$\eta$ is only finitely ramified at all places of~$K$, including $p$-adic places. It follows that~$\eta$ factors through a ray class group of~$K$, so~$\eta$ is a finite-order character. Thus~$\chi$ itself is a finite-order character.
	
	If instead~$K$ contains a CM subfield, we write~$E$ for its maximal CM subfield and~$\eta_0$ for the restriction of~$\eta$ to~$\bA_E^\times\subseteq\bA_K^\times$. The local component of~$\eta_0$ at a $p$-adic place~$v_0$ of~$E$ agrees, in a neighbourhood of~$1\in E_{v_0}^\times$, with the $r_{v_0}$th power of the norm map~$\rN_{E_{v_0}|\bQ_p}$, where~$r_{v_0}\colonequals \sum_{v\mid v_0}[K_v:E_{v_0}]r_v$. If~$v_0'$ is the place of~$E$ conjugate to~$v_0$, then we have $r_{v_0}+r_{v_0'}=0$ courtesy of Proposition~\ref{prop:friendly_characters}, and hence~$\eta_0\bar\eta_0$ is finitely ramified at every place of~$E$, where~$\bar\eta_0$ is the composite of~$\eta_0$ with the conjugation on~$E$. As before, this implies that~$\eta_0\bar\eta_0$ is a finite-order character of~$\bA_E^\times$. In particular, the restriction of~$\eta_v^2$ to~$E_{v_0}^\times$ is a finite-order character. Since~$\eta_v$ is finitely ramified, this implies that~$\eta_v$ is a finite-order character, and so is~$\chi|_{G_v}$.
\end{rmk}

\begin{cor}[cf.\ {\cite[Lemmas~2.9~and~2.10]{LVinventiones}}]\label{cor:friendly_representations}
Let $L/K$ be a finite extension and~$V$ a representation of~$G_L$ which is de Rham at all places of~$L$ above~$p$ and unramified and pure of weight $n\in\bZ$ outside a finite set of places of~$L$. Let~$v$ be a self-conjugate place of~$K$ above~$p$. Then
\[
\sum_{w\mid v}[L_w:K_v]\left(\sum_{i\in\bZ}i\dim_{L_w}\gr_{\rF}^i\sD_{\deR,w}(V)\right) = \frac{n\cdot[L:K]\cdot\dim_{\bQ_p}(V)}{2}\,,
\]
where the summation runs through places~$w$ of~$L$ dividing~$v$, and $\sD_{\deR,w}(V) \colonequals  \sD_{\deR}(V|_{G_w})$ as filtered $L_w$-vector spaces.
\end{cor}

\begin{proof}
We fix an identification $\Lbar \cong \Kbar$ to view~$G_L$ as a subgroup of~$G_K$.
The induced representation $W =  \Ind_{G_L}^{G_K}(V)$ is again unramified outside a finite set of places of $K$, and it is pure of the same weight $n$ (weights can be read off after restriction to open subgroups of $G_K$, and for a suitable such the induced representation will decompose as a direct sum of conjugates of the restriction of $V$). 

For a place $w \mid v$ above $p$ we also fix an identification $\Lbar_w \cong \Kbar_v$ to view~$G_w$ as a subgroup of~$G_v$. The local induced representation $W_w = \Ind_{G_w}^{G_v}(V)$ has
\[
\sD_{\deR,w}(V) = (\sB_\deR(\Lbar_w)\otimes_{\bQ_p}V)^{G_w} \cong (\sB_\deR(\Kbar_v)\otimes_{\bQ_p} W_w)^{G_v} = \sD_{\deR,v}(W_w)
\]
as filtered $K_v$-vector spaces. Since $W \simeq \bigoplus_{w \mid v} W_w$ as $G_v$-representations, the representation $W$ is also de Rham at all places above $p$. The Hodge--Tate weight $i$ occurs with multiplicity 
\[
\dim_{K_v} \Big(\gr_{\rF}^i\sD_{\deR,v}(W)\Big) 
= \sum_{w \mid v} [L_w:K_v] \cdot \dim_{L_w} \Big(\gr_{\rF}^i\sD_{\deR,w}(V)\Big).
\]

We now apply Proposition~\ref{prop:friendly_characters} to the character $\chi = \det(W)$. The character $\chi$ is unramified and pure of weight $n\cdot[L:K]\cdot\dim_{\bQ_p}(V)$ outside a finite set of places of $K$. It is also de Rham at places above $p$ and the Hodge-Tate weight at $v$ is 
\[
\sum_{i \in \bZ}  i\dim_{K_v}  \Big(\gr_{\rF}^i\sD_{\deR,v}(W)\Big) = \sum_{w\mid v}[L_w:K_v]\left(\sum_{i\in\bZ}i\dim_{L_w}  \Big(\gr_{\rF}^i\sD_{\deR,w}(V)\Big)\right).  \qedhere
\]
\end{proof}
\section{Preliminaries: families of Galois representations and $v$-adic period maps}\label{s:reps_in_families}

The method of Lawrence--Venkatesh revolves around the study of families of Galois representations arising from smooth proper morphisms $\pi\colon X\to Y$ of smooth
varieties. In this section, we recall how the local representations associated to local points $y\in Y(K_v)$ can be controlled using the theory of $v$-adic period maps. Our presentation here differs significantly from that in \cite{LVinventiones}, in that we do not make any good reduction assumptions on the family $X\to Y$ and treat the period map from a purely analytic perspective, without reference to any model. This is most evident in the proof we give that period maps control the variation of local Galois representations (Theorem~\ref{thm:period_maps_control_reps}), where we must forgo crystalline cohomology as a bridge between \'etale and de Rham cohomology, and must instead use tools from relative $p$-adic Hodge theory as developed by Scholze~\cite{scholze:relative_p-adic_hodge_theory} and Shimizu~\cite{shimizu:p-adic_monodromy}.

\subsection{Period maps}\label{ss:period_maps}

Suppose that~$Y$ is a smooth $K_v$-variety, and let~$(\dE,\nabla)$ be a vector bundle\footnote{A vector bundle on a rigid-analytic space~$S$ over~$K_v$ is an~$\cO_S$-module which is locally finite free. Here, the implicit topology can equivalently be taken to be the analytic, \'etale or pro-\'etale topology \cite[Lemma~7.3]{scholze:relative_p-adic_hodge_theory}.} with flat connection on the rigid analytification~$Y^\an$ (e.g.\ the analytification of an algebraic vector bundle on~$Y$ with flat connection). Suppose that~$\Nbd_{y_0}\subseteq Y^\an$ is an admissible open neighbourhood of a point $y_0\in Y(K_v)$, such that $\rH^0_\deR(\Nbd_{y_0}/K_v)=K_v$; for instance, $\Nbd_{y_0}$ could be isomorphic to a closed polydisc or a closed polyannulus. We say that~$(\dE,\nabla)$ has a \emph{full basis of horizontal sections} over~$\Nbd_{y_0}$ just when $(\dE,\nabla)|_{\Nbd_{y_0}}\simeq(\cO_{\Nbd_{y_0}},\rd)^{\oplus m}$ for some~$m$. The vector bundle $(\dE,\nabla)$ always admits a full basis of horizontal sections over a sufficiently small neighbourhood of~$y_0$ \cite[Theorem~9.7]{shimizu:p-adic_monodromy}.

When~$(\dE,\nabla)$ has a full basis of horizontal sections over~$\Nbd_{y_0}$, then there is a \emph{canonical} isomorphism
\[
T^\nabla_{y_0}\colon(\cO_{\Nbd_{y_0}}\otimes_{K_v}\dE_{y_0},\rd\otimes1)\xrightarrow\sim(\dE,\nabla)|_{\Nbd_{y_0}} \,,
\]
characterised by the fact that the fibre of~$T^\nabla_{y_0}$ at~$y_0$ is the identity map. The fibre of $T^\nabla_{y_0}$ at another point $y\in\Nbd_{y_0}(K_v)$ is an isomorphism $T^\nabla_{y_0,y}\colon \dE_{y_0}\xrightarrow\sim \dE_y$ from the fibre over~$y_0$ to the fibre over~$y$, known as the \emph{parallel transport} map.

Using parallel transport, one can define the period map associated to a filtered vector bundle with flat connection.

\begin{defi}\label{def:period_maps}
Let~$Y$ be a smooth variety over~$K_v$, and let $(\dE,\nabla)$ be a vector bundle with flat connection on~$Y^\an$ endowed with an exhaustive, separated descending filtration 
\[
\dots\geq\rF^{-1}\!\dE\geq\rF^0\!\dE\geq\rF^{1}\!\dE\geq\dots
\]
(not necessarily stable\footnote{In practice, the filtrations appearing will all satisfy Griffiths transversality with respect to the connection, but this is not necessary to define the period map.} under the connection~$\nabla$) whose graded pieces are all vector bundles.

For $y_0\in Y(K_v)$, we define the \emph{period domain} to be the flag variety~$\dG_{y_0}$ parametrising filtrations on~$\dE_{y_0}$ with the same dimension data as~$\rF^\bullet\!\dE_{y_0}$. If~$\Nbd_{y_0}\subseteq Y^\an$ is an admissible open neighbourhood of~$y_0$ such that~$\rH^0_\deR(\Nbd_{y_0}/K_v)=K_v$ and~$(\dE,\nabla)$ has a full basis of horizontal sections over~$\Nbd_{y_0}$, then we define the ($v$-adic) \emph{period map} to be the $K_v$-analytic map
\[
\Phi_{y_0}\colon\Nbd_{y_0}\to\dG_{y_0}^\an
\]
classifying the filtration on $\cO_{\Nbd_{y_0}}\otimes_{K_v}\dE_{y_0}$ given by pulling back the filtration~$\rF^\bullet\!\dE$ along the parallel transport map~$T_{y_0}^\nabla$. Concretely, if $y\in\Nbd_{y_0}(K_v)$, then $\Phi_{y_0}(y)$ is the $K_v$-point of $\dG_{y_0}$ corresponding to the filtration on~$\dE_{y_0}$ given by pulling back the filtration on~$\dE_y$ along the parallel transport map $T^\nabla_{y_0,y}$.
\end{defi}

Implicitly in this definition, we have used the following description of the analytification of the flag variety~$\dG_{y_0}$.

\begin{prop}\label{prop:analytic_flag_variety}
Let~$V$ be a finite-dimensional~$K_v$-vector space, and let $(d_i)_{i\in\bZ}$ be non-negative integers summing to~$\dim_{K_v}V$. Let~$\dG/K_v$ denote the flag variety parametrising filtrations on~$V$ with dimension data $(d_i)_{i\in\bZ}$. Then the rigid analytification $\dG^\an$ represents the functor from $K_v$-analytic spaces to sets
\[
S\mapsto \left\{
\begin{array}{c}
\text{exhaustive, separated filtrations $\rF^\bullet$ on $\cO_S\otimes_{K_v}V$ such that } \\
\text{each $\gr_{\rF}^i(\cO_S\otimes_{K_v}V)$ is a vector bundle of rank $d_i$}
\end{array}
\right\}.
\]
\begin{proof}
	We may assume~$V=K_v^{\oplus m}$. By a \emph{framing} of a filtration~$\rF^\bullet$ on~$\cO_S^{\oplus m}$ ($S$ a $K_v$-scheme) we mean a frame~$f_1,\dots,f_m$ of~$\cO_S^{\oplus m}$ which is adapted to~$\rF^\bullet$, meaning that~$\rF^i\!\cO_S^{\oplus m}$ is the span of~$f_1,\dots,f_{\sum_{j\geq i}d_j}$ for all~$i$. Since the filtration~$\rF^\bullet$ is uniquely determined by the framing, we see that the functor
	\[
	S\mapsto\{\text{framed filtrations~$\rF^\bullet$ on~$\cO_S^{\oplus m}$}\}
	\]
	is represented by~$\GL_m$. Moreover, two framed filtrations have the same underlying filtration if and only if their corresponding elements of~$\GL_m(S)$ differ by the right action of an $S$-point of the parabolic subgroup~$\rP\leq\GL_m$ consisting of the block upper-triangular matrices with block sizes~$d_i$. It follows from this description that$~\GL_m$ is a~$\rP$-torsor over~$\dG$, locally trivial in the Zariski topology.
	
	Now since analytification commutes with products~\cite[Example~4.3.3(4)]{fresnel-van_der_put:rigid_analytic_geometry}, it follows that~$\GL_m^\an$ is a group in the category of rigid spaces, $\rP^\an$ is a subgroup, and that the right-multiplication action makes~$\GL_m^\an$ into a~$\rP^\an$-torsor over~$\dG^\an$, locally trivial in the analytic topology. In particular, we see that~$\dG^\an$ represents the sheafification in the analytic topology of the functor
	\begin{equation}\label{eq:flag_presheaf}\tag{$\ast$}
	S \mapsto \GL^\an_m(S)/\rP^\an(S) \,.
	\end{equation}
	By definition of analytification, we see as in the algebraic setting that~$\GL_m^\an(S)$ is canonically in bijection with the set of framed filtrations on~$\cO_S^{\oplus m}$ (defined in the obvious way), and that two points differ by the action of an element of~$\rP^\an(S)$ if and only if they have the same underlying filtration. So~\eqref{eq:flag_presheaf} is the functor taking~$S$ to the set of framable filtrations on~$\cO_S^{\oplus m}$, which sheafifies to the claimed functor.
\end{proof}
\end{prop}

The importance of period maps is that they give a concrete description of how Galois representations vary in families. This is true in a great level of generality, but for our purposes, it suffices to consider only those families of Galois representations arising from smooth proper families of varieties. Consider a smooth proper morphism~$\pi\colon X\to Y$ of smooth $K_v$-varieties, and let~$\cH^i_\deR(X/Y)$ denote the $i$th relative de Rham cohomology of~$X$ over~$Y$ \cite[\S2]{katz-oda:relative_de_rham}. This is a vector bundle on~$Y$ endowed with a Hodge filtration~$\rF^\bullet$ and a flat Gau\ss--Manin connection~$\nabla$ \cite[Theorem~1]{katz-oda:relative_de_rham}. For any point~$y\in Y(K_v)$, the fibre of~$\cH^i_\deR(X/Y)^\an$ at~$y$ is canonically identified with the (algebraic) de Rham cohomology~$\rH^i_\deR(X_y/K_v)$ of the fibre~$X_y$, via the base change theorem for de Rham cohomology. If~$\Nbd\subseteq Y^\an$ is an admissible open subset, isomorphic to a closed polydisc, over which the analytification of~$\cH^i_\deR(X/Y)$ has a full basis of horizontal sections, and $y_0,y\in U(K_v)$, then by mild abuse of notation we will denote the composite
\[
\rH^i_\deR(X_{y_0}/K_v) \cong \cH^i_\deR(X/Y)^\an_{y_0} \xrightarrow[\sim]{T_{y_0,y}^\nabla} \cH^i_\deR(X/Y)^\an_y \cong \rH^i_\deR(X_y/K_v)
\]
also by~$T_{y_0,y}^\nabla$, the unlabelled isomorphisms being the base-change isomorphisms for de Rham cohomology.

\begin{thm}\label{thm:period_maps_control_reps}
Let $\pi\colon X\to Y$ be a smooth proper morphism of smooth $K_v$-varieties, and let $\Nbd\subseteq Y^\an$ be an admissible open subset, isomorphic to a closed polydisc, over which $\cH^i_\deR(X/Y)^\an$ has a full basis of horizontal sections for the Gau\ss--Manin connection. Then for every $y_0,y\in\Nbd(K_v)$, there is a unique isomorphism
\[
T_{y_0,y}\colon\sD_\pst\left(\rH^i_\et(X_{y_0,\Kbar_v},\bQ_p)\right) \xrightarrow\sim \sD_\pst\left(\rH^i_\et(X_{y,\Kbar_v},\bQ_p)\right)
\]
of $(\varphi,N,G_v)$-modules making the following diagram commute:
\begin{center}
\begin{tikzcd}
	\Kbar_v\otimes_{\bQ_p^\nr}\sD_\pst\left(\rH^i_\et(X_{y_0,\Kbar_v},\bQ_p)\right) \arrow[r,"c_\BO","\sim"']\arrow[d,"1\otimes T_{y_0,y}"',"\wr"] & \Kbar_v\otimes_{K_v}\sD_\deR\left(\rH^i_\et(X_{y_0,\Kbar_v},\bQ_p)\right) \arrow[r,"c_\deR","\sim"'] & \Kbar_v\otimes_{K_v}\rH^i_\deR(X_{y_0}/K_v) \arrow[d,"1\otimes T_{y_0,y}^\nabla","\wr"'] \\
	\Kbar_v\otimes_{\bQ_p^\nr}\sD_\pst\left(\rH^i_\et(X_{y,\Kbar_v},\bQ_p)\right) \arrow[r,"c_\BO","\sim"'] & \Kbar_v\otimes_{K_v}\sD_\deR\left(\rH^i_\et(X_{y,\Kbar_v},\bQ_p)\right) \arrow[r,"c_\deR","\sim"'] & \Kbar_v\otimes_{K_v}\rH^i_\deR(X_y/K_v) \,.
\end{tikzcd}
\end{center}
\begin{proof}
	This will come in~\S\ref{ss:shimizu}.
\end{proof}
\end{thm}

\begin{rmk}\label{rmk:which_comparison_iso}
	There are several different definitions of the \'etale--de Rham comparison isomorphism~$c_\deR$ in the literature (e.g.\ \cite[Theorem~8.1]{faltings:crystalline_cohomology}, \cite[Theorem~A1]{tsuji:c_st_survey}), and the validity of Theorem~\ref{thm:period_maps_control_reps} depends, \emph{a priori}, on which definition is used. For the sake of clarity, the comparison isomorphism~$c_\deR$ for which we prove Theorem~\ref{thm:period_maps_control_reps} is the comparison isomorphism constructed by Scholze \cite[Corollary~1.8]{scholze:relative_p-adic_hodge_theory} in the setting of rigid analysis. Throughout this paper, all references to~$c_\deR$ should be taken to refer to this particular comparison isomorphism, whose construction we will spell out in~\S\ref{ss:c_dR}.
	
	Later on, we will need to know that the comparison isomorphism~$c_\deR$ is compatible with all of the natural constructions in cohomology: pullbacks (functoriality), cup products, K\"unneth decompositions, trace maps, Poincar\'e duality, pushforwards, cycle class maps and Chern classes. Although these compatibilities are known for some definitions of~$c_\deR$, e.g.\ \cite[Theorem~A1]{tsuji:c_st_survey}, it will be important that these hold for the above-mentioned~$c_\deR$, and so we will be forced to re-prove these compatibilities ourselves.
	
	Let us also remark that Nizio\l\ has shown in~\cite{niziol:too_many_comparison_isos} that many definitions of the comparison isomorphism~$c_\deR$ are equivalent to one another, but Scholze's definition is not among those covered by her work.
\end{rmk}

From now on, we adopt the notation as in Example~\ref{ex:p-adic_hodge_cohomology}. That is, if~$Z$ is a variety defined over a finite extension~$L_w$ of~$K_v$ contained in~$\Kbar_v$ we adopt the shorthand $\rH^\bullet_\pst(Z/\bQ_p^\nr) =\sD_\pst(\rH^\bullet_\et(Z_{\Lbar_w},\bQ_p))$, and have the $p$-adic Hodge cohomology
\[
\rH^\bullet_\pH(Z/L_w) = (\rH^\bullet_\pst(Z/\bQ_p^\nr),\rH^\bullet_\deR(Z/L_w),c_\deR\circ c_\BO) \,,
\]
which is a filtered discrete~$(\varphi,N,G_w)$-module. So $\rH^\bullet_\pH(Z/L_w)\cong\sD_{\pH,w}(\rH^\bullet_\et(Z_{\Kbar_v},\bQ_p))$.

When~$\pi\colon X\to Y$ is a smooth proper morphism of smooth $K_v$-varieties, Theorem~\ref{thm:period_maps_control_reps} does \emph{not} say that the isomorphism class of the filtered discrete~$(\varphi,N,G_v)$-module $\rH^i_\pH(X_y/K_v)$ is constant on the neighbourhood~$\Nbd_{y_0}$, because the parallel transport map $T_{y_0,y}^\nabla$ does not preserve Hodge filtrations in general. Instead, it says that the isomorphism class of $\rH^i_\pst(X_y/\bQ_p^\nr)$ is constant on~$\Nbd_{y_0}$ as a discrete~$(\varphi,N,G_v)$-module, and that the variation of the Hodge filtration on $\rH^i_\pH(X_y/K_v)$ is controlled by the $v$-adic period map.

We give a precise statement. If~$\Phi$ is a filtration on $\rH^i_\deR(X_{y_0}/K_v)$, we define a filtered discrete $(\varphi,N,G_v)$-module~$\MM^i(\Phi)$ by
\[
\MM^i(\Phi) \colonequals (\rH^i_\pst(X_{y_0}/\bQ_p^\nr),\rH^i_\deR(X_{y_0}/K_v),c_\deR\circ c_\BO) \,,
\]
where the $(\varphi,N,G_v)$-module structure on $\rH^i_\pst(X_{y_0}/\bQ_p^\nr)$ is the usual one, but the filtration on~$\rH^i_\deR(X_{y_0}/K_v)$ is given by~$\Phi$ instead of the Hodge filtration. The following then follows directly from Theorem~\ref{thm:period_maps_control_reps}.

\begin{prop}[Period maps control variation of local Galois representations]\label{prop:period_maps_control_reps}
Let $\pi\colon X\to Y$ be a smooth proper morphism of smooth $K_v$-varieties, let $y_0\in Y(K_v)$ be a $K_v$-rational point, and let $\Nbd_{y_0}\subseteq Y^\an$ be an admissible open neighbourhood of~$y_0$, isomorphic to a closed polydisc, over which $\cH^i_\deR(X/Y)^\an$ has a full basis of horizontal sections.

Then for every~$y\in\Nbd_{y_0}(K_v)$, there is an isomorphism
\[
\rH^i_\pH(X_y/K_v) \simeq \MM^i(\Phi_{y_0}(y))
\]
of filtered discrete~$(\varphi,N,G_v)$-modules.
\end{prop}

Proposition~\ref{prop:period_maps_control_reps} can be expressed diagrammatically as asserting the commutativity of the diagram
\begin{center}
	\begin{tikzcd}
		Y(K_v) \arrow[r,phantom,"\supset"]\arrow[d, "\rH^i_\et"] &[-8ex] \Nbd_{y_0}(K_v) \arrow[r,"\Phi_{y_0}"] & 
		\dG_{y_0}(K_v) \arrow[d,"\MM^i"] \\
		\pi_0\Rep_{\bQ_p}^\deR(G_v) \arrow[rr,"\DpH",hook]  & & \pi_0\MF(\varphi,N,G_v) 
	\end{tikzcd}
\end{center}
where the left-hand vertical arrow sends the point~$y$ to the isomorphism class of the $G_v$-representation $\rH^i_\et(X_{y,\Kbar_v},\bQ_p)$ and $\pi_0$ denotes the set of isomorphism classes in an essentially small category. Since~$\DpH$ is a fully faithful functor, this diagram shows that the period map~$\Phi_{y_0}$ determines the isomorphism class of the local representation $\rH^i_\et(X_{y,\Kbar_v},\bQ_p)$ for all~$y\in\Nbd_{y_0}(K_v)$.
\smallskip

The rest of this section is devoted to the proof of Theorem~\ref{thm:period_maps_control_reps}. As mentioned before, when $X\to Y$ has suitably good reduction at~$v$, this can be proved by relating both \'etale and de Rham cohomology of~$X_{y_0}$ to the crystalline cohomology of its special fibre. This is the approach taken in~\cite{LVinventiones}. However, in the absence of any assumptions on the reduction type of $X\to Y$, we are forced to take a different approach, using the relative $p$-adic Hodge theory of Scholze \cite{scholze:relative_p-adic_hodge_theory}. Since the proof will be rather technical, the reader who is prepared to take Theorem~\ref{thm:period_maps_control_reps} on faith can skip ahead to the next section~\S\ref{s:abelian-by-finite}.
\smallskip

Let us sketch in outline the proof of Theorem~\ref{thm:period_maps_control_reps} which we will give. If~$Y^\an$ denotes the rigid analytification of~$Y$, then Scholze defines what it means for a $\bQ_p$-local system~$\bE$ and filtered vector bundle~$\dE$ with flat connection on~$Y^\an$ to be \emph{associated}, meaning roughly that they become isomorphic when tensored with a certain sheaf of de Rham periods on the pro-\'etale site of~$Y^\an$. Scholze's relative comparison theorem shows that the relative analytic \'etale cohomology and relative analytic de Rham cohomology of $\pi^\an\colon X^\an\to Y^\an$ are associated in this sense. Moreover, these relative analytic cohomology objects are just the analytifications of the corresponding algebraic relative cohomology objects.

So proving Theorem~\ref{thm:period_maps_control_reps} reduces to proving the following: if~$\bE$ and~$\dE$ are an associated $\bQ_p$-local system and filtered vector bundle with flat connection on~$Y^\an$, then for any two points $y_0,y\in Y(K_v)$ which are sufficiently close in the $v$-adic topology, there is an isomorphism
\[
T_{y_0,y} \colon \sD_\pst(\bE_{\bar y_0}) \xrightarrow\sim \sD_\pst(\bE_{\bar y})
\]
of discrete $(\varphi,N,G_v)$-modules making a certain rectangle commute. This is a consequence of theory developed by Shimizu \cite{shimizu:p-adic_monodromy}, on potential horizontal semistability of horizontal de Rham local systems on spherical polyannuli.

\subsection{Scholze's relative $p$-adic Hodge theory}\label{ss:relative_p-adic_hodge}

We begin by recalling some relative $p$-adic Hodge theory for rigid analytic varieties, as developed by Scholze \cite{scholze:relative_p-adic_hodge_theory}, in particular what it means for a local system and a filtered vector bundle with flat connection to be associated (Definition~\ref{def:de_rham_pair}). The main input we need from Scholze's theory is the relative comparison isomorphism (Theorem~\ref{thm:pushforward_of_pairs}), which shows that the relative analytic \'etale and analytic de Rham cohomology of smooth proper morphisms are associated in this sense.

Let~$U$ be a smooth $K_v$-analytic space. Scholze associates to~$U$ a site~$U_\proet$, called the pro-\'etale site \cite[Definition~3.9]{scholze:relative_p-adic_hodge_theory}\cite[Erratum~(1)]{scholze:relative_p-adic_hodge_theory_erratum}. This site carries several important sheaves of $\bZ_p$-algebras, including:
\begin{itemize}
	\item the structure sheaf~$\cO_U$ (a sheaf of $K_v$-algebras);
	\item the sheaf~$\hat\bZ_{p,U}\colonequals\varprojlim(\underline{\bZ/p^n}_U)$ where $\underline{\bZ/p^n}_U$ denotes the constant sheaf on~$U_\proet$ ($\hat\bZ_{p,U}$ is not the constant sheaf with value~$\bZ_p$, but plays a very similar role) \cite[Definition~8.1]{scholze:relative_p-adic_hodge_theory};
	\item the \emph{positive de Rham sheaf} $\bB_{\deR,U}^+$ \cite[Definition~6.1(ii)]{scholze:relative_p-adic_hodge_theory}, which is a sheaf of $\hat\bZ_{p,U}$-algebras; and
	\item the \emph{structural de Rham sheaf} $\cO\bB_{\deR,U}$ \cite[Definition~6.8(iv)]{scholze:relative_p-adic_hodge_theory}\cite[Erratum~(3)]{scholze:relative_p-adic_hodge_theory_erratum}, which is a sheaf of filtered $\cO_U$-algebras endowed with a flat connection
	\[
	\nabla\colon\cO\bB_{\deR,U}\to\Omega^1_{U/K_v}\otimes_{\cO_U}\cO\bB_{\deR,U}
	\]
	satisfying the Leibniz rule $\nabla(fg)=\nabla(f)g+\nabla(g)f$, extending the connection~$\rd$ on~$\cO_U$, and satisfying Griffiths transversality with respect to the filtration~$\rF^\bullet$. The positive de Rham period sheaf~$\bB_{\deR,U}^+=(\rF^0\!\cO\bB_{\deR,U})^{\nabla=0}$ is the sheaf of horizontal sections of~$\rF^0\!\cO\bB_{\deR,U}$ \cite[Proof of Lemma~7.7]{scholze:relative_p-adic_hodge_theory}.
\end{itemize}

These sheaves of algebras allow one to identify several important categories of sheaves on $U_\proet$.

\begin{defi}[Local systems and vector bundles]
\leavevmode
\begin{enumerate}
	\item A \emph{$\hat\bZ_p$-local system} on~$U$ is a sheaf~$\bE$ of~$\hat\bZ_{p,U}$-modules which, locally on~$U_\proet$, is isomorphic to the sheaf $\hat\bZ_{p,U}\otimes_{\bZ_p}M$ for a finitely generated $\bZ_p$-module~$M$ \cite[Definition~8.1]{scholze:relative_p-adic_hodge_theory}.
	\item A \emph{$\bB_\deR^+$-local system} on~$U$ is a sheaf~$\bM$ of~$\bB_{\deR,U}^+$-modules which, locally on~$U_\proet$, is free of finite rank \cite[Definition~7.1]{scholze:relative_p-adic_hodge_theory}.
	\item A \emph{filtered vector bundle with flat connection} on~$U$ is a locally finite free sheaf~$\dE$ of $\cO_U$-modules, together with an exhaustive, separated decreasing filtration $\rF^\bullet\!\dE$ whose graded pieces are all locally finite free $\cO_U$-modules, and a flat connection
	\[
	\nabla\colon\dE\to\Omega^1_{U/K_v}\otimes_{\cO_U}\dE
	\]
	satisfying the Leibniz rule with respect to the connection~$\rd$ on~$\cO_U$ and Griffiths transversality with respect to the filtration~$\rF^\bullet$ \cite[Definition~7.4]{scholze:relative_p-adic_hodge_theory}.
	\item A \emph{filtered $\cO\bB_\deR$-vector bundle with flat connection} on~$U$ is a sheaf~$\dM$ of filtered~$\cO\bB_{\deR,U}$-modules which, locally on~$U_\proet$, is filtered-isomorphic to $\cO\bB_{\deR,U}^{\oplus m}$ for some~$m$, together with a flat connection
	\[
	\nabla\colon\dM\to\Omega^1_{U/K_v}\otimes_{\cO_U}\dM
	\]
	satisfying the Leibniz rule with respect to the connection~$\nabla$ on~$\cO\bB_{\deR,U}$ and Griffiths transversality with respect to the filtration~$\rF^\bullet$ on~$\dM$.
\end{enumerate}
\end{defi}

We write~$\Loc(U,\hat\bZ_p)$, $\Loc(U,\bB_\deR^+)$, $\FMIC(U,\cO_U)$ and $\FMIC(U,\cO\bB_\deR)$ for the $\otimes$-categories of $\hat\bZ_p$-local systems, $\bB_\deR^+$-local systems, filtered vector bundles with flat connection, and filtered $\cO\bB_\deR$-vector bundles with flat connection on~$U$, respectively. There are $\bZ_p$-linear $\otimes$-functors
\begin{equation}\label{eq:base-change_functors}
\Loc(U,\hat\bZ_p) \rightarrow \Loc(U,\bB_\deR^+) \rightarrow \FMIC(U,\cO\bB_\deR) \leftarrow \FMIC(U,\cO_U)
\end{equation}
given by base change along the morphisms\footnote{There is a small thing to be checked here: that if $\dE$ is a filtered vector bundle with flat connection on~$U$, then $\cO\bB_{\deR,U}\otimes_{\cO_U}\dE$ is locally filtered-isomorphic to~$\cO\bB_{\deR,U}^{\oplus m}$ for some~$m$. This follows from the fact that~$\dE$ and all graded pieces of its filtration are locally free of finite rank, and the fact that, locally on~$U_\proet$, there is a section~$t$ of~$\cO\bB_{\deR,U}$ such that multiplication by~$t$ gives a filtered isomorphism~$\cO\bB_{\deR,U}\xrightarrow\sim\cO\bB_{\deR,U}$ of degree~$1$.}
\[
\hat\bZ_{p,U} \hookrightarrow \bB_{\deR,U}^+ \hookrightarrow \cO\bB_{\deR,U} \hookleftarrow \cO_U \,.
\]

\begin{prop}\label{prop:p-adic_riemann-hilbert}
The functors 
\[
\Loc(U,\bB_\deR^+) \hookrightarrow \FMIC(U,\cO\bB_\deR) \hookleftarrow \FMIC(U,\cO_U)
\]
are fully faithful, and the image of $\FMIC(U,\cO_U)\hookrightarrow\FMIC(U,\cO\bB_\deR)$ is contained in the essential image of $\Loc(U,\bB_\deR^+) \hookrightarrow \FMIC(U,\cO\bB_\deR)$.
\begin{proof}
Let $\dI\subseteq\FMIC(U,\cO\bB_\deR)$ denote the full subcategory consisting of those filtered $\cO\bB_\deR$-vector bundles with flat connection~$\dM$ such that\footnote{The authors do not know whether this condition is in fact true for all~$\dM\in\FMIC(U,\cO\bB_\deR)$.}~$\rF^0\!\dM^{\nabla=0}$ is a $\bB_\deR^+$-local system on~$U$ (for the natural~$\bB_{\deR,U}^+=\rF^0\!\cO\bB_{\deR,U}^{\nabla=0}$-module structure). The proof of \cite[Theorem~7.6(i)]{scholze:relative_p-adic_hodge_theory} shows that for any $\bB_\deR^+$-local system~$\bM$ on~$U$, the natural map $\eta_\bM\colon\bM\to\rF^0(\cO\bB_{\deR,U}\otimes_{\bB_{\deR,U}^+}\bM)^{\nabla=0}$ is an isomorphism. Thus the image of the functor $\cO\bB_{\deR,U}\otimes_{\bB_{\deR,U}^+}(-)\colon\Loc(U,\bB_\deR^+)\to\FMIC(U,\cO\bB_\deR)$ has image contained in~$\dI$.

Now the construction~$\dM\mapsto\rF^0\!\dM^{\nabla=0}$ provides a functor $\dI\to\Loc(U,\bB_\deR^+)$ which is right adjoint to the functor $\cO\bB_{\deR,U}\otimes_{\bB_{\deR,U}^+}(-)$. The unit of this adjunction is the natural isomorphism $\eta_\bM$ above, which implies that the functor $\cO\bB_{\deR,U}\otimes_{\bB_{\deR,U}}(-)\colon\Loc(U,\bB_\deR^+)\to\FMIC(U,\cO\bB_\deR)$ is fully faithful.

\smallskip

Then the proof of \cite[Theorem~7.6(ii)]{scholze:relative_p-adic_hodge_theory} shows that for every filtered vector bundle with flat connection~$\dE$, there is a $\bB_\deR^+$-local system~$\bM$ such that $\cO\bB_{\deR,U}\otimes_{\cO_U}\dE$ is isomorphic to $\cO\bB_{\deR,U}\otimes_{\bB_{\deR,U}^+}\bM$ compatibly with the filtrations and connection. This says that the essential image of the functor $\cO\bB_{\deR,U}\otimes_{\cO_U}(-)\colon\FMIC(U,\cO_U)\to\FMIC(U,\cO\bB_\deR)$ is contained in the essential image of the functor $\cO\bB_{\deR,U}\otimes_{\bB_{\deR,U}^+}(-)\colon\Loc(U,\bB_\deR^+)\to\FMIC(U,\cO\bB_\deR)$. The fact that the functor $\dE\mapsto\cO\bB_{\deR,U}\otimes_{\cO_U}\dE$ is fully faithful follows then from the fact that the composite functor $\dE\mapsto\rF^0\!(\cO\bB_{\deR,U}\otimes_{\cO_U}\dE)^{\nabla=0}$ is fully faithful \cite[Theorem~7.6(ii)]{scholze:relative_p-adic_hodge_theory}.
\end{proof}
\end{prop}

\begin{rmk}
In \cite{scholze:relative_p-adic_hodge_theory}, Scholze does not work with filtered $\cO\bB_\deR$-vector bundles with flat connection: instead, he uses $\cO\bB_\deR^+$-vector bundles with flat connection \cite[Definition~7.1(ii)]{scholze:relative_p-adic_hodge_theory} where $\cO\bB_{\deR,U}^+$ is the positive structural de Rham period sheaf \cite[Definition~6.8(iii)]{scholze:relative_p-adic_hodge_theory}\cite[Erratum~(3)]{scholze:relative_p-adic_hodge_theory_erratum}. The advantage of Scholze's approach is that the functor $\cO\bB_{\deR,U}^+\otimes_{\bB_{\deR,U}^+}(-)$ from $\bB_\deR^+$-local systems on~$U$ to $\cO\bB_\deR^+$-vector bundles with flat connection is then an equivalence \cite[Theorem~7.2]{scholze:relative_p-adic_hodge_theory} -- this is a $p$-adic version of the Riemann--Hilbert correspondence. On the other hand, defining a functor directly from $\FMIC(U,\cO_U)$ to $\MIC(U,\cO\bB_\deR^+)$ is not as simple as just base-changing along a morphism of sheaves of rings, so in the interest of simplicity of exposition, we have chosen to focus on the category of filtered $\cO\bB_\deR$-vector bundles with flat connection instead.
\end{rmk}

The $\otimes$-functors~\eqref{eq:base-change_functors} allow one to isolate a class of \emph{de Rham $\hat\bZ_p$-local systems} on~$U$, which are those for which there is a corresponding filtered vector bundle with flat connection under~\eqref{eq:base-change_functors}.

\begin{defi}\label{def:de_rham_pair}
	A \emph{de Rham pair} on~$U$ is a triple $(\bE,\dE,c)$ consisting of a $\hat\bZ_p$-local system~$\bE$ on~$U$, a filtered vector bundle with flat connection~$\dE$ on~$U$, and an isomorphism
	\[
	c\colon \cO\bB_{\deR,U}\otimes_{\hat\bZ_{p,U}}\bE\xrightarrow\sim\cO\bB_{\deR,U}\otimes_{\cO_U}\dE
	\]
	of filtered $\cO\bB_{\deR,U}$-vector bundles with flat connection on~$U$. The collection of all de Rham pairs naturally forms a $\bZ_p$-linear abelian $\otimes$-category, the precise formulation of which we leave to the reader.
	
	It follows from Proposition~\ref{prop:p-adic_riemann-hilbert} that the forgetful functor from de Rham pairs to $\hat\bZ_p$-local systems on~$U$ is fully faithful. We say that a $\hat\bZ_p$-local system~$\bE$ is \emph{de Rham} just when it lies in the essential image of this functor, i.e.\ just when~$\bE$ is the first component of a de Rham pair $(\bE,\dE,c)$.
\end{defi}

\begin{rmk}
	A $\hat\bZ_p$-local system~$\bE$ and filtered vector bundle~$\dE$ appearing in a de Rham pair $(\bE,\dE,c)$ are said to be \emph{associated} \cite[Definition~7.5]{scholze:relative_p-adic_hodge_theory}.
\end{rmk}

\subsubsection{Pullback of de Rham pairs}

The theory described above is contravariant-functorial in the smooth $K_v$-analytic space~$U$ in a natural way. Suppose that $f\colon V\to U$ is a morphism of smooth rigid-analytic spaces over~$K_v$. This induces a morphism $f\colon V_\proet\to U_\proet$ on pro-\'etale sites, and there are morphisms of sheaves of $\bZ_p$-algebras
\begin{align*}
	f^{-1}\hat\bZ_{p,U} &\to \hat\bZ_{p,V} & f^{-1}\bB_{\deR,U}^+ &\to \bB_{\deR,V}^+ \\
	f^{-1}\cO_U &\to \cO_V & f^{-1}\cO\bB_{\deR,U} &\to \cO\bB_{\deR,V}
\end{align*}
compatible with all inclusions among these sheaves. Via these maps, we obtain pullback $\otimes$-functors
\begin{align*}
	f^*\colon\Loc(U,\hat\bZ_p) &\to \Loc(V,\hat\bZ_p) & f^*\colon\Loc(U,\bB_\deR^+) &\to \Loc(V,\bB_\deR^+) \\
	f^*\colon\FMIC(U,\cO_U) &\to \FMIC(V,\cO_V) & f^*\colon\FMIC(U,\cO\bB_\deR) &\to \FMIC(V,\cO\bB_\deR) \,.
\end{align*}

Now if~$\bE$ is a $\hat\bZ_p$-local system on~$U$, resp.~$\dE$ is a filtered vector bundle with flat connection on~$U$, then we have\footnote{Strictly speaking, these are canonical isomorphisms of filtered $\cO\bB_\deR$-vector bundles on~$V$ rather than actual equalities.}
\[
\cO\bB_{\deR,V}\otimes_{\hat\bZ_{p,V}}f^*\bE = f^*(\cO\bB_{\deR,U}\otimes_{\hat\bZ_{p,U}}\bE) \,, \hspace{0.4cm}\text{resp.}\hspace{0.4cm} \cO\bB_{\deR,V}\otimes_{\cO_V}f^*\dE = f^*(\cO\bB_{\deR,U}\otimes_{\cO_U}\dE) \,.
\]
In particular, if $c\colon\cO\bB_{\deR,U}\otimes_{\hat\bZ_{p,U}}\bE\xrightarrow\sim\cO\bB_{\deR,U}\otimes_{\cO_U}\dE$ is an isomorphism of filtered $\cO\bB_\deR$-vector bundles, then so too is $f^*c\colon\cO\bB_{\deR,V}\otimes_{\hat\bZ_{p,V}}f^*\bE \xrightarrow\sim \cO\bB_{\deR,V}\otimes_{\cO_V}f^*\dE$. We have thus proven the following.

\begin{prop}
	Let~$f\colon V\to U$ be a morphism of smooth $K_v$-analytic spaces, and suppose that $(\bE,\dE,c)$ is a de Rham pair on~$U$. Then~$f^*(\bE,\dE,c)\colonequals(f^*\bE,f^*\dE,f^*c)$ is a de Rham pair on~$V$.
\end{prop}

\subsubsection{Derived pushforward of de Rham pairs}\label{sss:derived_pushforward}

In certain circumstances, it is also possible to push forward local systems and filtered vector bundles with flat connection. This will all take place in the setting of derived categories and filtered derived categories; for precise definitions and basic properties, see Appendix~\ref{appx:derived}. Note in particular that the definition of the filtered derived category in \cite{scholze:relative_p-adic_hodge_theory} corresponds to what we refer to as the \emph{complete} filtered derived category.

Suppose that~$\pi\colon V\to U$ is a smooth proper morphism of smooth $K_v$-analytic spaces. If~$\bE$ is a $\hat\bZ_p$-local system on~$V$, then the derived pushforward~$\rR^i\!\pi_*\bE$ is a $\hat\bZ_{p,U}$-module on~$U_\proet$. In fact, $\rR^i\!\pi_*\bE$ is always a $\hat\bZ_p$-local system on~$U$, as follows from \cite[Corollary~6.3.5]{diao-lan-liu-zhu:foundations} (in the case of trivial log structure) and \cite[Proposition~8.2]{scholze:relative_p-adic_hodge_theory} and the remark below it.

The derived pushforward of a filtered vector bundle with flat connection is slightly more complicated to define. If~$\dE$ is a filtered vector bundle with flat connection on~$V$, then one can form its \emph{relative de Rham complex}
\[
\DeR_{V/U}(\dE)\colonequals\dE\xrightarrow\nabla\Omega^1_{V/U}\otimes_{\cO_V}\dE\xrightarrow\nabla\Omega^2_{V/U}\otimes_{\cO_V}\dE\xrightarrow\nabla\dots
\]
defined e.g.\ as in \cite[\S1]{katz-oda:relative_de_rham}. This is a filtered complex of sheaves of $\pi^{-1}\cO_U$-modules on~$V$, where the filtration on~$\Omega^n_{V/U}\otimes_{\cO_V}\dE$ is the tensor product of the filtration on~$\dE$ and the filtration placing~$\Omega^n_{V/U}$ in degree~$n$. Hence the \emph{derived de Rham pushforward} $\rR^i\!\pi_{\deR*}\dE\colonequals\rR^i\!\pi_*\DeR_{V/U}(\dE)$ is a filtered sheaf of $\cO_U$-modules, where~$\rF^n\!\rR^i\!\pi_{\deR*}\dE$ is the image of the map
\[
\rR^i\!\pi_*(\rF^n\!\DeR_{V/U}(\dE)) \to \rR^i\!\pi_*(\DeR_{V/U}(\dE)) \,.
\]

Moreover, the derived de Rham pushforward $\rR^i\!\pi_{\deR*}\dE$ comes with a flat Gau\ss--Manin connection
\[
\nabla\colon\rR^i\!\pi_{\deR*}\dE \to \Omega^1_U\otimes_{\cO_U}\rR^i\!\pi_{\deR*}\dE
\]
satisfying Griffiths transversality with respect to the filtration. This is defined in the usual way~\cite{katz-oda:relative_de_rham}. The \emph{absolute de Rham complex}
\[
\DeR_V(\dE)\colonequals\dE\xrightarrow\nabla\Omega^1_{V}\otimes_{\cO_V}\dE\xrightarrow\nabla\Omega^2_{V}\otimes_{\cO_V}\dE\xrightarrow\nabla\dots
\]
is naturally a left dg-module under the de Rham complex $\Omega^\bullet_V$. There is a natural ``relatively stupid'' filtration $\sigma^\bullet$ on~$\Omega^\bullet_V$, where $\sigma^n\Omega^\bullet_V$ is the dg-ideal of~$\Omega^\bullet_V$ generated by $\pi^{-1}\Omega^n_U$. The graded pieces of the induced filtration on~$\DeR_{V/U}(\dE)$ can be canonically identified as
\[
\gr_n^\sigma\DeR_V(\dE) \cong \pi^{-1}\Omega^n_U\otimes_{\pi^{-1}\cO_U}\DeR_{V/U}(\dE)[-n] \,,
\]
just as in \cite[(3.2.4)]{katz:nilpotent_connections}\footnote{In \cite[(3.2.4)]{katz:nilpotent_connections}, the right-hand side of this expression is rendered as $\pi^*\Omega^n_U\otimes_{\cO_V}\DeR_{V/U}(\dE)[-n]$, but this does not strictly speaking make sense, since $\DeR_{V/U}(\dE)$ is not a complex of $\cO_V$-algebras (its differential is not $\cO_V$-linear).}. In particular, we have an extension
\[
0 \to \pi^{-1}\Omega^1_U\otimes_{\pi^{-1}\cO_U}\DeR_{V/U}(\dE)[-1] \to \DeR_V(\dE)/\sigma^2 \to \DeR_{V/U}(\dE) \to 0
\]
in the category of filtered complexes of abelian sheaves on~$V$. The coboundary map associated to this extension is, by definition, the Gau\ss--Manin connection~$\nabla$. It follows for purely formal reasons that the connection~$\nabla$ satisfies the Leibniz rule, is flat, and is Griffiths transverse with respect to the filtration.
\smallskip

Analogous constructions apply in the category of filtered $\cO\bB_\deR$-vector bundles with flat connection. If~$\dM$ is a filtered $\cO\bB_\deR$-vector bundle with flat connection on~$V$, then one can form its \emph{relative de Rham complex} $\DeR_{V/U}(\dM)$ exactly as for~$\DeR_{V/U}(\dE)$ above. The relative de Rham complex~$\DeR_{V/U}(\dM)$ is a filtered complex of $\pi^{-1}\cO\bB_{\deR,U}$-modules, so
\[
\rR^i\!\pi_{\deR*}\dM \colonequals \rR^i\!\pi_*(\DeR_{V/U}(\dM))
\]
is a filtered~$\cO\bB_{\deR,U}$-module. Moreover, $\rR^i\!\pi_{\deR*}\dM$ comes with a flat Gau\ss--Manin connection
\[
\nabla\colon\rR^i\!\pi_{\deR*}\dM \to \Omega^1_U\otimes_{\cO_U}\rR^i\!\pi_{\deR*}\dM
\]
satisfying the Leibniz rule with respect to the connection on~$\cO\bB_{\deR,U}$; this is constructed from the absolute de Rham complex~$\DeR_V(\dM)$ exactly as for~$\DeR_V(\dE)$ above.
\medskip

These constructions of the various derived pushforwards are compatible in a natural way. If~$\bE$ is a $\hat\bZ_p$-local system on~$V$, then the map $\bE\to\cO\bB_{\deR,V}\otimes_{\hat\bZ_{p,V}}\bE$ induces a $\hat\bZ_{p,V}$-linear morphism
\[
\bE \to \DeR_{V/U}(\cO\bB_{\deR,V}\otimes_{\hat\bZ_{p,V}}\bE) = \DeR_{V/U}(\cO\bB_{\deR,V})\otimes_{\hat\bZ_{p,V}}\bE
\]
of filtered complexes. Since this map even lifts to a morphism into the absolute de Rham complex $\DeR_V(\cO\bB_{\deR,V}\otimes_{\hat\bZ_{p,V}}\bE)$, it follows that the induced map
\[
\rR^i\!\pi_*\bE\to\rR^i\!\pi_{\deR*}(\cO\bB_{\deR,V}\otimes_{\hat\bZ_{p,v}}\bE)
\]
factors through the kernel of the connection on~$\rR^i\!\pi_{\deR*}(\cO\bB_{\deR,V}\otimes_{\hat\bZ_{p,v}}\bE)$. Hence the induced  $\cO\bB_{\deR,U}$-linear map
\[
\eta_\bE\colon\cO\bB_{\deR,U}\otimes_{\hat\bZ_{p,U}}\rR^i\!\pi_*\bE \to \rR^i\!\pi_{\deR*}(\cO\bB_{\deR,V}\otimes_{\hat\bZ_{p,v}}\bE)
\]
is compatible with connection and filtration.

Similarly, given a filtered vector bundle with flat connection~$\dE$ on~$V$, there is an evident morphism of relative de Rham complexes
\[
\DeR_{V/U}(\dE) \to \DeR_{V/U}(\cO\bB_{\deR,V}\otimes_{\cO_V}\dE)
\]
which is $\pi^{-1}\cO_U$-linear and compatible with filtrations. Since this morphism extends to a morphism between the absolute de Rham complexes, it follows that the induced $\cO\bB_{\deR,U}$-linear map
\[
\eta_\dE\colon\cO\bB_{\deR,U}\otimes_{\cO_U}\rR^i\!\pi_{\deR*}\dE \to \rR^i\!\pi_{\deR*}(\cO\bB_{\deR,V}\otimes_{\cO_V}\dE)
\]
is compatible with the connection and filtration.
\smallskip

Scholze proves the following comparison theorem relating these constructions.

\begin{thm}[{\cite[Theorem~1.10]{scholze:relative_p-adic_hodge_theory}}]\label{thm:pushforward_of_pairs}
	Let $\pi\colon V\to U$ be a smooth proper morphism of smooth rigid-analytic spaces over~$K_v$, and let~$(\bE,\dE,c)$ be a de Rham pair on~$V$. Then:
	\begin{enumerate}
		\item $\rR^i\!\pi_{\deR*}\dE$ is a filtered vector bundle with flat connection on~$U$ for all~$i$;
		\item\label{thmpart:completed_eta} the maps~$\eta_\bE$ and~$\eta_\dE$ above become isomorphisms after completing with respect to the filtration; and
		\item there is a unique isomorphism $\tilde c\colon\cO\bB_{\deR,U}\otimes_{\hat\bZ_{p,U}}\rR^i\!\pi_*\bE\xrightarrow\sim\cO\bB_{\deR,U}\otimes_{\cO_U}\rR^i\!\pi_{\deR*}\dE$ of filtered $\cO\bB_\deR$-vector bundles on~$U$ making the square
		\begin{center}
		\begin{tikzcd}[column sep = huge]
			(\cO\bB_{\deR,U}\otimes_{\hat\bZ_{p,U}}\rR^i\!\pi_*\bE)^\wedge \arrow[r,dashed,"\tilde c^\wedge","\sim"']\arrow[d,"\eta_\bE^\wedge","\wr"'] & (\cO\bB_{\deR,U}\otimes_{\cO_U}\rR^i\!\pi_{\deR*}\dE)^\wedge \arrow[d,"\eta_\dE^\wedge","\wr"'] \\
			\rR^i\!\pi_{\deR*}(\cO\bB_{\deR,V}\otimes_{\hat\bZ_{p,V}}\bE)^\wedge \arrow[r,"(\rR^i\!\pi_{\deR*}c)^\wedge","\sim"'] & \rR^i\!\pi_{\deR*}(\cO\bB_{\deR,V}\otimes_{\cO_V}\dE)^\wedge
		\end{tikzcd}
		\end{center}
		commute ($(-)^\wedge$ denotes completion with respect to the filtration).
	\end{enumerate}
	We write $\rR^i\!\pi_*(\bE,\dE,c)$ for the de Rham pair $(\rR^i\!\pi_*\bE,\rR^i\!\pi_{\deR*}\dE,\tilde c)$.
\end{thm}

\begin{rmk}
	It seems likely that~$\eta_\bE$ and~$\eta_\dE$ should be isomorphisms even before completing. However, we were unable to extract this from the theory in \cite{scholze:relative_p-adic_hodge_theory}. Note that $\cO\bB_{\deR,U}$ is not complete for its filtration in general, so the completion operations are non-trivial here.
\end{rmk}

\begin{proof}[Proof of Theorem~\ref{thm:pushforward_of_pairs}]
	This is essentially contained in \cite[Theorem~8.8]{scholze:relative_p-adic_hodge_theory}. We give a small additional commentary on the second two parts, to clarify where completions enter the picture. In the proof of \cite[Theorem~8.8]{scholze:relative_p-adic_hodge_theory}, Scholze shows that the natural maps
	\[
	\cO\bB_{\deR,U}\otimes_{\hat\bZ_{p,U}}\rR\!\pi_*\bE \to \rR\!\pi_{\deR*}(\cO\bB_{\deR,V}\otimes_{\hat\bZ_{p,V}}\bE) \hspace{0.4cm}\text{and}\hspace{0.4cm} \cO\bB_{\deR,U}\otimes_{\cO_U}\rR\!\pi_{\deR*}\dE \to \rR\!\pi_{\deR*}(\cO\bB_{\deR,V}\otimes_{\cO_V}\dE)
	\]
	(which act on cohomology objects as~$\eta_\bE$ and~$\eta_\dE$ above) are isomorphisms in the complete filtered derived category of abelian sheaves on~$U_\proet$, i.e.\ they induce quasi-isomorphisms on all graded pieces. Both $\cO\bB_{\deR,U}\otimes_{\hat\bZ_{p,U}}\rR\!\pi_*\bE$ and $\cO\bB_{\deR,U}\otimes_{\cO_U}\rR\!\pi_{\deR*}\dE$ have filtered cohomology objects in the sense of~\S\ref{ss:cohomology_objects} in the appendix; the second of these follows from the fact that the relative Hodge--de Rham spectral sequence for~$\dE$ degenerates at the first page using Lemma~\ref{lem:degeneration_implies_filtered_cohomology}. So by Lemma~\ref{lem:iso_on_completed_cohomology}, the two above maps induce isomorphisms on completed cohomology objects, i.e.\ $\eta_\bE^\wedge$ and~$\eta_\dE^\wedge$ are isomorphisms.
	
	The final part is then a special case of the following lemma.
\end{proof}

\begin{lem}\label{lem:decompletion}
	Let~$\dM_0$ and~$\dM_1$ be filtered $\cO\bB_\deR$-vector bundles with flat connection on a smooth rigid-analytic space~$U$ over~$K_v$. Suppose that~$\dM_0$ and~$\dM_1$ lie in the essential image of the functor $\Loc(U,\bB_\deR^+)\hookrightarrow\FMIC(U,\cO\bB_\deR)$. Then every filtered $\cO\bB_{\deR,U}^\wedge$-linear map
	\[
	\phi^\wedge\colon\dM_0^\wedge\to\dM_1^\wedge
	\]
	compatible with connections is the completion of a unique morphism $\phi\colon\dM_0\to\dM_1$ of filtered $\cO\bB_\deR$-vector bundles with flat connection. Moreover, $\phi$ is an isomorphism if and only if $\phi^\wedge$ is.
	\begin{proof}
		Let~$\bM_0\colonequals(\rF^0\!\dM_0)^{\nabla=0}$ be the $\bB_\deR^+$-local system on~$U$ corresponding to~$\dM_0$, so that $\dM_0=\cO\bB_{\deR,U}\otimes_{\bB_{\deR,U}^+}\bM_0$. Since~$\bM_0$ is locally finite free, we have $\dM_0^\wedge=\cO\bB_{\deR,U}^\wedge\otimes_{\bB_{\deR,U}^+}\bM_0$ with the filtration and connection induced from those on $\cO\bB_{\deR,U}^\wedge$. And since $\bB_{\deR,U}^+=(\rF^0\!\cO\bB_{\deR,U})^{\nabla=0}$ is complete, we find by taking completions that $\bB_{\deR,U}^+=(\rF^0\!\cO\bB_{\deR,U}^\wedge)^{\nabla=0}$ and hence
		\[
		\bM_0=(\rF^0\!\dM_0^\wedge)^{\nabla=0} \,.
		\]
		
		It follows from this that every $\cO\bB_{\deR,U}^\wedge$-linear map $\phi^\wedge\colon\dM_0^\wedge\to\dM_1^\wedge$ compatible with filtrations and connections is the base-change of a unique $\bB_{\deR,U}^+$-linear map $\phi\colon\bM_0\to\bM_1$. In light of Proposition~\ref{prop:p-adic_riemann-hilbert}, this implies that~$\phi^\wedge$ is the completion of a unique morphism $\phi\colon\dM_0\to\dM_1$ of filtered $\cO\bB_\deR$-vector bundles with flat connection. The final claim, that~$\phi$ is an isomorphism if and only if~$\phi^\wedge$ is comes by applying the lemma to~$(\phi^\wedge)^{-1}$.
	\end{proof}
\end{lem}

\begin{rmk}
	In the proof of Lemma~\ref{lem:decompletion}, we used the fact that~$\bB_{\deR,U}^+$ was complete for its filtration. This is actually rather subtle: although $\bB_{\deR,U}^+$ is defined as the completion of a filtered period sheaf $\bB_{\inf,U}$, this alone is not enough to guarantee that~$\bB_{\deR,U}^+$ is complete. We sketch the proof. To show that~$\bB_{\deR,U}^+$ is complete, it suffices to show that the natural maps
	\[
	\rF^i\!\bB_{\deR,U}^+/\rF^j\!\bB_{\deR,U}^+\to\rF^i\!\bB_{\inf,U}/\rF^j\!\bB_{\inf,U}
	\]
	are isomorphisms for all $j\geq i\geq0$. Using the sequence
	\[
	0 \to \rF^j\!\bB_{\deR,U}^+ \to \rF^i\!\bB_{\deR,U}^+ \to \rF^i\!\bB_{\inf,U}/\rF^j\!\bB_{\inf,U} \to \rR^1\!\varprojlim_{k\geq j}(\rF^j\!\bB_{\inf,U}/\rF^k\!\bB_{\inf,U}) \,,
	\]
	it suffices to show the vanishing of the right-hand term. This follows from \cite[Lemma~3.18]{scholze:relative_p-adic_hodge_theory}, using the fact that $\gr^k_{\rF}\!\bB_{\inf,U}\cong\hat\cO_U(k)$ has vanishing higher cohomology over any affinoid perfectoid $U'\in U_\proet$ over which $\bZ_p(1)$ is trivial, for any~$k\geq0$ \cite[Lemma~4.10(v)]{scholze:relative_p-adic_hodge_theory}.
\end{rmk}

\subsubsection{Compatibilities}

We will also need to know that Scholze's comparison isomorphism is compatible with cup product and base change. We take care to spell these out carefully here.

First, cup product. Fix a smooth proper morphism~$\pi\colon V\to U$ of smooth rigid-analytic spaces over~$K_v$. If we are given a pairing $\beta_\bE\colon \bE_1\otimes_{\hat\bZ_{p,V}}\bE_2\to\bE_3$ in the category of $\hat\bZ_p$-local systems on~$V$, resp.\ a pairing $\beta_\dE\colon \dE_1\otimes_{\cO_V}\dE_2\to\dE_3$ in the category of filtered vector bundles with flat connection on~$V$, then there is an induced cup product map
\begin{equation}\label{eq:cup_products}
	\beta_{\bE*}\colon\rR^i\!\pi_*\bE_1 \otimes_{\hat\bZ_{p,U}} \rR^j\!\pi_*\bE_2 \to \rR^{i+j}\!\pi_*\bE_3 \,, \hspace{0.4cm}\text{resp.}\hspace{0.4cm} \beta_{\dE*}\colon\rR^i\!\pi_{\deR*}\dE_1 \otimes_{\cO_U} \rR^j\!\pi_{\deR*}\dE_2 \to \rR^{i+j}\!\pi_{\deR*}\bE_3
\end{equation}
for all~$i,j$. The latter of these is induced by the evident $\pi^{-1}\cO_U$-linear pairing $\DeR_{V/U}(\dE_1)\otimes_{\pi^{-1}\cO_U}\DeR_{V/U}(\dE_2)\to\DeR_{V/U}(\dE_3)$ on relative de Rham complexes. Compatibility of Scholze's comparison isomorphism with these cup product maps amounts to the following.

\begin{prop}\label{prop:cup_product_compatibility_analytic}
	Let~$\pi\colon V\to U$ be a smooth proper morphism of smooth rigid-analytic spaces over~$K_v$, and let~$\beta=(\beta_\bE,\beta_\dE)\colon(\bE_1,\dE_1,c_1)\otimes(\bE_2,\dE_2,c_2)\to(\bE_3,\dE_3,c_3)$ be a pairing in the category of de Rham pairs on~$V$. Suppose that the derived pushforwards $\rR^i\!\pi_*\bE_1$, $\rR^i\!\pi_*\bE_2$ and $\rR^i\!\pi_*\bE_3$ are all $\hat\bZ_p$-local systems on~$U$ for all~$i$. Then the cup product maps~\eqref{eq:cup_products} are the components of a pairing
	\[
	\beta_*\colon\rR^i\!\pi_*(\bE_1,\dE_1,c_1)\otimes\rR^j\!\pi_*(\bE_2,\dE_2,c_2) \to \rR^{i+j}\!\pi_*(\bE_3,\dE_3,c_3)
	\]
	in the category of de Rham pairs on~$U$ for all~$i,j$.
	\begin{proof}
		If~$\beta_\dM\colon\dM_1\otimes_{\cO\bB_{\deR,V}}\dM_2\to\dM_3$ is a pairing in the category of filtered $\cO\bB_\deR$-vector bundles on~$V$, then there is an induced cup product pairing
		\[
		\beta_{\dM*}\colon\rR^i\!\pi_{\deR*}(\dM_1)\otimes_{\cO\bB_{\deR,U}}\rR^j\!\pi_{\deR*}(\dM_2) \to \rR^{i+j}\!\pi_{\deR*}(\dM_3)
		\]
		for all~$i,j$, defined analogously to~$\beta_{\dE*}$ above. We consider the diagram
		\begin{equation}\label{diag:cup_product_compatibility}
		\begin{tikzcd}
			\cO\bB_{\deR,U}\otimes\rR^i\!\pi_*(\bE_1)\otimes\rR^j\!\pi_*(\bE_2) \arrow[r,"1\otimes\beta_{\bE*}"]\arrow[d,"\eta_{\bE_1}\otimes\eta_{\bE_2}"]\arrow[ddd,bend right = 75,shift right = 21,dashed,"\wr"'] & \cO\bB_{\deR,U}\otimes\rR^{i+j}\!\pi_*(\bE_3) \arrow[d,"\eta_{\bE_3}"]\arrow[d]\arrow[ddd,bend left = 60,shift left = 12,dashed,"\wr"] \\
			\rR^i\!\pi_{\deR*}(\cO\bB_{\deR,V}\otimes\bE_1)\otimes\rR^j\!\pi_{\deR*}(\cO\bB_{\deR,V}\otimes\bE_2) \arrow[r]\arrow[d,"\rR^i\!\pi_{\deR*}(c_1)\otimes\rR^j\!\pi_{\deR*}(c_2)","\wr"'] & \rR^{i+j}\!\pi_{\deR*}(\cO\bB_{\deR,V}\otimes\bE_3) \arrow[d,"\rR^{i+j}\!\pi_{\deR*}(c_3)","\wr"'] \\
			\rR^i\!\pi_{\deR*}(\cO\bB_{\deR,V}\otimes\dE_1)\otimes\rR^j\!\pi_{\deR*}(\cO\bB_{\deR,V}\otimes\dE_2) \arrow[r] & \rR^{i+j}\!\pi_{\deR*}(\cO\bB_{\deR,V}\otimes\dE_3) \\
			\cO\bB_{\deR,U}\otimes\rR^i\!\pi_{\deR*}(\dE_1)\otimes\rR^i\!\pi_{\deR*}(\dE_2) \arrow[r,"1\otimes\beta_{\dE*}"]\arrow[u,"\eta_{\dE_1}\otimes\eta_{\dE_2}"'] & \cO\bB_{\deR,U}\otimes\rR^i\!\pi_{\deR*}(\dE_3) \arrow[u,"\eta_{\dE_3}"']
		\end{tikzcd}
		\end{equation}
		in the category of filtered $\cO\bB_\deR$-vector bundles on~$U$ (in which subscripts on tensor products are omitted for readability). The horizontal maps are the various cup product pairings, and the dashed vertical maps are the comparison isomorphisms for the de Rham pairs~$\rR^i\!\pi_*(\bE_1,\dE_1,c_1)\otimes\rR^j\!\pi_*(\bE_2,\dE_2,c_2)$ and~$\rR^{i+j}\!\pi_*(\bE_3,\dE_3,c_3)$.
		
		The three squares in~\eqref{diag:cup_product_compatibility} formed by the solid arrows each commute. For the top square, this follows from the commutativity of the square
		\begin{center}
		\begin{tikzcd}
			\bE_1\otimes_{\hat\bZ_{p,V}}\bE_2 \arrow[r]\arrow[d] & \bE_3 \arrow[d] \\
			\DeR_{V/U}(\cO\bB_{\deR,V}\otimes_{\hat\bZ_{p,V}}\bE_1)\otimes_{\pi^{-1}\cO\bB_{\deR,U}}\DeR_{V/U}(\cO\bB_{\deR,V}\otimes_{\hat\bZ_{p,V}}\bE_2) \arrow[r] & \DeR_{V/U}(\cO\bB_{\deR,V}\otimes_{\hat\bZ_{p,V}}\bE_3)
		\end{tikzcd}
		\end{center}
		of pairings of filtered complexes of abelian sheaves on~$U_\proet$; commutativity of the other squares in~\eqref{diag:cup_product_compatibility} follows similarly.
		
		Now once we take completions with respect to the filtration, the two squares in~\eqref{diag:cup_product_compatibility} formed by the dashed arrows and the vertical maps both commute, and all the maps in the rightmost of these squares become isomorphisms. It follows that the outermost square in~\eqref{diag:cup_product_compatibility} formed by the dashed arrows and the top and bottom horizontal maps commutes after completing. By Lemma~\ref{lem:decompletion} it even commutes before completing, so $(\beta_\bE,\beta_\dE)$ is a morphism of de Rham pairs, as desired.
	\end{proof}
\end{prop}

Second, base change. Fix a commuting square
\begin{equation}\label{diag:base-change_square}
\begin{tikzcd}
	V' \arrow[r,"g"]\arrow[d,"\pi'"] & V \arrow[d,"\pi"] \\
	U' \arrow[r,"f"] & U
\end{tikzcd}
\end{equation}
in the category of smooth rigid-analytic spaces over~$K_v$, with~$\pi$ and~$\pi'$ both smooth and proper. This square need not be a base change square. If~$\bE$ is a $\hat\bZ_p$-local system on~$V$, resp.~$\dE$ is a filtered vector bundle with flat connection on~$V$, then there are associated base-change maps
\begin{equation}\label{eq:base_change_maps}
	\bc_\bE\colon f^*\!\rR^i\!\pi_*\bE \to \rR^i\!\pi'_*g^*\bE \,, \hspace{0.4cm}\text{resp.}\hspace{0.4cm} \bc_\dE\colon f^*\!\rR^i\!\pi_{\deR*}\dE \to \rR^i\!\pi'_{\deR*}g^*\dE
\end{equation}
for all~$i$. We recall the construction of the latter for the benefit of the reader. There is a natural map
\[
g^{-1}\DeR_{V/U}(\dE) \to \DeR_{V'/U'}(g^*\dE)
\]
in the category of filtered sheaves of~$g^{-1}\pi^{-1}\cO_U$-modules on~$V'_\proet$. Applying~$\rR^i\!\pi'_*$ and precomposing with the base-change map $f^{-1}\!\rR^i\!\pi_*\to\rR^i\!\pi'_*g^{-1}$ for abelian sheaves yields a $f^{-1}\cO_U$-linear map
\[
f^{-1}\!\rR^i\!\pi_{\deR*}\dE \to \rR^i\!\pi'_{\deR*}g^*\dE \,;
\]
the base change map~$\bc_\dE$ is the unique $\cO_{U'}$-linear map through which this factors. Compatibility of Scholze's comparison isomorphism with these base change maps amounts to the following.

\begin{prop}\label{prop:base_change_compatibility_analytic}
	Suppose we are in the above setup: \eqref{diag:base-change_square} is a commuting square in the category of smooth rigid-analytic spaces over~$K_v$, with~$\pi$ and~$\pi'$ both smooth and proper. Let~$(\bE,\dE,c)$ be a de Rham pair on~$V$, and suppose that $\rR^i\!\pi_*\bE$ and $\rR^i\!\pi'_*g^*\bE$ are $\hat\bZ_p$-local systems on~$U$ and~$U'$, respectively, for all~$i$. Then $\rR^i\!\pi_{\deR*}\dE$ and $\rR^i\!\pi'_{\deR*}g^*\dE$ are filtered vector bundles with flat connection on~$U$ and~$U'$, respectively, and the base-change maps~\eqref{eq:base_change_maps} are the components of a morphism of de Rham pairs
	\[
	\bc_{\bE,\dE}\colon f^*\!\rR^i\!\pi_*(\bE,\dE,c) \to \rR^i\!\pi'_*g^*(\bE,\dE,c)
	\]
	for all~$i$.
	\begin{proof}
		For a filtered $\cO\bB_\deR$-vector bundle~$\dM$ on~$V$ there is a base change map
		\[
		\bc_\dM\colon g^*\rR^i\!\pi'_{\deR*}\dM \to \rR^i\!\pi_{\deR*}g^*\dM
		\]
		for all~$i$, defined analogously to~$\bc_\dE$ above. We consider the diagram
		\begin{equation}\label{diag:base_change_compatibility}
		\begin{tikzcd}
			\cO\bB_{\deR,U'}\otimes f^*\!\rR^i\!\pi'_*\bE \arrow[r,"1\otimes\bc_\bE"]\arrow[d,"f^*\eta_\bE"]\arrow[ddd,bend right = 60,shift right = 12,dashed,"\wr"'] & \cO\bB_{\deR,U'}\otimes\rR^i\!\pi_*g^*\bE \arrow[d,"\eta_{g^*\bE}"]\arrow[ddd,bend left = 60,shift left = 12,dashed,"\wr"] \\
			f^*\!\rR^i\!\pi'_{\deR*}(\cO\bB_{\deR,V}\otimes\bE) \arrow[r]\arrow[d,"f^*\!\rR^i\!\pi'_{\deR*}(c)","\wr"'] & \rR^i\!\pi_{\deR*}(g^*(\cO\bB_{\deR,V}\otimes\bE)) \arrow[d,"\rR^i\!\pi_{\deR*}g^*(c)","\wr"'] \\
			f^*\!\rR^i\!\pi'_{\deR*}(\cO\bB_{\deR,V}\otimes\dE) \arrow[r] & \rR^i\!\pi_{\deR*}(g^*(\cO\bB_{\deR,V}\otimes\dE)) \\
			\cO\bB_{\deR,U'}\otimes f^*\!\rR^i\!\pi'_{\deR*}(\dE) \arrow[r,"1\otimes\bc_\dE"]\arrow[u,"f^*\eta_\dE"'] & \cO\bB_{\deR,U'}\otimes\rR^i\!\pi_{\deR*}g^*\dE \arrow[u,"\eta_{g^*\dE}"']
		\end{tikzcd}
		\end{equation}
		in the category of filtered~$\cO\bB_\deR$-vector bundles on~$U$ (in which subscripts on tensor products are again omitted for readability). The horizontal maps are the various base change maps, and the dashed vertical maps are the comparison isomorphisms for the de Rham pairs~$f^*\!\rR^i\!\pi'_*(\bE,\dE,c)$ and~$\rR^i\!\pi_*g^*(\bE,\dE,c)$.
		
		The three squares in~\eqref{diag:base_change_compatibility} formed by the solid arrows each commute. To see this for the top square, start from the commuting square
		\begin{center}
		\begin{tikzcd}
			g^{-1}\bE \arrow[r]\arrow[d] & g^*\bE \arrow[d] \\
			g^{-1}\DeR_{V/U}(\cO\bB_{\deR,V}\otimes\bE) \arrow[r] & \DeR_{V'/U'}(\cO\bB_{\deR,V'}\otimes g^*\bE)
		\end{tikzcd}
		\end{center}
		in the category of filtered complexes of abelian sheaves on~$U'_\proet$. Applying the functor~$\rR^i\!\pi_*$ and precomposing the horizontal maps with the base change map $f^{-1}\!\rR^i\!\pi'_*\to\rR^i\!\pi_*g^*$ for abelian sheaves yields a commuting square
		\begin{center}
		\begin{tikzcd}
			f^{-1}\!\rR^i\!\pi'_*\bE \arrow[r]\arrow[d] & \rR^i\!\pi_*g^*\bE \arrow[d] \\
			f^{-1}\!\rR^i\!\pi'_{\deR*}(\cO\bB_{\deR,V}\otimes\bE) \arrow[r] & \rR^i\!\pi_{\deR*}g^*(\cO\bB_{\deR,V}\otimes\bE)
		\end{tikzcd}
		\end{center}
		in the category of filtered abelian sheaves on~$U'_\proet$. The top map is $f^{-1}\hat\bZ_{p,U}$-linear, so factors uniquely through a $\hat\bZ_{p,U'}$-linear map $f^*\!\rR^i\!\pi'_*\bE \to \rR^i\!\pi_*g^*\bE$, namely the base change map~$\bc_\bE$. Similarly, the bottom map factors uniquely through a $\cO\bB_{\deR,U'}$-linear map $f^*\!\rR^i\!\pi'_{\deR*}(\cO\bB_{\deR,V}\otimes\bE) \to \rR^i\!\pi_{\deR*}g^*(\cO\bB_{\deR,V}\otimes\bE)$, which is the base change map for~$\cO\bB_{\deR,V}\otimes\bE$. Commutativity of the top square in~\eqref{diag:base_change_compatibility} follows; commutativity of the other squares in~\eqref{diag:base_change_compatibility} follow similarly.
		
		Now once we take completions with respect to the filtration, the two squares in~\eqref{diag:base_change_compatibility} formed by the dashed arrows and the vertical maps both commute, and all the maps in the rightmost of these squares become isomorphisms. It follows that the outermost square in~\eqref{diag:base_change_compatibility} formed by the dashed arrows and the top and bottom horizontal maps commutes after completing. By Lemma~\ref{lem:decompletion} it even commutes before completing, so $(\bc_\bE,\bc_\dE)$ is a morphism of de Rham pairs, as desired.
	\end{proof}
\end{prop}

\subsection{Analytification of algebraic maps}

We now specialise the preceding discussion to the case of analytifications of algebraic varieties, following \cite[\S9]{scholze:relative_p-adic_hodge_theory}.

If~$X$ is a smooth algebraic variety over~$K_v$, then there is a natural morphism of sites $X^\an_\proet\to X_\et$, and we write $(-)^\an$ for the pullback map on categories of sheaves. So if~$\dE$ is a filtered vector bundle\footnote{Here, we always think of algebraic vector bundles as locally finite free $\cO_X$-modules in the \'etale topology. This is equivalent to the usual definition with the Zariski topology.} with flat connection on~$X$, then~$\dE^\an$ is a filtered vector bundle with flat connection on~$X^\an$, and if~$\bE$ is a $\bZ_p$-local system on~$X$ given as the inverse system of $\bZ/p^n$-local systems~$\bE_n$, then $\hat\bE^\an\colonequals\varprojlim\bE_n^\an$ is a $\hat\bZ_p$-local system on~$X^\an$.

Suppose now that~$\pi\colon X\to Y$ is a smooth proper morphism of smooth $K_v$-varieties, so that we have a 2-commuting square
\begin{equation}\label{diag:analytification_square}
	\begin{tikzcd}
		X^\an_\proet \arrow[r]\arrow[d,"\pi^\an"] & X_\et \arrow[d,"\pi"] \\
		Y^\an_\proet \arrow[r] & Y_\et
	\end{tikzcd}
\end{equation}
of sites. Scholze proves that the expected compatibility between analytic and algebraic derived pushforwards holds. We will use this in the proof of Theorem~\ref{thm:period_maps_control_reps} to relate relative algebraic \'etale and de Rham cohomology by comparing their analytic cousins via Scholze's relative comparison theorem.

\begin{prop}\label{prop:analytification_of_algebraic}\leavevmode
	\begin{enumerate}
		\item If~$\dE$ is a filtered vector bundle with flat connection on~$X$, then the base change map
		\[
		\left(\rR^i\!\pi_{\deR*}(\dE)\right)^\an \to \rR^i\!\pi^\an_{\deR*}(\dE^\an)
		\]
		associated to the square~\eqref{diag:analytification_square} is an isomorphism for all~$i$.
		\item If~$\bE$ is a $\bZ_p$-local system on~$X$, then the derived pushforward $\rR^i\!\pi^\an_*(\hat\bE^\an)$ is a $\hat\bZ_p$-local system on~$Y^\an$ for all~$i$ and the base change map
		\[
		\left(\rR^i\!\pi_{\et*}(\bE)\right)^{\wedge,\an} \to \rR^i\!\pi^\an_*(\hat\bE^\an)
		\]
		associated to the square~\eqref{diag:analytification_square} is an isomorphism for all~$i$.
	\end{enumerate}
	\begin{proof}
		See \cite[Theorem~9.1(ii)]{scholze:relative_p-adic_hodge_theory} for the first point; the second follows from a combination of \cite[Theorem~9.3]{scholze:relative_p-adic_hodge_theory}, \cite[Corollary~3.17(ii)]{scholze:relative_p-adic_hodge_theory} and \cite[Proposition~8.2]{scholze:relative_p-adic_hodge_theory}.
	\end{proof}
\end{prop}

\subsection{The absolute comparison theorem for algebraic varieties}\label{ss:c_dR}

Before continuing with the proof of Theorem~\ref{thm:period_maps_control_reps}, let us take the time to precisely spell out the definition of the comparison isomorphism~$c_\deR$ appearing there, as promised in Remark~\ref{rmk:which_comparison_iso}.

This essentially amounts to unpacking Scholze's comparison theory in the case of the analytification of a smooth proper morphism $\pi\colon X\to \Spec(L_w)$ for a finite extension~$L_w$ of~$K_v$ contained in~$\Kbar_v$. We write~$U=\Sp(L_w)=\Spec(L_w)^\an$, and consider the object~$\Ubar\in U_\proet$ given by
\[
\Ubar = \varprojlim_{L_w'}\Sp(L_w') \,,
\]
where the colimit is taken over all finite extensions~$L_w'$ of~$L_w$ contained in~$\Kbar_v$. There is a natural right action of the absolute Galois group~$G_w$ of~$L_w$ on~$\Ubar$, and so for any sheaf~$\dF$ on~$U_\proet$ there is an induced left $G_w$-action on $\dF(\Ubar)$. In the particular case of the de Rham period sheaf~$\cO\bB_{\deR,U}$, its sections over~$\Ubar$ exactly recovers Fontaine's ring~$\sB_\deR$ of de Rham periods, with its Hodge filtration and Galois action.

So, applying Theorem~\ref{thm:pushforward_of_pairs} to the map $\pi^\an\colon X^\an \to \Sp(L_w)$ and the de Rham pair $(\bE,\dE,c)=(\hat\bZ_{p,X^\an},\cO_{X^\an},1)$, we obtain a series of isomorphisms
\begin{align*}
	\cO\bB_{\deR,U}\otimes_{\hat\bZ_{p,U}}\bigl(\rR^i\!\pi_{\et*}\bZ_{p,X}\bigr)^{\wedge,\an} &\cong \cO\bB_{\deR,U}\otimes_{\hat\bZ_{p,U}}\rR^i\!\pi^\an_*\hat\bZ_{p,X^\an} \\
	&\cong \cO\bB_{\deR,U}\otimes_{\cO_U}\rR^i\!\pi^\an_{\deR*}\cO_{X^\an} \\
	&\cong \cO\bB_{\deR,U}\otimes_{\cO_U}\bigl(\rR^i\!\pi_{\deR*}\cO_X\bigr)^\an
\end{align*}
of filtered $\cO\bB_\deR$-vector bundles with flat connection on~$U$, the first and third of which are the isomorphisms from Proposition~\ref{prop:analytification_of_algebraic}. Taking sections over~$\Ubar$, we thus obtain a comparison isomorphism
\begin{equation}\label{eq:comparison_iso}
c_\deR\colon \sB_\deR\otimes_{\bQ_p}\rH^i_\et(X_{\Lbar_w},\bQ_p) \xrightarrow\sim \sB_\deR\otimes_{L_w}\rH^i_\deR(X/L_w)
\end{equation}
for every smooth proper algebraic variety~$X/L_w$, which is $\sB_\deR$-linear, $G_w$-equivariant and strictly compatible with filtrations. (We write~$\Lbar_w$ instead of~$\Kbar_v$ on the left-hand side to emphasis that the base-change is from~$L_w$ to~$\Lbar_w=\Kbar_v$, not from~$K_v$ to~$\Kbar_v$.) This is the comparison isomorphism~$c_\deR$ for which we will prove Theorem~\ref{thm:period_maps_control_reps}.

\subsubsection{Compatibilities}

We will also need to know that the comparison isomorphism~$c_\deR$ from~\eqref{eq:comparison_iso} is compatible with the usual constructions in cohomology. As we said in Remark~\ref{rmk:which_comparison_iso}, these compatibilities are already known for the comparison isomorphisms constructed by Faltings, Tsuji and others; however, it is important for us that these compatibilities hold for the comparison isomorphism defined above, so we have to check this by hand.

The compatibilities come in two types. The first are those compatibilities which follow formally from the setup.

\begin{prop}\label{prop:compatibility_i}
	The comparison isomorphism~$c_\deR$ from~\eqref{eq:comparison_iso} has the following properties.
	\begin{enumerate}
		\item\label{proppart:compatibility_cup_products} For any smooth proper variety~$X/L_w$, the total comparison isomorphism
		\[
		c_\deR\colon \sB_\deR\otimes_{\bQ_p}\rH^\bullet_\et(X_{\Lbar_w},\bQ_p) \xrightarrow\sim \sB_\deR\otimes_{L_w}\rH^\bullet_\deR(X/L_w)
		\]
		is an isomorphism of graded algebras with respect to cup product.
		\item\label{proppart:compatibility_functoriality} For fixed~$L_w$, the isomorphism~$c_\deR$ is natural in~$X$.
		\item\label{proppart:compatibility_base_change} For a finite extension~$L_{w_2}/L_{w_1}$ contained in~$\Kbar_v$ and a smooth proper variety~$X/L_{w_1}$, the square
		\begin{center}
		\begin{tikzcd}
			\sB_\deR\otimes_{\bQ_p}\rH^i_\et(X_{\Lbar_{w_1}},\bQ_p) \arrow[r,"c_\deR","\sim"']\arrow[d,"\wr"] & \sB_\deR\otimes_{L_{w_1}}\rH^i_\deR(X/L_{w_1}) \arrow[d] \\
			\sB_\deR\otimes_{\bQ_p}\rH^i_\et((X_{L_{w_2}})_{\Lbar_{w_2}},\bQ_p) \arrow[r,"c_\deR","\sim"'] & \sB_\deR\otimes_{L_{w_2}}\rH^i_\deR(X_{L_{w_2}}/L_{w_2})
		\end{tikzcd}
		\end{center}
		commutes for all~$i$, where the vertical maps are the base change maps.
		\item\label{proppart:compatibility_kuenneth} For two smooth proper varieties~$X,Y/L_w$, the square
		\begin{center}
		\begin{tikzcd}[column sep = large]
			\sB_\deR\otimes_{\bQ_p}\rH^\bullet_\et(X_{\Lbar_w},\bQ_p)\otimes_{\bQ_p}\rH^\bullet_\et(Y_{\Lbar_w},\bQ_p) \arrow[r,"c_\deR\otimes c_\deR","\sim"']\arrow[d,"\wr"] & \sB_\deR\otimes_{L_w}\rH^\bullet_\deR(X/L_w)\otimes_{L_w}\rH^\bullet_\deR(Y/L_w) \arrow[d,"\wr"] \\
			\sB_\deR\otimes_{\bQ_p}\rH^\bullet_\et((X\times_{L_w}Y)_{\Lbar_w},\bQ_p) \arrow[r,"c_\deR","\sim"'] & \sB_\deR\otimes_{L_w}\rH^\bullet_\deR(X\times_{L_w}Y/L_w)
		\end{tikzcd}
		\end{center}
		commutes, where the vertical maps are the K\"unneth isomorphisms.
	\end{enumerate}
	\begin{proof}
		The first three points follow straightforwardly from Propositions~\ref{prop:cup_product_compatibility_analytic} and~\ref{prop:base_change_compatibility_analytic}. For the final point, the K\"unneth isomorphism in de Rham cohomology is the unique isomorphism of graded algebras induced by the pullback maps $\rH^\bullet_\deR(X/L_w)\to\rH^\bullet_\deR(X\times_{L_w}Y/L_w)$ and $\rH^\bullet_\deR(Y/L_w)\to\rH^\bullet_\deR(X\otimes_{L_w}Y/L_w)$ induced from the product projections, and similarly for \'etale cohomology. Hence~\eqref{proppart:compatibility_kuenneth} follows from~\eqref{proppart:compatibility_cup_products} and~\eqref{proppart:compatibility_functoriality}.
	\end{proof}
\end{prop}

The second set of compatibilities are those related to Poincar\'e duality. In the statement, we use~$\langle n\rangle$ to denote a shift in filtration by~$n$, i.e.\ if~$V$ is a filtered object then~$V\langle n\rangle$ denotes the filtered object with~$\rF^i(V\langle n\rangle)=\rF^{i+n}\!V$.

\begin{prop}\label{prop:compatibility_ii}
	There is a $G_v$-equivariant filtered $\sB_\deR$-linear isomorphism
	\[
	a\colon \sB_\deR(-1) \xrightarrow\sim \sB_\deR\langle-1\rangle
	\]
	with the following properties.
	\begin{enumerate}\setcounter{enumi}{4}
		\item\label{proppart:compatibility_traces} For any smooth proper geometrically connected variety~$X/L_w$ of dimension~$n$, the square
		\begin{equation}\label{diag:trace_compatibility}
		\begin{tikzcd}
			\sB_\deR\otimes\rH^{2n}_\et(X_{\Lbar_w},\bQ_p) \arrow[r,"c_\deR","\sim"']\arrow[d,"1\otimes\tr_\et","\wr"'] & \sB_\deR\otimes\rH^{2n}_\deR(X/L_w) \arrow[d,"1\otimes\tr_\deR","\wr"'] \\
			\sB_\deR(-n) \arrow[r,"a^{\otimes n}","\sim"'] & \sB_\deR\langle-n\rangle
		\end{tikzcd}
		\end{equation}
		commutes, where~$\tr_\et$ and~$\tr_\deR$ are the trace maps.
		\item\label{proppart:compatibility_poincare} For any smooth proper geometrically connected variety~$X/L_w$ of dimension~$n$, the square
		\begin{center}
		\begin{tikzcd}
			\sB_\deR\otimes\rH^i_\et(X_{\Lbar_w},\bQ_p)\otimes_{\bQ_p}\rH^{2n-i}(X_{\Lbar_w},\bQ_p) \arrow[r,"c_\deR\otimes c_\deR","\sim"']\arrow[d] & \sB_\deR\otimes\rH^i_\deR(X/L_w)\otimes\rH^{2n-i}_\deR(X/L_w) \arrow[d] \\
			\sB_\deR(-n) \arrow[r,"a^{\otimes n}","\sim"'] & \sB_\deR\langle-n\rangle
		\end{tikzcd}
		\end{center}
		commutes for all~$i$, where the vertical maps are the Poincar\'e duality pairings.
		\item\label{proppart:compatibility_cycle_classes} For any smooth proper variety~$X/L_w$ and any codimension~$r$ cycle~$Z$ on~$X$, the \'etale and de Rham cycle classes of~$Z$ are identified with one another under the isomorphism
		\begin{equation}\label{eq:twisted_comparison}
		c_\deR\otimes a^{\otimes-r}\colon \sB_\deR\otimes_{\bQ_p}\rH^{2r}_\et(X_{\Lbar_w},\bQ_p)(r) \xrightarrow\sim \sB_\deR\otimes_{L_w}\rH^{2r}_\deR(X/L_w)\langle r\rangle \,.
		\end{equation}
		\item\label{proppart:compatibility_chern_classes} For any smooth proper variety~$X/L_w$ and vector bundle~$V$ on~$X$, the $r$th \'etale and de Rham Chern classes of~$V$ are identified with one another under the isomorphism~\eqref{eq:twisted_comparison}.
	\end{enumerate}
	\begin{proof}
		First we must define the isomorphism~$a$: we take~$a$ to be the unique isomorphism making~\eqref{diag:trace_compatibility} commute for~$X=\bP^1_{K_v}$.
		
		Now we prove~\eqref{proppart:compatibility_traces}. Note firstly that if~$f\colon X'\to X$ is a generically finite morphism of degree~$d$ between smooth proper geometrically connected $L_w$-varieties of dimension~$n$, then the squares
		\begin{center}
		\begin{tikzcd}
			\rH^{2n}_\et(X_{\Lbar_w},\bQ_p) \arrow[r,"f^*"]\arrow[d,"\tr_\et"] & \rH^{2n}_\et(X'_{\Lbar_w},\bQ_p) \arrow[d,"\tr_\et"] & \rH^{2n}_\deR(X/L_w) \arrow[r,"f^*"]\arrow[d,"\tr_\deR"] & \rH^{2n}_\deR(X'/L_w) \arrow[d,"\tr_\deR"] \\
			\bQ_p(-n) \arrow[r,"d"] & \bQ_p(-n) & L_w\langle-n\rangle \arrow[r,"d"] & L_w\langle-n\rangle
		\end{tikzcd}
		\end{center}
		both commute. Hence~\eqref{diag:trace_compatibility} commutes for~$X$ if and only if it does for~$X'$. Hence it suffices to prove commutativity of~\eqref{diag:trace_compatibility} for one $L_w$-variety of each dimension~$n$: we do this for~$X=(\bP^1_{L_w})^n$. When~$n=1$ and~$L_w=K_v$ this follows by definition of~$a$; when~$n=1$ and~$L_w$ is general this follows by point~\eqref{proppart:compatibility_base_change}. For~$n$ general, the trace maps for~$(\bP^1_{L_w})^n$ are identified, via the K\"unneth isomorphism, with the $n$th tensor power of the trace maps for~$\bP^1_{L_w}$. So commutativity of~\eqref{diag:trace_compatibility} follows from point~\eqref{proppart:compatibility_kuenneth}.
		
		Point~\eqref{proppart:compatibility_poincare} follows immediately from points~\eqref{proppart:compatibility_cup_products} and~\eqref{proppart:compatibility_traces}. For~\eqref{proppart:compatibility_cycle_classes}, by~\eqref{proppart:compatibility_functoriality} and~\eqref{proppart:compatibility_base_change} we need only consider the case that~$X$ is geometrically connected of dimension~$n$ and~$Z$ is geometrically integral. Choose a resolution of singularities~$\tilde Z\to Z$, and let~$\tilde\iota$ denote the composite $\tilde Z\to Z\hookrightarrow X$. The de Rham cycle class of~$Z$ is the unique element of~$\rH^{2r}(X/L_w)\langle r\rangle$ satisfying $\langle\cl_\deR(Z),\xi\rangle=\tr_\deR(\tilde\iota^*\xi)$ for all~$\xi\in\rH^{2n-2r}_\deR(X/L_w)$, where~$\langle\cdot,\cdot\rangle$ is the Poincar\'e duality pairing, and similarly for the \'etale cycle class. So~\eqref{proppart:compatibility_cycle_classes} follows from points~\eqref{proppart:compatibility_functoriality} and~\eqref{proppart:compatibility_poincare}.
		
		For point~\eqref{proppart:compatibility_chern_classes}, Grothendieck's formalism of Chern classes implies that it suffices to prove the result when~$V$ is a line bundle and~$r=1$. But in this case the projection $V\to X$ from the total space of~$V$ induces an isomorphism on \'etale and de Rham cohomology, and the first Chern class of~$V$ is none other than the cycle class of the zero section in~$V$. So we are done by~\eqref{proppart:compatibility_functoriality} and~\eqref{proppart:compatibility_cycle_classes}.
	\end{proof}
\end{prop}

\begin{rmk}
	The construction of the period ring~$\sB_\deR$ provides it with a canonical $G_v$-equivariant map $\bQ_p(1)\to\sB_\deR$ such that the image of any non-zero element spans~$\rF^1\!\sB_\deR$ as a $\sB_\deR^+$-module. Thus there is a canonical choice of an isomorphism $\sB_\deR(-1)\cong\sB_\deR\langle-1\rangle$; we suspect that the isomorphism~$a$ from Proposition~\ref{prop:compatibility_ii} should be this canonical choice of isomorphism, but we do not prove it here.
\end{rmk}

\subsection{Horizontal de Rham local systems}\label{ss:shimizu}

\newcommand{\clash}[1]{{\color{red}#1}}

With Scholze's comparison theorem in hand, Theorem~\ref{thm:period_maps_control_reps} becomes a special case of a general result about de Rham local systems over polydiscs. In this level of generality, the problem was studied by Shimizu~\cite{shimizu:p-adic_monodromy}; we explain carefully here how to extract the result we need from his theory. The same result we need already appears in work by the first author~\cite[Theorem~6.1]{alex:local_constancy}, but we repeat the derivation here for the sake of completeness.

If~$\bE$ is a $\hat\bZ_p$-local system on a smooth rigid-analytic space~$U$ over~$K_v$ and~$y\in U(K_v)$ is a $K_v$-rational point, then one obtains a continuous representation~$\bE_{\bar y}$ of~$G_v$ on a finite $\bZ_p$-module by first pulling back~$\bE$ along $y\colon\Sp(K_v)\to U$ and then taking sections over the object
\[
\varprojlim_{L_w}\Sp(L_w) \in \Sp(K_v)_\proet
\]
where~$L_w$ ranges over finite extensions of~$K_v$ inside~$\Kbar_v$, as in \S\ref{ss:c_dR}. The result we need to extract from Shimizu's theory says that if~$U$ is a closed polydisc, and if~$\bE$ is a de Rham local system whose associated vector bundle has a full basis of horizontal sections over~$U$, then the $(\varphi,N,G_v)$-modules $\sD_\pst(\bE_{\bar y})$ are all canonically isomorphic to one another for~$y\in U(K_v)$. The precise statement is as follows.

\begin{thm}\label{thm:parallel_transport_is_phi-compatible}
	Let~$U$ be a rigid-analytic space over~$K_v$ isomorphic to a closed polydisc or spherical polyannulus\footnote{A spherical polyannulus over~$K_v$ is a rigid-analytic space over~$K_v$ isomorphic to $\Sp(K_v\langle t_1^{\pm1},\dots,t_n^{\pm1}\rangle)$ for some~$n\geq0$. In other words, it is a polyannulus whose inner and outer radii are equal.}, and let~$(\bE,\dE,c)$ be a de Rham pair on~$U$. Suppose that~$\dE$ has a full basis of horizontal sections. Then for every $y_0,y\in\Nbd(K_v)$, there is a unique isomorphism
	\[
	T_{y_0,y}\colon\sD_\pst\left(\bE_{\bar y_0}\right) \xrightarrow\sim \sD_\pst\left(\bE_{\bar y}\right)
	\]
	of $(\varphi,N,G_v)$-modules making the diagram
	\begin{center}
	\begin{tikzcd}
		\Kbar_v\otimes_{\bQ_p^\nr}\sD_\pst\left(\bE_{\bar y_0}\right) \arrow[r,"c_\BO","\sim"']\arrow[d,"1\otimes T_{y_0,y}"',"\wr"] & \Kbar_v\otimes_{K_v}\sD_\deR\left(\bE_{\bar y_0}\right) \arrow[r,"1\otimes c_{y_0}","\sim"'] & \Kbar_v\otimes_{K_v}\dE_{y_0} \arrow[d,"1\otimes T_{y_0,y}^\nabla","\wr"'] \\
		\Kbar_v\otimes_{\bQ_p^\nr}\sD_\pst\left(\bE_{\bar y}\right) \arrow[r,"c_\BO","\sim"'] & \Kbar_v\otimes_{K_v}\sD_\deR\left(\bE_{\bar y}\right) \arrow[r,"1\otimes c_y","\sim"'] & \Kbar_v\otimes_{K_v}\dE_y
	\end{tikzcd}
	\end{center}
	commute.
\end{thm}

\begin{proof}[Proof of Theorem~\ref{thm:period_maps_control_reps}] It is straightforward to deduce Theorem~\ref{thm:period_maps_control_reps} from Theorem~\ref{thm:parallel_transport_is_phi-compatible}. Given a smooth proper morphism $\pi\colon X\to Y$ of smooth $K_v$-varieties, we know by Theorem~\ref{thm:pushforward_of_pairs} applied to the unit de Rham pair~$(\hat\bZ_{p,X^\an},\cO_{X^\an},1)$ on~$X^\an$ that there is an isomorphism
\[
c_\deR\colon\cO\bB_{\deR,Y^\an}\otimes_{\hat\bZ_{p,Y^\an}}\rR^i\!\pi^\an_*\hat\bZ_{p,X^\an} \xrightarrow\sim \cO\bB_{\deR,Y^\an}\otimes_{\cO_{Y^\an}}\rR^i\!\pi^\an_{\deR*}\cO_{X^\an}
\]
making $(\rR^i\!\pi^\an_*\hat\bZ_{p,X^\an},\rR^i\!\pi^\an_{\deR*}\cO_{X^\an},c_\deR)$ into a de Rham pair on~$Y^\an$. As discussed in \S\ref{ss:c_dR}, $\rR^i\!\pi^\an_*\hat\bZ_{p,X^\an}$ and $\rR^i\!\pi^\an_{\deR*}\cO_{X^\an}$ can be identified as the analytifications of the relative \'etale and de Rham cohomologies $\rR^i\!\pi_{\et*}\bZ_{p,X_\et}$ and~$\cH^i_\deR(X/Y)$, respectively, and the fibre of the comparison isomorphism~$c_\deR$ at a $K_v$-point $y\in Y^\an(K_v)=Y(K_v)$ is the comparison isomorphism
\[
c_\deR\colon \sB_\deR\otimes_{\bQ_p}\rH^i_\et(X_{y,\Kbar_v},\bQ_p) \xrightarrow\sim \sB_\deR\otimes_{K_v}\rH^i_\deR(X_y/K_v)
\]
of~\eqref{eq:comparison_iso}. So we obtain Theorem~\ref{thm:period_maps_control_reps} as the special case of Theorem~\ref{thm:parallel_transport_is_phi-compatible}, applied to the de Rham pair 
\[
(\rR^i\!\pi^\an_*\hat\bZ_{p,X^\an}|_U,\rR^i\!\pi^\an_{\deR*}\cO_{X^\an}|_U,c_\deR|_U). \qedhere
\]
\end{proof}

Now we turn to the proof of Theorem~\ref{thm:period_maps_control_reps}, which we want to extract from \cite{shimizu:p-adic_monodromy}. The first step is to reduce to the case of spherical polyannuli.

\begin{lem}
	Suppose that Theorem~\ref{thm:parallel_transport_is_phi-compatible} holds whenever~$U$ is a spherical polyannulus, for all finite extensions~$K_v/\bQ_p$. Then it holds in general.
	\begin{proof}
		Let~$U$ be a closed polydisc, and suppose first that the residue field of~$K_v$ is not~$\bF_2$. Then there is a spherical polyannulus~$U^\circ\subset U$ such that~$y_0,y\in U^\circ(K_v)$. Indeed, after picking coordinates on~$U$ we may write~$y_0=(a_1,\dots,a_n)$ and $y=(b_1,\dots,b_n)$ with all $a_i,b_i\in\cO_v$; the desired inclusion $U^\circ\hookrightarrow U$ is then given by $t_i\mapsto t_i-c_i$ where $c_i\in\cO_v$ is not in the residue disc of~$a_i$ or~$b_i$. Since~$\dE|_{U^\circ}$ again has a full basis of horizontal sections, the result for~$(U,y_0,y)$ follows from the corresponding result for~$(U^\circ,y_0,y)$.
		
		If instead the residue field of~$K_v$ is~$\bF_2$, let us write~$L_{w,r}$ for the unramified extension of~$K_v$ of degree~$r$ inside~$\Kbar_v$, and~$G_{w,r}$ for its absolute Galois group. The result applied to~$(U_{L_{w,r}},y_0,y)$ implies that for each~$r>1$ there exists a unique isomorphism
		\[
		T_{y_0,y,r}\colon\sD_\pst(\bE_{\bar y_0}) \xrightarrow\sim \sD_\pst(\bE_{\bar y})
		\]
		of $(\varphi,N,G_{w,r})$-modules making the claimed diagram commute. Unicity implies that the maps~$T_{y_0,y,r}$ are all equal, to~$T_{y_0,y}$ say, and so $T_{y_0,y}$ commutes with the actions of~$\varphi$ and~$N$, and with the action of~$G_{w,r}$ for all~$r>1$. Since~$G_{w,2}$ and~$G_{w,3}$ generate~$G_v$, it follows that~$T_{y_0,y}$ is $G_v$-equivariant, i.e.\ an isomorphism of~$(\varphi,N,G_v)$-modules. So we are done also in this case.
	\end{proof}
\end{lem}

We now build up to the proof of Theorem~\ref{thm:parallel_transport_is_phi-compatible}. Let~$U=\bT^n\colonequals\Sp(R_v[p^{-1}])$ be the spherical polyannulus with $R_v\colonequals\cO_v\langle t_1^{\pm1},t_2^{\pm1},\dots,t_n^{\pm1}\rangle$. As in~\cite[\S3]{shimizu:p-adic_monodromy}, fix an algebraic closure~$\overline{\Frac(R_v)}$ of the fraction field of~$R_v$, and let~$\Rbar_v\subseteq\overline{\Frac(R_v)}$ be the union of all finite $R_v$-subalgebras~$R_v'$ of~$\overline{\Frac(R_v)}$ for which~$R_v'[p^{-1}]$ is \'etale over~$R_v[p^{-1}]$. We write
\[
G_{R_v}\colonequals\Aut(\Rbar_v[p^{-1}]/R_v[p^{-1}]) \,.
\]
This group is canonically isomorphic to the \'etale fundamental group of~$\Spec(R_v[p^{-1}])$ based at the geometric point determined by~$\overline{\Frac(R_v)}$ \cite[Exp.~I, Proposition~10.2]{sga1}, and hence carries a natural profinite topology. We always suppose that we chose~$\overline{\Frac(R_v)}$ to contain~$\Kbar_v$, so that~$\Kbar_v\subseteq\Rbar_v[p^{-1}]$ and we have a restriction homomorphism
\begin{equation}\label{eq:galois_restriction_map}
G_{R_v}\twoheadrightarrow G_v \,.
\end{equation}
If~$L_w$ is a finite extension of~$K_v$ contained in~$\Kbar_v$, we write~$G_{R_w}\subseteq G_{R_v}$ for the preimage of~$G_w$ under the restriction map.
\smallskip

Shimizu's theory isolates a class of $G_{R_v}$-representations, generalising the class of potentially semistable representations of~$G_v$ (which is recovered by setting~$R_v=\cO_v$). For this, he introduces a certain \emph{horizontal semistable period ring}~$\sB_\st^\nabla(R_v)$, which is a $\bQ_p^\nr$-algebra and comes with a semilinear crystalline Frobenius endomorphism~$\varphi$ and monodromy operator~$N$ \cite[Definition~4.7]{shimizu:p-adic_monodromy}. Moreover, the construction of~$\sB_\st^\nabla(R_v)$ is functorial in~$\Rbar_v$, and in particular carries a natural action of~$G_{R_v}$ induced by the tautological action on~$\Rbar_v$.  For any finite extension~$L_w/K_v$ contained in~$\Kbar_v$, we have that $(\sB_\st^\nabla)^{G_{R_w}}=L_{w,0}=L_w\cap\bQ_p^\nr$  \cite[Corollary~4.10]{shimizu:p-adic_monodromy}. This period ring allows one to identify the desired class of $G_{R_v}$-representations.

\begin{defi}[Potentially horizontal semistable representations {\cite[Definition~4.15]{shimizu:p-adic_monodromy}}]
	If~$E$ is a continuous representation of~$G_{R_v}$ on a finite free~$\bZ_p$-module, then we write
	\[
	\sD_\pst^\nabla(E)\colonequals\varinjlim_w(\sB_\st^\nabla(R_v)\otimes_{\bZ_p}E)^{G_{R_w}}
	\]
	where the colimit is taken over finite extensions $L_w/K_v$ contained in~$\Kbar_v$. The natural action of $G_{R_v}$ on~$\sD_\pst^\nabla(E)$ factors through~$G_v$ and has open point-stabilisers: together with the crystalline Frobenius~$\varphi$ and monodromy operator~$N$ induced from those on~$\sB_\st^\nabla(R_v)$, this makes $\sD_\pst^\nabla(E)$ into a discrete $(\varphi,N,G_v)$-module. We always have the inequality $\dim_{\bQ_p^\nr}\sD_\pst^\nabla(E)\leq\dim_{\bQ_p}E$ \cite[Lemma~4.12]{shimizu:p-adic_monodromy}, and we say that~$E$ is \emph{potentially horizontal semistable} just when equality holds. Equivalently, $E$ is potentially horizontal semistable just when the evident map
	\[
	\alpha_\pst^\nabla\colon \sB_\st^\nabla(R_v)\otimes_{\bQ_p^\nr}\sD_\pst^\nabla(E) \to \sB_\st^\nabla(R_v)\otimes_{\bZ_p}E
	\]
	is a $\sB_\st^\nabla(R_v)$-linear isomorphism.
\end{defi}

In the particular case that~$R_v=\cO_v$ (and $\Rbar_v=\cO_{\Kbar_v}$), we have that~$\sB_\st^\nabla(\cO_v)=\sB_\st$ is the usual semistable period ring of Fontaine. In general, if~$y\in U(K_v)$ is a $K_v$-point, we can choose an extension of the pullback map $y^*\colon R_v[p^{-1}]\to K_v$ to a $\Kbar_v$-algebra homomorphism $\tilde y^*\colon \Rbar_v[p^{-1}]\to\Kbar_v$. The map~$\tilde y^*$ determines a continuous homomorphism
\[
s_{\tilde y}\colon G_v \to G_{R_v} \,,
\]
identifying~$G_v$ as the setwise stabiliser of the kernel of~$\tilde y^*$. For a continuous representation~$E$ of~$G_{R_v}$ on a finite free~$\bZ_p$-module, we write~$E_{\tilde y}$ for the $G_v$-representation given by~$E$ with the $G_v$-action given by restriction of the $G_{R_v}$-action along~$s_{\tilde y}$.

Since the construction of~$\sB_\st^\nabla(R_v)$ is functorial in~$\Rbar_v$, there is an induced morphism
\[
\tilde y^*\colon\sB_\st^\nabla(R_v) \to \sB_\st
\]
of~$\bQ_p^\nr$-algebras, compatible with crystalline Frobenii and monodromy operators. It is also $G_v$-equivariant where $G_v$ acts on~$\sB_\st^\nabla(R_v)$ via restriction along~$s_{\tilde y}$. So for a continuous representation~$E$ of~$G_{R_v}$ on a finitely generated~$\bZ_p$-module there is then a natural morphism
\[
\rho_{\tilde y}\colon\sD_\pst^\nabla(E) \to \sD_\pst(E_{\tilde y})
\]
of $(\varphi,N,G_v)$-modules given by taking invariants in $\tilde y^*\otimes1\colon\sB_\st^\nabla(R_v)\otimes_{\bZ_p}E \to \sB_\st\otimes_{\bZ_p}E_{\tilde y}$.

\begin{lem}
Suppose that~$E$ is potentially horizontal semistable. Then~$E_{\tilde y}$ is potentially semistable as a representation of~$G_v$ and~$\rho_{\tilde y}$ is an isomorphism.
\begin{proof}
	Special case of~\cite[Lemma~4.4]{shimizu:p-adic_monodromy}.
\end{proof}
\end{lem}

\begin{cor}\label{cor:pst_parallel_transport}
	Let~$y_0,y\in U(K_v)$ and choose lifts $y_0^*,y^*\colon\Rbar_v[p^{-1}]\to\Kbar_v$. Then for any potentially horizontal semistable representation~$E$ of~$G_{R_v}$ there is a canonical isomorphism
	\[
	T_{y_0,y}\colon\sD_\pst(E_{\tilde y_0}) \xrightarrow\sim \sD_\pst(E_{\tilde y})
	\]
	of $(\varphi,N,G_v)$-modules, given by~$T_{y_0,y}\colonequals \rho_{\tilde y}\circ\rho_{\tilde y_0}^{-1}$.
\end{cor}

Now we want to translate between the languages of $\hat\bZ_p$-local systems on~$U$ and representations of~$G_{R_v}$. For this, let us write
\[
\Ubar\colonequals\varprojlim_{R_v'}\Sp(R_v'[p^{-1}]) \,,
\]
where $R_v'$ runs over all finite $R_v$-algebras in~$\overline{\Frac(R_v)}$ such that $R_v'[p^{-1}]$ is \'etale over~$R_v[p^{-1}]$. Thus~$\Ubar$ is an object of the pro-\'etale site of~$U$, on which~$G_{R_v}$ acts naturally from the right.

\begin{lem}\label{lem:coverings_functor}
	The category~$\Cov_U^\alg$ of finite \'etale coverings of~$U$ is a Galois category \cite[Exp.~V, D\'efinition~5.1]{sga1}, and the functor
	\begin{align*}
		F_{\Ubar}\colon \Cov_U^\alg &\to \{\text{finite sets}\} \\
		V &\mapsto \Hom_U(\Ubar,V)
	\end{align*}
	is a fibre functor.
	
	In particular, the functor from $\hat\bZ_p$-local systems on~$U$ to continuous representations of~$G_{R_v}$ on finitely generated~$\bZ_p$-modules given by~$\bE\mapsto\rH^0(\Ubar,\bE)$ is an equivalence, where the~$G_{R_v}$-action on~$\rH^0(\Ubar,\bE)$ is the one induced from the action on~$\Ubar$.
	\begin{proof}
		It follows from \cite[Definition~4.5.7~\&~Proposition~8.1.1]{fresnel-van_der_put:rigid_analytic_geometry} that the category of finite \'etale coverings of~$U$ is opposite to the category of finite \'etale $R_v[p^{-1}]$-algebras (see the proof of \cite[Lemma~6.4]{alex:local_constancy} for a careful argument). That~$F_{\Ubar}$ is a fibre functor follows from \cite[Exp.~V, Proposition~5.6]{sga1}. The final part follows since the category of~$\hat\bZ_p$-local systems on~$U$ in the pro-\'etale topology is equivalent to the category of~$\bZ_p$-local systems in the \'etale topology \cite[Proposition~8.2]{scholze:relative_p-adic_hodge_theory}.
	\end{proof}
\end{lem}

\begin{rmk}\label{rmk:galois_action_on_fibres}
	If~$y\in U(K_v)$ is a $K_v$-rational point, then a lifting of~$y^*\colon R_v[p^{-1}]\to K_v$ to some $\tilde y^*\colon\Rbar_v[p^{-1}] \to \Kbar_v$ is equivalent to a choice of point~$\tilde y\in\Ubar_{\bar y}$ in the fibre of~$\Ubar$ over~$\bar y$. Hence by the Yoneda Lemma $\tilde y$ determines a natural isomorphism $\gamma_{\tilde y}\colon F_{\Ubar}\xrightarrow\sim F_{\bar y}$ of fibre functors on~$\Cov_U^\alg$. This isomorphism~$\gamma_{\tilde y}$ gives a $G_v$-equivariant isomorphism
	\[
	\rH^0(\Ubar,\bE)_{\tilde y} \cong \bE_{\bar y}
	\]
	for all $\hat\bZ_p$-local systems~$\bE$ on~$U$.
\end{rmk}

\begin{rmk}
	Lemma~\ref{lem:coverings_functor} implies that the group~$G_{R_v}$ is isomorphic to the algebraic fundamental group of~$U$ in the sense of de Jong \cite[p.~94]{de_jong:rigid_pi_1}, albeit based at a fibre functor that does not canonically come from a geometric point of the Berkovich space associated to~$U$.
\end{rmk}

In the proof of Theorem~\ref{thm:parallel_transport_is_phi-compatible} which follows, we adopt the notation of \cite[p.~47]{shimizu:p-adic_monodromy} in writing
\[
\sB_\deR(R_v)\colonequals\rH^0(\Ubar,\cO\bB_{\deR,U}) \,.
\]
This is a filtered $\Rbar_v[p^{-1}]$-algebra endowed with an action of~$G_{R_v}$ extending the tautological action on~$\Rbar_v[p^{-1}]$, and with a flat connection
\[
\nabla\colon \sB_\deR(R_v) \to \Omega^{1,f}_{R_v[p^{-1}]/K_v}\otimes_{R_v[p^{-1}]}\sB_\deR(R_v)
\]
satisfying the Leibniz rule with respect to the derivation on~$R_v[p^{-1}]$. Here~$\Omega^{1,f}_{R_v[p^{-1}]/K_v}=\rH^0(U,\Omega^1_{U/K_v})$ is the module of finite differentials. The horizontal semistable period ring~$\sB_\st^\nabla(R_v)$ is then a subring of~$\sB_\deR(R_v)$ contained in the kernel of the connection~$\nabla$.

\begin{proof}[Proof of Theorem~\ref{thm:parallel_transport_is_phi-compatible} for~$U$ a spherical polyannulus]
	Let~$E=\rH^0(\Ubar,\bE)$ be the $G_{R_v}$-representation corresponding to the local system~$\bE$. Since~$\dE$ has a full basis of horizontal sections, we have that~$E$ is potentially horizontal semistable~\cite[Lemma~8.9]{shimizu:p-adic_monodromy}. We fix lifts~$\tilde y_0^*,\tilde y^*\colon\Rbar_v[p^{-1}]\to\Kbar_v$ of~$y_0^*,y^*$ as usual, so there are canonical $G_v$-equivariant identifications~$E_{\tilde y_0}\cong\bE_{\bar y_0}$ and~$E_{\tilde y}\cong\bE_{\bar y}$ by Remark~\ref{rmk:galois_action_on_fibres}.
	
	We will show that, under these identifications, the map
	\[
	T_{y_0,y} \colonequals \rho_{\tilde y}\circ \rho_{\tilde y_0}^{-1} \colon \sD_\pst(E_{\tilde y_0}) \xrightarrow\sim \sD_\pst(E_{\tilde y})
	\]
	from Corollary~\ref{cor:pst_parallel_transport} makes the rectangle in Theorem~\ref{thm:parallel_transport_is_phi-compatible} commute; it is automatically the unique such map. For this, we write~$\dE_U\colonequals\rH^0(U,\dE)$ for the projective $R_v[p^{-1}]$-module associated to~$\dE$. We have $G_{R_v}$-equivariant isomorphisms
	\begin{equation}\label{eq:trivialisation_over_annulus}\tag{$\ast$}
		\sB_\deR(R_v)\otimes_{\bQ_p^\nr}\sD_\pst^\nabla(E) \xrightarrow\sim \sB_\deR(R_v)\otimes_{\bZ_p}E \xrightarrow[\sim]{c} \sB_\deR(R_v)\otimes_{R_v[p^{-1}]}\dE_U
	\end{equation}
	of $\sB_\deR(R_v)$-modules with flat connection, the first of which is the base-change of the natural map $\sB_\st^\nabla(R_v)\otimes_{\bQ_p^\nr}\sD_\pst^\nabla(E)\xrightarrow\sim\sB_\st^\nabla(R_v)\otimes_{\bZ_p}E$, and the second of which is the sections over~$\Ubar$ of the isomorphism~$c\colon\cO\bB_{\deR,U}\otimes_{\hat\bZ_{p,U}}\bE\xrightarrow\sim\cO\bB_{\deR,U}\otimes_{\cO_U}\dE$.
	
	Base-changing~\eqref{eq:trivialisation_over_annulus} along~$\tilde y_0^*\colon\sB_\deR(R_v)\to\sB_\deR$ yields a commuting diagram
	\begin{equation}\label{diag:trivialisation_over_annulus_fibres}\tag{$\ast\ast$}
		\begin{tikzcd}
			\sB_\deR(R_v)\otimes_{\bQ_p^\nr}\sD_\pst^\nabla(E) \arrow[r,"\sim"]\arrow[d,"\tilde y_0^*\otimes\rho_{\tilde y_0}"] & \sB_\deR(R_v)\otimes_{\bZ_p}E \arrow[r,"c","\sim"']\arrow[d,"\tilde y_0^*\otimes1"] & \sB_\deR(R_v)\otimes_{R_v[p^{-1}]}\dE_U \arrow[d,"\tilde y_0^*"] \\
			\sB_\deR\otimes_{\bQ_p^\nr}\sD_\pst(E_{\tilde y_0}) \arrow[r,"\sim"] & \sB_\deR\otimes_{\bZ_p}E_{\tilde y_0} \arrow[r,"c_{y_0}","\sim"'] & \sB_\deR\otimes_{K_v}\dE_{y_0}
		\end{tikzcd}
	\end{equation}
	of $G_v$-equivariant $\sB_\deR$-modules. Using this, we claim that the rectangle
	\begin{equation}\label{diag:parallel_transport_over_annulus}\tag{$\dag$}
		\begin{tikzcd}
			\sB_\deR(R_v)\otimes_{\bQ_p^\nr}\sD_\pst^\nabla(E) \arrow[r,"\sim"]\arrow[d,"1\otimes\rho_{\tilde y_0}","\wr"'] & \sB_\deR(R_v)\otimes_{\bZ_p}E \arrow[r,"c","\sim"'] & \sB_\deR(R_v)\otimes_{R_v[p^{-1}]}\dE_U \\
			\sB_\deR(R_v)\otimes_{\bQ_p^\nr}\sD_\pst(E_{\tilde y_0}) \arrow[r,"\sim"] & \sB_\deR(R_v)\otimes_{\bZ_p}E_{\tilde y_0} \arrow[r,"c_{y_0}","\sim"'] & \sB_\deR(R_v)\otimes_{K_v}\dE_{y_0} \arrow[u,"T_{y_0}^\nabla"',"\wr"]
		\end{tikzcd}
	\end{equation}
	commutes. Since all the maps in~\eqref{diag:parallel_transport_over_annulus} are $G_{R_v}$-equivariant isomorphisms of~$\sB_\deR(R_v)$-modules with connection and $\sB_\deR(R_v)^{G_{R_v},\nabla=0}=K_v$ \cite[Proposition~4.9]{shimizu:p-adic_monodromy}, it is certainly true that~\eqref{diag:parallel_transport_over_annulus} commutes up to an element of~$\GL_n(K_v)$. So it suffices to check that~\eqref{diag:parallel_transport_over_annulus} commutes after base-changing along~$\tilde y_0^*\colon\sB_\deR(R_v)\to\sB_\deR$, which follows from commutativity of~\eqref{diag:trivialisation_over_annulus_fibres}.
	
	Base-changing~\eqref{diag:parallel_transport_over_annulus} along~$\tilde y^*$ then yields the desired result.
\end{proof}
\section{Abelian-by-finite families}\label{s:abelian-by-finite}

In the method of Lawrence and Venkatesh, one studies smooth projective families $X\to Y$ of a particular form, known as \emph{abelian-by-finite families}.

\begin{defi}[{\cite[Definition~5.1]{LVinventiones}}]
Let~$Y$ be a scheme. An \emph{abelian-by-finite family} over~$Y$ 
\[
X\to Y'\to Y
\]
consists of a surjective finite \'etale covering $\pif\colon Y'\to Y$ and 
a polarised abelian scheme $\pia\colon X\to Y'$. We usually suppress the polarisation from the notation; when we need to refer to it, we denote it~$\lambda$.
\end{defi}

In this section, we recall the basic theory of abelian-by-finite families, primarily that the \'etale and de Rham cohomology of the fibres of $X\to Y$ carries extra structures arising from the abelian-by-finite structure. Using this, we show that the period map associated to an abelian-by-finite family takes on a particular form.

\subsection{Cohomology of fibres of abelian-by-finite families}

To begin with, we introduce the extra structures on the cohomology of the fibres of an abelian-by-finite family. It suffices to discuss this in the case that $Y=\Spec(K)$ where~$K$ is a field, assumed for simplicity of characteristic~$0$ (we will apply this with $K$ a number field, or $K=K_v$ a finite extension of~$\bQ_p$). So~$X$ is then a disjoint union of polarised abelian varieties defined over finite extensions of~$K$.

\subsubsection{\'Etale cohomology}

If we fix an algebraic closure~$\Kbar$ of~$K$, then the Leray spectral sequence implies that the cohomology algebra $\rH^\bullet_\et\colonequals\rH^\bullet_\et(X_{\Kbar},\bQ_p)$ is the sections over $Y'_{\Kbar}$ of the derived \'etale pushforward $\rR^\bullet\!\pia_{\et*}\bQ_{p,X_{\Kbar}}$ to~$Y'_{\Kbar}$. It follows from the usual calculation of the cohomology of abelian varieties that the cup product induces canonical isomorphisms
\[
\bigwedgek\nolimits_{\rH^0_\et}\rH^1_\et \xrightarrow\sim \rH^k_\et
\]
of $\rH^0_\et$-modules for all~$k\geq0$, where the $\rH^0_\et$-module structure on either side is the one coming from cup product. Concretely, this is just saying the following. If we choose, for each closed point~$y'\in|Y'|$, a $K$-embedding $K(y')\hookrightarrow\Kbar$, then the absolute Galois group~$G_{K(y')}$ of~$K(y')$ is identified as an open subgroup of~$G_K$. Since~$X$ is a disjoint union of abelian schemes over the fields~$K(y')$, we have
\begin{equation}\label{eq:etale_decomposition}
\rH^\bullet_\et(X_{\Kbar},\bQ_p) = \prod_{y'\in|Y'|}\Ind_{G_{K(y')}}^{G_K}\rH^\bullet_\et(X_{\bar y'},\bQ_p)=\prod_{y'\in|Y'|}\Ind_{G_{K(y')}}^{G_K}\bigwedge\nolimits^{\!\!\bullet}\rH^1_\et(X_{\bar y'},\bQ_p)
\end{equation}
as $\bQ_p$-algebras with $G_K$-action, where $X_{\bar y'}$ denotes the fibre of~$X$ over the $\Kbar$-point of~$Y'$ determined by the embedding $K(y')\hookrightarrow\Kbar$.
\smallskip

Moreover, the polarisation~$\lambda$ on~$X$ induces a further structure on the \'etale cohomology, in the form of a pairing on $\rH^1_\et$. This is constructed as follows. Let $c_1^\et(\lambda)\in\rH^2_\et(1)$ denote the first \'etale Chern class of a line bundle on~$X_{\Kbar}$ representing the polarisation~$\lambda$ (this is independent of the choice of line bundle). Equivalently, if~$\dL\colonequals(1,\lambda)^*\dP$ with $\dP$ the Poincar\'e line bundle on~$X\times_{Y'}X^\vee$, then $c_1^\et(\lambda)=\frac12 c_1^\et(\dL)$. This latter construction makes it clear that in fact $c_1^\et(\lambda)\in\rH^2_\et(1)^{G_K}$.

Since $\rH^2_\et=\bigwedgesquare_{\rH^0_\et}\rH^1_\et$, we can view $c_1^\et(\lambda)$ as an element of $\Hom_{\rH^0_\et}(\bigwedgesquare_{\rH^0_\et}\rH_1^\et,\rH^0_\et(1))^{G_K}$, where 
\[
\rH_1^\et\colonequals\Hom_{\rH^0_\et}(\rH^1_\et,\rH^0_\et).
\]
In other words, we can view $c_1^\et(\lambda)$ as a $G_K$-equivariant $\rH^0_\et$-linear pairing on $\rH_1^\et$.

\begin{lem}\label{lem:etale_pairing_perfect}
	The $\rH^0_\et$-linear pairing
	\[
	\check\omega_\lambda^\et\colon\bigwedgesquare\nolimits_{\rH^0_\et}\rH_1^\et\to\rH^0_\et(1)
	\]
	determined by~$c_1^\et(\lambda)$ is perfect, i.e.\ determines an $\rH^0_\et$-linear $G_K$-equivariant isomorphism
	\[
	\rH_1^\et \xrightarrow\sim \Hom_{\rH^0_\et}(\rH_1^\et,\rH^0_\et(1)) = \rH^1_\et(1) \,.
	\]
	\begin{proof}
		It suffices to prove this in the case that $K=\Kbar$ is algebraically closed. In this case, $X$ is a disjoint union of polarised abelian varieties $(X_{y'},\lambda_{y'})$ indexed by closed points~$y'$ of~$Y'$. It follows that $c_1^\et(\lambda)\in\rH^2_\et(X,\bQ_p)=\bigoplus_{y'\in|Y'|}\rH^2_\et(X_{y'},\bQ_p)$ is the element whose $y'$th component is $c_1^\et(\lambda_{y'})$, so the pairing $\check\omega_\lambda^\et$ on $\rH_1^\et=\bigoplus_{y'\in|Y'|}\rH_1^\et(X_{y'},\bQ_p)$ is the orthogonal direct sum of the pairings $\check\omega_{\lambda_{y'}}^\et$ on $\rH_1^\et(X_{y'},\bQ_p)$. It thus suffices to treat only the case that $\pif\colon Y'\to Y$ is the identity, so that $\pia\colon X\to Y$ is a polarised abelian variety.
		
		In this case, $\rH_1^\et=\rH^1_\et(X,\bQ_p)^*$ is isomorphic to the $\bQ_p$-linear Tate module of~$X$, and under this identification, the pairing $\check\omega_{\lambda}^\et$ is identified up to sign with the Weil pairing associated to~$\lambda$, see the proof of \cite[Lemma~2.6]{orr-skorobogatov-zarhin:k3_uniformity}. The Weil pairing is known to be perfect.
	\end{proof}
\end{lem}

Now the pairing~$\check\omega_\lambda^\et$ induced by~$c_1^\et(\lambda)$, being perfect, induces a dual pairing on $\rH^1_\et(1)$, and hence by Tate twisting, we obtain a $G_K$-equivariant $\rH^0_\et$-linear perfect pairing
\[
\omega_\lambda^\et\colon\bigwedgesquare\nolimits_{\rH^0_\et}\rH^1_\et \to \rH^0_\et(-1) \,.
\]
Under the identification $\rH^1_\et=\bigoplus_{y'\in|Y'|}\Ind_{G_{K(y')}}^{G_K}\rH^1_\et(X_{\bar y'},\bQ_p)$, the pairing~$\omega_\lambda^\et$ is just the orthogonal direct sum of the inductions of the pairings on each $\rH^1_\et(X_{\bar y'},\bQ_p)$ induced from the polarisation on the abelian variety~$X_{\bar y'}$.
\smallskip

In summary, we have seen that the triple
\[
(\rH^1_\et,\rH^0_\et(-1),\omega_\lambda^\et)
\]
is an example of a symplectic $\rH^0_\et$-module in the following sense.

\begin{defi}
	Let~$A$ be a finite-dimensional $\bQ_p$-algebra endowed with a continuous action of~$G_K$ compatible with the algebra structure. A \emph{symplectic $A$-module} is a triple~$V=(V,L,\omega)$ consisting of:
	\begin{itemize}
		\item a finite locally free $A$-module~$V$ and a free rank $1$ $A$-module~$L$, each endowed with a continuous action of~$G_K$ compatible with the action on~$A$; and
		\item a $G_K$-equivariant $A$-linear perfect pairing
		\[
		\bigwedgesquare\nolimits_AV\to L \,, 
		\]
		meaning that the induced map $V\to\Hom_A(V,L)$ is an isomorphism.
	\end{itemize}
	Symplectic $A$-modules form a category, where a morphism $f\colon(V,L,\omega)\to(V',L',\omega')$ consists of a pair of $G_K$-equivariant $A$-linear maps $f_V\colon V\to V'$ and $f_L\colon L\to L'$ compatible with the pairings~$\omega$ and~$\omega'$.
\end{defi}

Later, we will want to also consider symplectic modules over varying algebras~$A$, for which we adopt the following terminology.

\begin{defi}
	A \emph{symplectic pair} in the category of $G_K$-representations is a pair 
	\[
	P=(A,V)
	\]
	consisting of a finite-dimensional $\bQ_p$-algebra~$A$ endowed with a continuous action of $G_K$ and a symplectic $A$-module~$V=(V,L,\omega)$. Symplectic pairs form a category 
	\[
	\SPair(\Rep_{\bQ_p}(G_K)),
	\]
	whose morphisms $f\colon(A,(V,L,\omega)) \to (A',(V',L',\omega'))$ are triples of $G_K$-equivariant maps $f_A\colon A\to A'$, $f_V\colon V\to V'$ and $f_L\colon L\to L'$ compatible with all relevant structures (algebra structures, module structures, pairing).
\end{defi}

\begin{ex}
	If~$X\to Y'\to Y$ is an abelian-by-finite family over a base scheme~$Y$ and~$y\in Y(K)$ is a $K$-rational point, then the fibre $X_y\to Y'_y\to \Spec(K)$ is an abelian-by-finite family over~$\Spec(K)$. So it follows from the above discussion that
	\[
	\rH^{\leq 1}_{\et}(X_{y,\Kbar},\bQ_p) \colonequals  (\rH^0_\et(X_{y,\Kbar},\bQ_p),\rH^1_\et(X_{y,\Kbar},\bQ_p))
	\]
	is a symplectic pair in the category of $G_K$-representations, the symplectic pairing on $\rH^1_\et=\rH^1_\et(X_{y,\Kbar},\bQ_p)$ being the pairing discussed above arising from the polarisation. We will see later in \S\ref{ss:relative_etale_cohomology} that these symplectic pairs for varying~$y$ are interpolated by a single symplectic pair in the category of~$\bQ_p$-local systems on~$Y_\et$, given by the relative \'etale cohomology of~$X\to Y$.
\end{ex}

One can equally make sense of symplectic modules and symplectic pairs in other appropriate categories.

\subsubsection{De Rham cohomology}

A similar story holds for the de Rham cohomology of the abelian-by-finite family $X\to Y'\to Y=\Spec(K)$ for any field~$K$ of characteristic~$0$. The cohomology algebra $\rH^\bullet_\deR\colonequals\rH^\bullet_\deR(X/K)$ is canonically isomorphic to the graded-commutative algebra $\bigwedgebullet_{\rH^0_\deR}\!\rH^1_\deR$. Again, this can be described concretely: the decomposition of~$X$ into its connected components gives an isomorphism
\begin{equation}\label{eq:de_rham_decomposition}
\rH^\bullet_\deR \cong \prod_{y'\in|Y'|}\rH^\bullet_\deR(X_{y'}/K(y')) \cong \prod_{y'\in|Y'|}\bigwedgebullet\nolimits_{K(y')}\rH^1_\deR(X_{y'}/K(y'))
\end{equation}
of $K$-algebras, where the $K(y')$-vector spaces $\rH^\bullet_\deR(X_{y'}/K(y'))$ are viewed as $K$-vector spaces in the usual way.

The polarisation~$\lambda$ on~$X$ induces a pairing on $\rH^1_\deR$, which is constructed analogously to the pairing on \'etale cohomology. That is, we define the first de Rham Chern class $c_1^\deR(\lambda)\in\rF^1\!\rH^2_\deR$ of the polarisation~$\lambda$ by $c_1^\deR(\lambda)\colonequals\frac12c_1^\deR(\dL)$ where $\dL=(1,\lambda)^*\dP$ with $\dP$ the Poincar\'e line bundle on~$X\times_{Y'}X^\vee$. Via the identification $\rH^2_\deR\cong\bigwedgesquare_{\rH^0_\deR}\!\rH^1_\deR$, the class $c_1^\deR(\lambda)$ can be thought of as an element
\[
\check\omega_\lambda^\deR \in \rF^0\!\Hom_{\rH^0_\deR}(\bigwedgesquare\nolimits_{\rH^0_\deR}\!\rH_1^\deR,\rH^0_\deR\langle1\rangle),
\]
where $\rH_1^\deR\colonequals\Hom_{\rH^0_\deR}(\rH^1_\deR,\rH^0_\deR)$ and $(-)\langle1\rangle$ denotes a shift in filtration: $\rF^i(V\langle j\rangle)\colonequals\rF^{i+j}\!V$. Analogously to the \'etale case, we have the following.

\begin{lem}\label{lem:de_rham_pairing_perfect}
	The $\rH^0_\deR$-linear pairing
	\[
	\check\omega_\lambda^\deR\colon\bigwedgesquare\nolimits_{\rH^0_\deR}\rH_1^\deR\to\rH^0_\deR\langle1\rangle
	\]
	determined by~$c_1^\deR(\lambda)$ is perfect, i.e.\ determines an $\rH^0_\deR$-linear filtered isomorphism
	\[
	\rH_1^\deR \xrightarrow\sim \Hom_{\rH^0_\deR}(\rH_1^\deR,\rH^0_\deR\langle1\rangle) = \rH^1_\deR\langle1\rangle \,.
	\]
	\begin{proof}
		The easiest way to prove this is by comparing with \'etale cohomology. That is, via a Lefschetz argument, it suffices to deal only with the case that~$K=K_v$ is a finite extension of~$\bQ_p$. As we shall explain shortly, the isomorphism $\sD_\deR(\rH_1^\et)\cong\rH_1^\deR$ induced from the comparison isomorphisms $\sD_\deR(\rH^k_\et)\cong\rH^k_\deR$ for~$k=0,1$ carries the \'etale Chern class $c_1^\et(\lambda)$ to $c_1^\deR(\lambda)$. Hence perfectness follows from the corresponding statement for \'etale cohomology (Lemma~\ref{lem:etale_pairing_perfect}).
	\end{proof}
\end{lem}

Again, the fact that $\check\omega_\lambda^\deR$ is perfect means that it induces a twisted dual pairing
\[
\omega_\lambda^\deR\colon\bigwedgesquare\nolimits_{\rH^0_\deR}\rH^1_\deR\to\rH^0_\deR\langle-1\rangle \,,
\]
which is a filtered $\rH^0_\deR$-linear perfect pairing, equal to the orthogonal direct sum of the pairings on each $\rH^1_\deR(X_{y'}/K(y'))$ induced from the polarisation on the abelian variety~$X_{y'}$. As in the \'etale case, this can be conveniently summarised by saying that  the triple
\[
(\rH^1_\deR,\rH^0_\deR\langle-1\rangle,\omega_\lambda^\deR)
\]
is a symplectic $\rH^0_\deR$-module in the category of filtered vector spaces.

This means something very concrete with regards to the Hodge filtration. Since the Hodge filtration on $\rH^0_\deR$ is supported in degree~$0$ and the Hodge filtration on $\rH^1_\deR$ is supported in degrees~$0$ and~$1$, the only interesting step of the Hodge filtration is $\rF^1\!\rH^1_\deR$. The fact that $\rH^1_\deR$ is a symplectic $\rH^0_\deR$-module says exactly that $\rF^1\!\rH^1_\deR$ is a \emph{Lagrangian} $\rH^0_\deR$-submodule, i.e.\ is a $\rH^0_\deR$-submodule on which~$\omega_\lambda^\deR$ vanishes, maximal with this property. In other words, with respect to the decomposition 
\[
\rH^1_\deR\cong\prod_{y'\in|Y'|}\rH^1_\deR(X_{y'}/K(y')),
\]
the subspace $\rF^1\!\rH^1_\deR$ is the product of Lagrangian $K(y')$-subspaces $\rF^1\!\rH^1_\deR(X_{y'}/K(y'))$ for each $y'\in|Y'|$.

\subsubsection{The comparison isomorphism}\label{sss:comparison_of_pairs}

In the particular case that~$K=K_v$ is a finite extension of~$\bQ_p$, the \'etale and de Rham cohomology of an abelian-by-finite family $X\to Y'\to\Spec(K_v)$ are related by the comparison isomorphisms $c_\deR\colon\sD_\deR(\rH^\bullet_\et)\xrightarrow\sim\rH^\bullet_\deR$. Being an isomorphism of algebras with respect to cup product, $c_\deR$ is an isomorphism of algebras on~$\rH^0$, and an isomorphism of $\rH^0$-modules on~$\rH^1$. The fact that $c_\deR$ preserves Chern classes ensures that the isomorphism $c_\deR\colon\sD_\deR(\rH^2_\et(1))\xrightarrow\sim\rH^2_\deR\langle1\rangle$ takes $c_1^\et(\lambda)=\frac12c_1^\et(\dL)$ to $c_1^\deR(\lambda)=\frac12c_1^\deR(\dL)$, where $\dL=(1,\lambda)^*\dP$ with $\dP$ the Poincar\'e line bundle as above. This ensures that the pairings $\check\omega_\lambda^\et$ and $\check\omega_\lambda^\deR$ are identified under~$c_\deR$ (so in particular the latter pairing is perfect as per Lemma~\ref{lem:de_rham_pairing_perfect}), and hence so too are the twisted dual pairings $\omega_\lambda^\et$ and~$\omega_\lambda^\deR$.

This can all be succinctly summarised on the level of the $p$-adic Hodge cohomology groups 
$\rH^\bullet_\pH\colonequals\rH^\bullet_\pH(X/K_v)$, see   Example~\ref{ex:p-adic_hodge_cohomology} for the definition. The triple
\[
(\rH^1_\pH,\rH^0_\pH\langle-1\rangle,\omega_\lambda^\pH)
\]
is a symplectic $\rH^0_\pH$-module in the category of filtered discrete $(\varphi,N,G_v)$-modules, where 
\[
\omega_\lambda^\pH\colon\bigwedgesquare\nolimits_{\rH^0_\pH}\rH^1_\pH\to\rH^0_\pH\langle-1\rangle
\]
is the morphism with components $\sD_\pst(\omega_\lambda^\et)$ and $\omega_\lambda^\deR$.

So in particular, if $X \to Y$ is an abelian-by-finite family over a general scheme~$Y$ and~$y\in Y(K_v)$ is a $K_v$-rational point, then
\[
\rH^{\leq 1}_\pH(X_y/K_v) \colonequals  (\rH^0_\pH(X_y/K_v),\rH^1_\pH(X_y/K_v))
\]
is a symplectic pair in the category of filtered discrete $(\ph, N, G_v)$-modules, the symplectic pairing on $\rH^1_\pH=\rH^1_\pH(X_{y}/K_v)$ being the pairing discussed above arising from the polarisation.

\subsection{The period map associated to an abelian-by-finite family}\label{ss:abf_period_map}

The particular structure of abelian-by-finite families implies that their period maps take on a particular form. Before we state this, we observe that de Rham Chern classes of line bundles behave well in families.

\begin{lem}\label{lem:relative_de_rham_chern}
	Let~$\pi\colon X\to Y$ be a smooth morphism of smooth varieties over a field~$K$ of characteristic zero. Then to any line bundle~$\dL$ on~$X$ one can associate a \emph{first relative de Rham Chern class} 
	\[
	c_1^\deR(\dL)_{/Y}\in \rF^1\!\rH^0(Y,\cH^2_\deR(X/Y))^{\nabla=0}
	\]
	with the following property. For any field extension~$L/K$ and any point~$y\in Y(L)$, the first de Rham Chern class $c_1^\deR(\dL|_{X_y})\in\rF^1\!\rH^2_\deR(X_y/L)$ of the restriction of~$\dL$ to the fibre at~$y$ is equal to the fibre of~$c_1^\deR(\dL)_{/Y}$ at~$y$ (once we identify $\cH^2_\deR(X/Y)_y\cong\rH^2_\deR(X_y/L)$ in the usual way).
	\begin{proof}
		To begin with, let us recall the definition of the first de Rham Chern class~$c_1^\deR(\dL)$ in the absolute case ($Y=\Spec(K)$), from \cite[\S7.7]{hartshorne:de_rham_cohomology}\footnote{In the proof of Proposition~\ref{prop:compatibility_ii}, we used a different definition of the first Chern class, namely that it is the cycle class of the zero-section in the total space of~$\dL$. That these two definitions agree follows from \cite[Proposition~7.7.1]{hartshorne:de_rham_cohomology}.}. The image of the morphism $\rd\log\colon\cO_X^\times\to\Omega^1_{X/K}$ of abelian sheaves, given by~$\rd\log(f) = f^{-1} \rd\! f$,	consists of exact differentials. So it induces a map $\cO_X^\times\to\Omega^\bullet_{X/K}[1]$ of complexes. The first de Rham Chern class of~$\dL$ is defined to be the image of~$\dL$ under the induced map
		\[
		\Pic(X)=\rH^1(X,\cO_X^\times) \to \rH^2(X,\Omega^\bullet_{X/K})=\rH^2_\deR(X/K) \,.
		\]
		The fact that~$\rd\log\colon\cO_X\to\Omega^\bullet_{X/K}[1]$ factors through $\rF^1\!\Omega^\bullet_{X/K}$ ensures that $c_1^\deR(\dL)\in\rF^1\!\rH^2_\deR(X/K)$.
		
		To relativise this construction, if~$\pi\colon X\to Y$ is a smooth morphism of smooth $K$-schemes, then there are maps
		\[
		\rH^1(X,\cO_X^\times) \to \rH^0(Y,\rR^1\!\pi_*\cO_X^\times) \xrightarrow{\rd \log} \rH^0(Y,\rR^2\!\pi_*\Omega^\bullet_{X/K}) \to \rH^0(Y,\rR^2\!\pi_*\Omega^\bullet_{X/Y})
		\]
		where the first map is part of the exact sequence of low-degree terms in the Leray spectral sequence, and the second and third maps are induced by the maps 
		\[
		\cO_X^\times\xrightarrow{\rd\log}\Omega^\bullet_{X/K}[1]\to\Omega^\bullet_{X/Y}[1].
		\]
		By definition of the Gau\ss--Manin connection on $\cH^2_\deR(X/Y)=\rR^2\!\pi_*\Omega^\bullet_{X/Y}$ as a connecting map \cite{katz-oda:relative_de_rham}, the image of the right-hand map is contained in $\rH^0(Y,\cH^2_\deR(X/Y))^{\nabla=0}$. We define the first relative de Rham Chern class 
		\[
		c_1^\deR(\dL)_{/Y}\in\rH^0(Y,\cH^2_\deR(X/Y))^{\nabla=0}
		\]
		to be the image of~$[\dL]\in\Pic(X)=\rH^1(X,\cO_X^\times)$ under the composite of the above maps. It is easy to check, using the naturality properties of the Leray spectral sequence and base-change maps, that this construction satisfies the desired property: the fibre of~$c_1^\deR(\dL)_{/Y}$ at $y\in Y(L)$ is $c_1^\deR(\dL|_{X_y})$. In particular, since each $c_1^\deR(\dL|_{X_y})$ lies in~$\rF^1$, so too does $c_1^\deR(\dL)_{/Y}$.
	\end{proof}
\end{lem}

So suppose now that $X\to Y'\to Y$ is an abelian-by-finite family with~$Y$ a smooth variety over a finite extension~$K_v$ of~$\bQ_p$, and write~$\cH^\bullet_\deR=\cH^\bullet_\deR(X/Y)$ for its relative de Rham cohomology. We define the first relative de Rham Chern class of the polarisation~$\lambda$ to be 
\[
c_1^\deR(\lambda)_{/Y}=\frac12c_1^\deR(\dL)_{/Y}\in\rF^1\!\rH^0(Y,\cH^2_\deR)^{\nabla=0},
\]
where $\dL=(1,\lambda)^*\dP$ with~$\dP$ the Poincar\'e line bundle on~$X\times_{Y'}X$. The class~$c_1^\deR(\lambda)_{/Y}$ corresponds to a $\cH^0_\deR$-linear morphism
\[
\check\omega^\deR_\lambda\colon \bigwedgesquare\nolimits_{\cH^0_\deR}\cH_1^\deR \to \cH^0_\deR\langle1\rangle
\]
of filtered vector bundles with integrable connection on~$Y$, where $\cH_1^\deR=\cHom_{\cH^0_\deR}(\cH^1_\deR,\cH^0_\deR)$. The fibre of~$\check\omega^\deR_\lambda$ at a point~$y\in Y(K_v)$ is equal to the pairing~$\check\omega^\deR_{\lambda_y}$ associated to the polarisation~$\lambda_y$ on the fibre~$X_y$. In particular, it follows from Lemma~\ref{lem:de_rham_pairing_perfect} that $\check\omega_\lambda^\deR$ is a perfect pairing, i.e.\ induces a $\cH^0_\deR$-linear isomorphism
\[
\cH_1^\deR \xrightarrow\sim \cHom_{\cH^0_\deR}(\cH_1^\deR,\cH^0_\deR\langle1\rangle) = \cH^1_\deR\langle1\rangle
\]
of filtered vector bundles with integrable connection on~$Y$. Hence, we have the twisted dual pairing
\[
\omega_\lambda^\deR\colon\bigwedgesquare\nolimits_{\cH^0_\deR}\cH^1_\deR\to\cH^0_\deR\langle-1\rangle \,,
\]
which is a $\cH^0_\deR$-linear perfect pairing whose fibre at a $K_v$-point~$y\in Y(K_v)$ is the pairing $\omega_{\lambda_y}^\deR$. The fact that the pairing~$\omega_\lambda^\deR$ is compatible with connections implies the following compatibility result for parallel transport.

\begin{lem}\label{lem:abf_compatibility_for_parallel_transport}
	Let~$y_0\in Y(K_v)$ be a $K_v$-rational point, and let $\Nbd_{y_0}\subseteq Y^\an$ be an admissible open neighbourhood of~$y_0$, isomorphic to a closed polydisc, over which $\cH^0_\deR(X/Y)^\an$ and $\cH^1_\deR(X/Y)^\an$ have a full basis of flat sections. Then the parallel transport maps
	\[
	T_{y_0}^\nabla\colon (\cO_{\Nbd_{y_0}}\otimes_{K_v}\rH^i_\deR(X_{y_0}/K_v),\rd\otimes1) \xrightarrow\sim (\cH^i_\deR(X/Y),\nabla)^\an|_{\Nbd_{y_0}}
	\]
	for~$i=0,1$ are compatible with all extra structures: algebra structure, module structure, and pairing.
\end{lem}

\begin{rmk}\label{rmk:trivialised_finite_cover}
	The fact that $\cH^0_\deR(X/Y)^\an=\cH^0_\deR(Y'/Y)^\an$ admits a trivialisation compatible with the algebra structure over the neighbourhood~$\Nbd_{y_0}$ in Lemma~\ref{lem:abf_compatibility_for_parallel_transport} means that the finite \'etale covering $Y'\to Y$ becomes trivialised over~$\Nbd_{y_0}$, i.e.\ there is an isomorphism $(Y')^\an|_{\Nbd_{y_0}}\cong\Nbd_{y_0}\times (Y')^\an_{y_0}$ of rigid-analytic spaces over~$K_v$, uniquely characterised by the fact that it restricts to the identity on the fibre at~$y_0$. On the fibre at another point~$y\in\Nbd_{y_0}(K_v)$, this isomorphism $(Y'_y)^\an\cong(Y'_{y_0})^\an$ is the one induced from the isomorphism
	\[
	\cO(Y'_{y_0}) = \rH^0_\deR(Y'_{y_0}/K_v) \xrightarrow[\sim]{T_{y_0,y}^\nabla} \rH^0_\deR(Y'_y/K_v) = \cO(Y'_y)
	\]
	of $K_v$-algebras.
\end{rmk}

If now~$y_0\in Y(K_v)$ is a $K_v$-rational point, let us define $\dH_{y_0}$ to be the $K_v$-variety parametrising Lagrangian $\rH^0_\deR(X_{y_0}/K_v)$-submodules of $\rH^1_\deR(X_{y_0}/K_v)$. This is a closed $K_v$-subvariety of the flag variety~$\dG_{y_0}$ parametrising filtrations on~$\rH^1_\deR(X_{y_0}/K_v)$ with the same dimension data as $\rF^\bullet\!\rH^1_\deR(X_{y_0}/K_v)$. As a consequence of Lemma~\ref{lem:abf_compatibility_for_parallel_transport}, if~$\Nbd_{y_0}\subseteq Y^\an$ is an admissible open neighbourhood of~$y_0$ over which $\cH^0_\deR(X/Y)^\an$ and $\cH^1_\deR(X/Y)^\an$ have a full basis of flat sections, then the image of the $v$-adic period map $\Phi_{y_0}\colon\Nbd_{y_0}\to\dG_{y_0}^\an$ of \S\ref{ss:period_maps} is contained in~$\dH_{y_0}^\an$. When we refer to the period map of an abelian-by-finite family, we will always regard it as having codomain~$\dH_{y_0}^\an$ rather than the flag variety $\dG_{y_0}^\an$.

To every point~$\Phi\in\dH_{y_0}(K_v)$, we assign a symplectic pair $\MM^{\leq 1}(\Phi)=(\MM^0,\MM^1(\Phi))$ in the category of filtered discrete $(\varphi,N,G_v)$-modules as follows.

\begin{defi}
	Let~$\Phi\in\dH_{y_0}(K_v)$, i.e.\ $\Phi$ is a Lagrangian $\rH^0_\deR(X_{y_0}/K_v)$-submodule of~$\rH^1_\deR(X_{y_0}/K_v)$. We define a filtered discrete $(\varphi,N,G_v)$-module~$\MM^1(\Phi)$ by
	\[
	\MM^1(\Phi) \colonequals (\rH^1_\pst(X_{y_0}/\bQ_p^\nr),\rH^1_\deR(X_{y_0}/K_v),c_\deR\circ c_\BO) \,,
	\]
	where the $(\varphi,N,G_v)$-module structure on~$\rH^1_\pst(X_{y_0}/\bQ_p^\nr)$ is the usual one, but where the filtration on $\rH^1_\deR(X_{y_0}/K_v)$ is the one given by
	\[
	\rF^i\!\rH^1_\deR(X_{y_0}/K_v) \colonequals 
	\begin{cases}
		\rH^1_\deR(X_{y_0}/K_v) & \text{if $i\leq0$,} \\
		\Phi & \text{if $i=1$,} \\
		0 & \text{if $i\geq2$,}
	\end{cases}
	\]
	instead of the Hodge filtration. Since~$\Phi$ is a Lagrangian $\rH^0_\deR(X_{y_0}/K_v)$-submodule of~$\rH^1_\deR(X_{y_0}/K_v)$, this implies that $\MM^1(\Phi)$ is a symplectic module under
		\[
	\MM^0 \colonequals (\rH^0_\pst(X_{y_0}/\bQ_p^\nr),\rH^0_\deR(X_{y_0}/K_v),c_\deR\circ c_\BO)
	\]
	in the category of filtered discrete $(\varphi,N,G_v)$-modules, with respect to the pairing~$\omega_\lambda^\pH$ described in~\S\ref{sss:comparison_of_pairs}. We write
	\[
	\MM^{\leq1}(\Phi) \colonequals (\MM^0,\MM^1(\Phi))
	\]
	for the corresponding symplectic pair.
\end{defi}

As in Proposition~\ref{prop:period_maps_control_reps}, the $v$-adic period map associated to an abelian-by-finite family controls the variation of the local Galois representations attached to $K_v$-points, now compatibly with symplectic structures. In the proposition below, we write $\pi_0\SPair(\Rep_{\bQ_p}^\deR(G_v))$ (respectively $\pi_0\SPair(\MF(\varphi,N,G_v))$) for the set of isomorphism classes of symplectic pairs in the category of de Rham $G_v$-representations (respectively filtered discrete $(\varphi,N,G_v)$-modules).

\begin{prop}\label{prop:period_maps_control_abf_reps}
	Let $X\to Y'\to Y$ be an abelian-by-finite family over a smooth $K_v$-variety~$Y$, let $y_0\in Y(K_v)$ be a $K_v$-rational point, and let $\Nbd_{y_0}\subseteq Y^\an$ be an admissible open neighbourhood of~$y_0$, isomorphic to a closed polydisc, over which $\cH^0_\deR(X/Y)^\an$ and $\cH^1_\deR(X/Y)^\an$ have a full basis of horizontal sections.
	
	Then the map $Y(K_v)\to\pi_0\SPair(\Rep_{\bQ_p}^\deR(G_v))$ sending a $K_v$-point $y\in Y(K_v)$ to the isomorphism class of $\rH^{\leq 1}_\et(X_{y,\Kbar_v},\bQ_p)$ fits into a commuting diagram
	\begin{center}
	\begin{tikzcd}
		Y(K_v) \arrow[r,phantom,"\supset"]\arrow[d] &[-8ex] \Nbd_{y_0}(K_v) \arrow[r,"\Phi_{y_0}"] & 
		\dH_{y_0}(K_v) \arrow[d,"\MM^{\leq 1}"] \\
		\pi_0\SPair(\Rep_{\bQ_p}^\deR(G_v)) \arrow[rr,"\DpH",hook]  & & \pi_0\SPair(\MF(\varphi,N,G_v)) \,.
	\end{tikzcd}
	\end{center}
\end{prop}
\begin{proof}
This follows from
\[
\DpH\big(\rH^{\leq 1}_\et(X_{y,\Kbar_v},\bQ_p)\big) \simeq \rH^{\leq 1}_\pH(X_{y}/K_v) \simeq 
\MM^{\leq 1}(\Phi_{y_0}(y)) \,,
\]
where the first isomorphism is the comparison isomorphism and the second is parallel transport.
\end{proof}

\subsubsection{Decomposition of the period map}\label{sss:period_map_decomposition}

The period map can be understood somewhat explicitly. For this, let us enumerate the closed points of~$Y'_{y_0}$ as $(y_i')_{i\in I}$. For each~$i$, we write~$L_{w_i}\colonequals K_v(y_i')$ for short, and fix a $K_v$-embedding $L_{w_i}\hookrightarrow\Kbar_v$ so as to make the absolute Galois group~$G_{w_i}$ of~$L_{w_i}$ into an open subgroup of~$G_v$.

From~\eqref{eq:de_rham_decomposition}, we see that every Lagrangian $\rH^0_\deR(X_{y_0}/K_v)=\prod_iL_{w_i}$-submodule~$\Phi$ of
\[
\rH^1_\deR(X_{y_0}/K_v)=\prod_i\rH^1_\deR(X_{y_i'}/L_{w_i})
\]
factorises as a product $\prod_i\Phi_i'$, where each $\Phi_i'$ is a Lagrangian $L_{w_i}$-subspace of $\rH^1_\deR(X_{y_i'}/L_{w_i})$. Accordingly, the period domain~$\dH_{y_0}$ factorises as
\begin{equation}\label{eq:grassmannian_decomposition}
\dH_{y_0} = \prod_{y_0'\in|Y'_{y_0}|}\Res^{L_{w_i}}_{K_v}\dH_{y_i'} \,,
\end{equation}
where we write $\dH_{y_i'}\colonequals\LGrass(\rH^1_\deR(X_{y_i'}/L_{w_i}))$ for the Grassmannian of Lagrangian $L_{w_i}$-subspaces of $\rH^1_\deR(X_{y_i'}/L_{w_i})$ viewed as a $L_{w_i}$-variety.
\smallskip

This decomposition of~$\dH_{y_0}$ induces a corresponding decomposition for the map 
\[
M^{\leq1}\colon\dH_{y_0}(K_v)\to\pi_0\SPair(\MF(\varphi,N,G_v)) \,.
\]
For a Lagrangian $L_{w_i}$-subspace~$\Phi_i'$ of~$\rH^1_\deR(X_{y_i'}/L_{w_i})$, let us write~$\MM^1_i(\Phi_i')$ for the symplectic filtered discrete $(\varphi,N,G_{w_i})$-module with
\[
\MM^1_i(\Phi_i') \colonequals (\rH^1_\pst(X_{y_i'}/\bQ_p^\nr),\rH^1_\deR(X_{y_i'}/L_{w_i}),c_\deR\circ c_\BO) \,,
\]
where the $(\varphi,N,G_{w_i})$-module structure on~$\rH^1_\pst(X_{y_i'}/\bQ_p^\nr)$ is the usual one, but where the filtration on $\rH^1_\deR(X_{y_i'}/L_{w_i})$ is the one given by
\[
\rF^j\!\rH^1_\deR(X_{y_i'}/L_{w_i}) = 
\begin{cases}
	\rH^1_\deR(X_{y_i'}/L_{w_i}) & \text{if $j\leq0$,} \\
	\Phi_i' & \text{if $j=1$,} \\
	0 & \text{if $j\geq2$,}
\end{cases}
\]
instead of the usual Hodge filtration. 

Denoting the unit object in $\MF(\varphi,N,G_{w_i})$ by $\unit = (\bQ_p^\nr, L_{w_i}, 1)$, we then have the following.

\begin{lem}\label{lem:M_decomposition}
	Let~$\Phi\in\dH_{y_0}(K_v)$ be a Lagrangian $\rH^0_\deR(X_{y_0}/K_v)$-submodule of $\rH^1_\deR(X_{y_0}/K_v)$, factorising as the product $\prod_i\Phi_i'$ of Lagrangian $L_{w_i}$-subspaces~$\Phi_i'$ of~$\rH^1_\deR(X_{y_0'}/L_{w_i})$. Then there is a decomposition
	\[
	\MM^1(\Phi) = \prod_i\Ind_{G_{w_i}}^{G_v}\!\MM^1_i(\Phi_i')
	\]
	in the category of filtered discrete $(\varphi,N,G_v)$-modules, compatible with symplectic module structures over
	\[
	\MM^0 \cong \prod_i\Ind_{G_{w_i}}^{G_v}\!\unit \,.
	\]
	\begin{proof}
		As in Example~\ref{ex:restriction_and_coinduction_of_p-adic_hodge_coh}, the decomposition~$X_{y_0}=\coprod_iX_{y_i'}$ induces a decomposition
		\[
		\rH^1_\pH(X_{y_0}/K_v) = \prod_i\Ind_{G_{w_i}}^{G_v}\rH^1_\pH(X_{y_i'}/L_{w_i})
		\]
		in the category of filtered discrete $(\varphi,N,G_v)$-modules, compatible with symplectic module structures over~$\rH^0_\pH(X_{y_0}/K_v)=M^0$, whose de Rham component is the decomposition~\eqref{eq:de_rham_decomposition}. After replacing the Hodge filtration on~$\rH^1_\deR(X_{y_0}/K_v)$ and on each~$\rH^1_\deR(X_{y_i'}/L_{w_i})$ by the filtrations determined by~$\Phi$ and~$\Phi_i'$, respectively, we obtain the desired decomposition of~$M^1(\Phi)$.
	\end{proof}
\end{lem}
\section{Comparison of period maps and full monodromy}\label{s:big_monodromy}

The relevance of $v$-adic period maps in the context of Diophantine geometry is that they provide a tool for proving finiteness of subsets of the $v$-adic points on a smooth curve~$Y$: this idea underlies the approach of Lawrence and Venkatesh. In order to make this work, one needs to be able to ensure that the $v$-adic period map associated to some smooth proper family $X\to Y$ has large image. This is achieved by means of a comparison theorem relating $v$-adic period maps with period maps over the complex numbers, the latter of which can be controlled by monodromy computations and basic topology.

In this section, we review the theory of period maps over the complex numbers and its comparison with $v$-adic period maps. Using this, we recall what it means for an abelian-by-finite family to have \emph{full monodromy} \cite[p.~927]{LVinventiones}, and give a criterion for certain subsets of~$Y(K_v)$ to be finite in terms of $v$-adic period maps. The discussion here closely parallels \cite[\S3]{LVinventiones}, save that we need to compare $v$-adic and complex period maps around any $K_v$-point of~$Y$, not just those defined over~$K$.

\subsection{Period maps over the complex numbers}

Let~$Y$ be a smooth connected variety over~$\bC$. By an \emph{\'etale neighbourhood}~$\Nbd_{y_0}$ of a point~$y_0\in Y^\an$ in the analytification~$Y^\an$ of~$Y$, we mean a local biholomorphism $\Nbd_{y_0}\to Y^\an$ of complex manifolds, together with a chosen point~$\tilde y_0\in\Nbd_{y_0}$ mapping to~$y_0$. For example, $\Nbd_{y_0}$ could be an open neighbourhood of~$y_0$ in the analytic topology, or could be the universal cover of~$Y^\an$ with a chosen point in the fibre over~$y_0$.

If~$\Nbd_{y_0}$ is a simply connected \'etale neighbourhood of~$y_0$ and $\dE=(\dE,\nabla)$ is a holomorphic vector bundle with flat connection on~$Y^\an$, then by the Riemann--Hilbert correspondence, the restriction (pullback) of~$\dE$ to~$\Nbd_{y_0}$ is trivial: there is a unique isomorphism
\[
T_{\tilde y_0}^\nabla\colon(\cO_{\Nbd_{y_0}}\otimes_\bC\dE_{y_0},\rd\otimes1) \xrightarrow\sim(\dE|_{\Nbd_{y_0}},\nabla|_{\Nbd_{y_0}})
\]
of holomorphic vector bundles with flat connection on~$\Nbd_{y_0}$, characterised by the fact that the fibre of $T_{\tilde y_0}^\nabla$ at~$\tilde y_0$ is the identity on~$\dE_{y_0}$. Suppose moreover that $\dE$ comes with an exhaustive, separated descending filtration $\rF^\bullet\!\dE$ whose graded pieces are all vector bundles. One then defines, as in Definition~\ref{def:period_maps}, the period domain~$\dG_{y_0}$ to be the complex flag variety parametrising filtrations on~$\dE_{y_0}$ with the same dimension data as~$\rF^\bullet\!\dE_{y_0}$, and defines the \emph{complex period map}
\[
\Phi_{\tilde y_0}\colon\Nbd_{y_0}\to\dG_{y_0}^\an
\]
to be the holomorphic map classifying the filtration on $\cO_{\Nbd_{y_0}}\otimes_{\bC}\dE_{y_0}$ given by pulling back the filtration~$\rF^\bullet$ along the parallel transport map~$T_{\tilde y_0}^\nabla$. Here we are implicitly using a description of the functor of points of~$\dG_{y_0}^\an$ similar to Proposition~\ref{prop:analytic_flag_variety}.

Later, we will be interested in the Zariski-closure of the image of the complex period map~$\Phi_{\tilde y_0}$, by which we mean the smallest closed subvariety~$Z\subseteq\dG_{y_0}$ such that~$\Phi_{\tilde y_0}$ factors through~$Z^\an$. We note the following regarding this image.

\begin{lem}\label{lem:period_image_same}
	The Zariski-closure of the image of the complex period map $\Phi_{\tilde y_0}\colon\Nbd_{y_0}\to\dG_{y_0}^\an$ is independent of the choice of simply connected \'etale neighbourhood~$\Nbd_{y_0}$.
	\begin{proof}
		It suffices to prove the following: if~$\Nbd_{y_0}'$ is a simply connected open subset of~$\Nbd_{y_0}$ containing~$\tilde y_0$, then the Zariski-closure of the image of~$\Phi_{\tilde y_0}$ is equal to the Zariski-closure of the image of~$\Phi_{\tilde y_0}|_{\Nbd_{y_0}'}$. Let us write~$Z$ and~$Z'$, respectively, for the Zariski-closures of these images. We clearly have $Z'\subseteq Z$. For the converse inclusion, $\Phi_{\tilde y_0}^{-1}(Z')$ is a closed analytic subvariety of~$\Nbd_{y_0}$ with non-empty interior (since it contains~$\Nbd_{y_0}'$). Hence by isolation of zeroes $\Phi_{\tilde y_0}^{-1}(Z')=\Nbd_{y_0}$, so we have the converse inclusion $Z\subseteq Z'$.
	\end{proof}
\end{lem}

\subsubsection{Monodromy}

One can gain some control on the Zariski-closure of the image of the complex period map using monodromy actions. Let~$\bE\colonequals\dE^{\nabla=0}$ be the $\bC$-local system on~$Y^\an$ corresponding to~$\dE$ under the Riemann--Hilbert correspondence. There is thus a monodromy action of $\pi_1(Y^\an,y_0)$ on the fibre $\bE_{y_0}=\dE_{y_0}$, and hence on the flag variety~$\dG_{y_0}$. This monodromy action gives a lower bound on the image of the period map, as follows.

\begin{lem}\label{lem:monodromy_controls_period_image}
	The Zariski-closure of the image of the complex period map $\Phi_{\tilde y_0}\colon\Nbd_{y_0}\to\dG_{y_0}^\an$ contains the orbit $\pi_1(Y^\an,y_0)\cdot h_0$ of the point~$h_0\in\dG_{y_0}(\bC)$ corresponding to the filtration $\rF^\bullet\!\dE_{y_0}$.
	\begin{proof}
		By Lemma~\ref{lem:period_image_same}, it suffices to prove this in the case that~$\Nbd_{y_0}$ is the universal cover of~$Y^\an$. We will show that in this case, $\pi_1(Y^\an,y_0)\cdot h_0$ is already contained in the image of~$\Phi_{\tilde y_0}$, without passing to Zariski-closures.
		
		So let~$\gamma\in\pi_1(Y^\an,y_0)$, and lift~$\gamma$ to a path~$\tilde\gamma$ in~$\Nbd_{y_0}$ starting from~$\tilde y_0$. The end point~$\tilde y_0'$ of~$\tilde\gamma$ also lies above~$y_0$, and the monodromy action of~$\gamma$ is given by the parallel transport map\footnote{Our convention for composition in fundamental groups is that $\gamma_2\gamma_1$ denotes the composite loop given by first following~$\gamma_1$ and then following~$\gamma_2$. This is the opposite of the usual convention in topology.}
		\[
		T_{\tilde y_0,\tilde y_0'}^\nabla\colon\dE_{y_0}\to\dE_{y_0} \,.
		\]
		Hence $\gamma^{-1}\cdot h_0=\Phi_{\tilde y_0}(\tilde y_0')$ lies in the image of the complex period map, which is what we wanted to show.
	\end{proof}
\end{lem}

\subsection{Comparison between $v$-adic and complex period maps}

Complex period maps can be used to control the image of $v$-adic period maps. Let~$K_v$ be a finite extension of~$\bQ_p$ and $Y$ a smooth geometrically connected variety over~$K_v$ with a fixed basepoint $y_0\in Y(K_v)$, as in \S\ref{ss:period_maps}. Let~$(\dE,\nabla,\rF^\bullet)$ be an \emph{algebraic} filtered vector bundle with flat connection on~$Y$, and choose an admissible open subset~$\Nbd_{y_0,v}\subseteq Y^\an$ in the rigid analytification~$Y^\an$ of~$Y$ over which~$\dE^\an$ has a full basis of flat sections, so that we have a $v$-adic period map
\[
\Phi_{y_0,v}\colon\Nbd_{y_0,v}\to\dG_{y_0}^\an \,,
\]
where~$\dG_{y_0}$ is a flag variety parametrising filtrations on~$\dE_{y_0}$, as in Definition~\ref{def:period_maps}.

Choose an embedding $\iota\colon K_v\hookrightarrow\bC$ and let~$Y_\bC^\an$ denote the complex analytification of the base change~$Y_\bC$ of~$Y$ to~$\bC$ along~$\iota$. Choose a simply connected \'etale neighbourhood~$\Nbd_{y_0,\infty}$ of~$y_0$ in~$Y_\bC^\an$, so that we also have a complex period map
\[
\Phi_{\tilde y_0,\infty}\colon\Nbd_{y_0,\infty}\to\dG_{y_0,\bC}^\an \,.
\]
We will prove the following.

\begin{prop}\label{prop:period_map_comparison}
	Let~$Z_v\subseteq\dG_{y_0}$ and~$Z_\infty\subseteq\dG_{y_0,\bC}$ denote the Zariski-closure of the image of $v$-adic period map~$\Phi_{y_0,v}$ and the complex period map~$\Phi_{\tilde y_0,\infty}$, respectively. Then $Z_\infty=Z_{v,\bC}$ is the base change of~$Z_v$ to~$\bC$ along~$\iota$.
\end{prop}

The proof of Proposition~\ref{prop:period_map_comparison} is essentially contained in \cite[\S3.4]{LVinventiones}, except that there the variety~$Y$ and point~$y_0$ are defined over a number field rather than~$K_v$. For the sake of completeness, we give the full argument here.

We begin with a preparatory observation, which parallels Lemma~\ref{lem:period_image_same} for complex period maps.

\begin{lem}\label{lem:period_image_same_v-adic}
	The Zariski-closure of the image of the $v$-adic period map $\Phi_{\tilde y_0,v}\colon\Nbd_{y_0,v}\to\dG_{y_0}^\an$ is independent of the choice of $v$-adic neighbourhood~$\Nbd_{y_0,v}$.
	\begin{proof}
		A similar proof to that of Lemma~\ref{lem:period_image_same} works. The key point is that any closed analytic subvariety of a polydisc with non-empty interior is equal to the whole polydisc: this follows since open inclusions of polydiscs induce injective maps on their affinoid rings.
	\end{proof}
\end{lem}

Hence, in proving Proposition~\ref{prop:period_map_comparison} we are free to replace~$Y$ by a Zariski-open subvariety containing~$y_0$, shrinking~$\Nbd_{y_0,v}$ and~$\Nbd_{y_0,\infty}$ correspondingly. Thus, we are free to assume that~$Y$ is affine and connected, that the vector bundle~$\dE\simeq\cO_Y^{\oplus m}$ is trivial, and that so too are the graded pieces of the filtration~$\rF^\bullet\!\dE$. We fix an identification $\dE=\cO_Y^{\oplus m}$ for which $\rF^i\!\dE=\cO_Y^{\oplus m_i}$ is the constant subbundle spanned by the first~$m_i$ basis sections, where~$m_i\colonequals\rank\rF^i\!\dE$. Hence, the connection~$\nabla$ on~$\dE=\cO_Y^{\oplus m}$ is given by 
\[
\nabla(f)=\rd(f)+\omega\cdot f,
\]
where the connection matrix~$\omega$ is an $m\times m$ matrix with coefficients in $\Gamma(Y,\Omega^1_{Y/K_v})$. Flatness of the connection is equivalent to the equality $\rd\omega+\omega\wedge\omega=0$.

Now let us fix a system $\bt = t_1,\dots,t_n$ of local parameters at~$y_0$. \cite[Proposition~8.9]{katz:nilpotent_connections} ensures that there is a \emph{unique} $m\times m$ matrix $T$ with coefficients in $K_v\llbrack\bt\rrbrack=K_v\llbrack t_1,\dots,t_n\rrbrack$ satisfying the differential equation
\begin{equation}\label{eq:de_for_parallel_transport}
	\rd T=-\omega\cdot T\;\text{ subject to the initial condition }\;T|_{\bt=0} = \mathbf I_m \,,
\end{equation}
where 
\[
\rd\colon K_v\llbrack\bt\rrbrack\to\Omega^{1,f}_{K_v\llbrack\bt\rrbrack/K_v}=\bigoplus_{i=1}^nK_v\llbrack\bt\rrbrack\rd\! t_i
\]
is the universal $\bt$-adically continuous derivation. Here, by mild abuse of notation, we regard $\omega$ as a matrix with coefficients in $\Omega^{1,f}_{K_v\llbrack\bt\rrbrack/K_v}$ via the inclusion $\Gamma(Y,\Omega^1_{Y/K_v})\hookrightarrow\Omega^{1,f}_{K_v\llbrack\bt\rrbrack/K_v}$ given by taking power series expansions.

For $1\leq i\leq m$ we write $T_i$ for the $i$th column of $T$, and let
\[
\hat\Phi_{y_0}\colon\Spec(K_v\llbrack\bt\rrbrack)\to\dG_{y_0}
\]
denote the $K_v$-scheme map classifying the filtration on $K_v\llbrack\bt\rrbrack^{\oplus m}$ where $\rF^i$ is the span of the vectors $T_1,\dots,T_{m_i}$. Explicitly, the flag variety $\dG_{y_0}$ is a closed subvariety of $\prod_{i\in\bZ}\bP_{K_v}^{{m\choose m_i}-1}$ via the embedding taking a filtration~$\rF^\bullet$ on $\dE_{y_0}=K_v^{\oplus m}$ to the sequence of one-dimensional subspaces $\bigwedge^{m_i}\rF^i$ inside $\bigwedge^{m_i}(K_v^{\oplus m})=K_v^{\oplus{m\choose m_i}}$ for $i\in\bZ$, and $\hat\Phi_{y_0}$ is the map whose $i$th component is given by $T_1\wedge\dots\wedge T_{m_i}$, viewed as a map from $K_v\llbrack\bt\rrbrack\to\bA^{m\choose m_i}\setminus\{0\}$.
\smallskip

Now shrinking the $v$-adic neighbourhood~$\Nbd_{y_0,v}$ and rescaling the local parameters $t_1,\dots,t_n$ if necessary, we may assume that $\Nbd_{y_0,v}=\Sp(K_v\langle\bt\rangle)$ is the polydisc of radii~$1$ in the parameters~$t_1,\dots,t_n$. The sheaf of $K_v$-analytic $1$-forms on~$\Nbd_{y_0,v}$ is the free vector bundle generated by $\dt_1,\dots,\dt_n$ \cite[Proof of Theorem~3.6.1]{fresnel-van_der_put:rigid_analytic_geometry}, so the $v$-adic parallel transport map
\[
T_{y_0,v}^\nabla\colon(\cO_{\Nbd_{y_0,v}}^{\oplus m},\rd^{\oplus m})\xrightarrow\sim(\cO_{\Nbd_{y_0,v}}^{\oplus m},\nabla)
\]
is represented by an $m\times m$ matrix with coefficients in $K_v\langle\bt\rangle\subset K_v\llbrack\bt\rrbrack$ which satisfies the differential equation~\eqref{eq:de_for_parallel_transport}. It thus follows by unicity of solutions to~\eqref{eq:de_for_parallel_transport} that $T_{y_0,v}^\nabla$ is represented by the matrix~$T$ (which thus has coefficients in $K_v\langle\bt\rangle$). It then follows that the $v$-adic period map
\[
\Phi_{y_0,v}\colon\Nbd_{y_0,v}\to\dG_{y_0}^\an\subseteq\prod_{i\in\bZ}\bP_{K_v}^{{m\choose m_i}-1,\an}
\]
is the map whose $i$th component is given by $T_1\wedge\dots\wedge T_{m_i}$: in other words, the Taylor expansion of $\Phi_{y_0,v}$ at~$y_0$ is the map $\hat\Phi_{y_0}$. Using this, we see

\begin{lem}\label{lem:v-adic_image_is_formal}
	The Zariski-closure~$Z_v$ of the image of the $v$-adic period map $\Phi_{y_0,v}\colon\Nbd_{y_0,v}\to\dG_{y_0}^\an$ is the scheme-theoretic image of~$\hat\Phi_{y_0}$.
	\begin{proof}
		Choose a projective embedding $\dG_{y_0}\subseteq\bP_{K_v}^N$. The above discussion shows that there are elements $f_0,\dots,f_N\in K_v\langle\bt\rangle$ such that both $\Phi_{y_0,v}\colon\Nbd_{y_0,v}\to\dG_{y_0}^\an$ and $\hat\Phi_{y_0}\colon\Spec(K_v\llbrack\bt\rrbrack)\to\dG_{y_0}$ are given in projective coordinates by $(f_0:\ldots:f_N)$. The homogenous ideal of definition of~$Z_v\subseteq\bP^N_{K_v}$ is then the ideal generated by those homogenous elements $F\in K_v[X_0,\dots,X_N]$ such that $F(f_0,\dots,f_N)=0$. But this is also the homogenous ideal of definition of the scheme-theoretic image of~$\hat\Phi_{y_0}$.
	\end{proof}
\end{lem}
	
We can apply exactly the same argument to the complex-analytic period map. Shrinking~$\Nbd_{y_0,\infty}$ if necessary, we may assume that the local parameters $t_1,\dots,t_n$ are defined on all of~$\Nbd_{y_0,\infty}$, and that $\dt_1,\dots,\dt_n$ forms a basis for the $\cO(\Nbd_{y_0,\infty})$-module of holomorphic $1$-forms on~$\Nbd_{y_0,\infty}$. The embedding $\iota\colon K_v\hookrightarrow\bC$ induces an embedding $K_v\llbrack\bt\rrbrack\hookrightarrow\bC\llbrack\bt\rrbrack$, which we also denote by~$\iota$. Since the sheaf of holomorphic $1$-forms on $\Nbd_{y_0,\infty}$ is the free vector bundle generated by $\dt_1,\dots,\dt_n$, the complex-analytic parallel transport map
\[
T_{y_0,\infty}^\nabla\colon(\cO_{\Nbd_{y_0,\infty}}^{\oplus m},\rd^{\oplus m})\xrightarrow\sim(\cO_{\Nbd_{y_0,\infty}}^{\oplus m},\nabla)
\]
is represented by an $m\times m$ matrix $T_\infty$ with coefficients in the ring $\cO(\Nbd_{y_0,\infty})\subset\bC\llbrack\bt\rrbrack$ of holomorphic functions on~$\Nbd_{y_0,\infty}$ which satisfies the differential equation
\[
\rd T_\infty=-\iota(\omega)\cdot T_\infty\;\text{ subject to the initial condition }\;T_\infty|_{\bt=0} = \mathbf I_m \,.
\]
Since this differential equation has a unique solution over $\bC\llbrack\bt\rrbrack$, it follows that $T_\infty=\iota(T)$. It then follows that the complex-analytic period map
\[
\Phi_{y_0,\infty}\colon\Nbd_{y_0,\infty}\to\dG_{y_0,\bC}^\an\subseteq\prod_{i\in\bZ}\bP_\bC^{{m\choose m_i}-1,\an}
\]
is the map whose $i$th component is given by $\iota(T_1)\wedge\dots\wedge\iota(T_{m_i})$. As for the~$v$-adic period map, this implies

\begin{lem}\label{lem:complex_image_is_formal}
	The Zariski-closure~$Z_\infty$ of the image of the complex period map $\Phi_{\tilde y_0,\infty}\colon\Nbd_{y_0,\infty}\to\dG_{y_0,\bC}^\an$ is the scheme-theoretic image of $\hat\Phi_{y_0,\bC}\colon\Spec(\bC\otimes_\iota K_v\llbrack\bt\rrbrack)\to\dG_{y_0,\bC}$.
	\begin{proof}
		A similar proof to that of Lemma~\ref{lem:v-adic_image_is_formal} works. That is, if we pick a projective embedding $\dG_{y_0}\subseteq\bP^N_{K_v}$, then we have shown that there are elements $f_0,\dots,f_N\in K_v\llbrack\bt\rrbrack$ with each $\iota(f_i)\in\bC\llbrack\bt\rrbrack$ convergent on~$\Nbd_{y_0,\infty}$, such that $\hat\Phi_{y_0}\colon\Spec(K_v\llbrack\bt\rrbrack)\to\dG_{y_0}\subseteq\bP^N_{K_v}$ and $\Phi_{y_0,\infty}\colon\Nbd_{y_0,\infty}\to\dG_{y_0,\bC}^\an\subseteq\bP_{\bC}^{N,\an}$ are given in projective coordinates by $(f_0:\dots:f_N)$ and $(\iota(f_0):\dots:\iota(f_N))$, respectively. The homogenous ideal of definition of $Z_\infty\subseteq\bP^N_{K_v}$ is then the ideal generated by those homogenous elements $F\in\bC[X_0,\dots,X_N]$ such that $F(\iota(f_0),\dots,\iota(f_N))=0$. But this is also the homogenous ideal of definition of the scheme-theoretic image of~$\hat\Phi_{y_0,\bC}$.
	\end{proof}
\end{lem}

Taken together, Lemmas~\ref{lem:v-adic_image_is_formal} and~\ref{lem:complex_image_is_formal} imply Proposition~\ref{prop:period_map_comparison}. Indeed, for any open affine $W\subseteq\dG_{y_0}$ containing the base point (corresponding to the filtration $\rF^\bullet\dE_{y_0}$ on $\dE_{y_0}$), the intersection $Z_v\cap W$ is the subscheme of~$W$ cut out by the kernel of $\hat\Phi_{y_0}^*\colon\cO(W)\to K_v\llbrack\bt\rrbrack$, while $Z_\infty\cap W_\bC$ is the subscheme of~$W_\bC$ cut out by the kernel of $\bC\otimes_\iota\hat\Phi_{y_0}^*\colon\bC\otimes_\iota\cO(W)\to\bC\otimes_\iota K_v\llbrack\bt\rrbrack$. But $\ker(\bC\otimes_\iota\hat\Phi_{y_0}^*)=\bC\otimes_\iota\ker(\hat\Phi_{y_0}^*)$, whence $Z_\infty\cap W_\bC=(Z_v\cap W)_\bC$. Thus $Z_\infty=Z_{v,\bC}$ as desired.\qed

\begin{rmk}
	The germ~$\hat\Phi_{y_0}$ of the period map can be characterised intrinsically, without choosing bases: the restriction~$\dE_{\hat y_0}$ of~$\dE$ to the formal neighbourhood $\Spec(\hat\cO_{Y,y_0})$ of~$y_0$ in~$Y$ is a finite $\hat\cO_{Y,y_0}$-module with a continuous flat connection, so by \cite[Proposition~8.9]{katz:nilpotent_connections} admits a unique trivialisation
	\[
	\hat T_{y_0}\colon(\hat\cO_{Y,y_0}\otimes_{K_v}\dE_{y_0},\rd\otimes1) \xrightarrow\sim (\dE_{\hat y_0},\nabla_{\hat y_0}) \,.
	\]
	The map $\hat\Phi_{y_0}\colon\Spec(\hat\cO_{Y,y_0})\to\dG_{y_0}$ is just the map classifying the filtration on $\hat\cO_{Y,y_0}\otimes_{K_v}\dE_{y_0}$ given by pulling back the filtration on~$\dE_{\hat y_0}$ along~$\hat T_{y_0}$. This description makes it clear that if the variety~$Y$, the point~$y_0$ and the filtered vector bundle with flat connection~$(\dE,\nabla,\rF^\bullet)$ are all defined over some subfield~$K$ of~$K_v$, then so too is~$\hat\Phi_{y_0}$.
\end{rmk}

\subsection{Full monodromy and a criterion for finiteness}

Let us now specialise all of the above theory to the case of abelian-by-finite families. Suppose initially that $X\to Y'\to Y$ is an abelian-by-finite family, where~$Y$ is a smooth connected variety over~$\bC$. The relative de Rham cohomology $\cH^\bullet_\deR(X/Y)^\an$ corresponds, under the Riemann--Hilbert correspondence, to the relative Betti cohomology $\rR^\bullet\!\pi^\an_*\ubC_{X^\an}$. The fibre of~$\rR^\bullet\!\pi^\an_*\ubC_{X^\an}$ at a point $y_0\in Y^\an$ is the usual $\bC$-linear Betti cohomology $\rH^\bullet_\Betti(X_{y_0},\bC)$ of the fibre~$X_{y_0}$, i.e.\ the singular cohomology of its analytification~$X^\an_{y_0}$. The fibre~$X_{y_0}$ is a disjoint union of polarised complex abelian varieties, indexed by the closed points of~$Y'$ above~$y_0$, so we have a decomposition
\[
\rH^1_\Betti(X_{y_0},\bQ) = \bigoplus_{y_0'\in|Y'_{y_0}|}\rH^1_\Betti(X_{y_0'},\bQ) \,,
\]
where each~$\rH^1_\Betti(X_{y_0'},\bQ)$ carries a perfect alternating pairing induced from the polarisation on~$X_{y_0'}$ (constructed analogously to the pairings in~\S\ref{s:abelian-by-finite}).

The following definition will play a key role.

\begin{defi}[cf.\ {\cite[(i) on p928]{LVinventiones}}]\label{def:full_monodromy}
	We say that the abelian-by-finite family $X\to Y'\to Y$ has \emph{full monodromy} just when the Zariski-closure of the image of the monodromy representation
	\[
	\rho\colon\pi_1(Y^\an,y_0)\rightarrow\GL(\rH^1_\Betti(X_{y_0},\bQ))(\bQ)
	\]
	contains 
	\[
	\prod_{y_0'\in|Y'_{y_0}|}\Sp(\rH^1_\Betti(X_{y_0'},\bQ)).
	\]
	It is easy to check that this property is independent of the point~$y_0$.
	
	More generally, if~$Y$ is a smooth geometrically connected variety over a characteristic~$0$ field~$K$, we say that an abelian-by-finite family $X\to Y'\to Y$ has \emph{full monodromy} with respect to a complex embedding $\iota\colon K\hookrightarrow\bC$ just when $X_\bC\to Y'_\bC\to Y_\bC$ has full monodromy.
\end{defi}

For us, the importance of full monodromy is that it gives a simple criterion for the image of the $v$-adic period map to be as large as possible.

\begin{lem}\label{lem:full_monodromy_gives_dense_image}
	Suppose that~$Y$ is a smooth geometrically connected variety over a finite extension~$K_v$ of~$\bQ_p$, and that $X\to Y'\to Y$ is an abelian-by-finite family with full monodromy with respect to some complex embedding $\iota\colon K_v\hookrightarrow\bC$. Then for all~$y_0\in Y(K_v)$ the $v$-adic period map
	\[
	\Phi_{y_0}\colon\Nbd_{y_0}\rightarrow\dH_{y_0}^\an
	\] 
	has Zariski-dense image, where~$\dH_{y_0}$ is the $K_v$-variety parametrising Lagrangian $\rH^0_\deR(X_{y_0}/K_v)$-submodules of $\rH^1_\deR(X_{y_0}/K_v)$, as in~\S\ref{ss:abf_period_map}.
	\begin{proof}
		By Proposition~\ref{prop:period_map_comparison} and Lemma~\ref{lem:monodromy_controls_period_image}, it suffices to show that the orbit of the base point $h_0\in\dH_{y_0}(\bC)$ under the monodromy action of $\pi_1(Y_\bC^\an,y_0)$ is Zariski-dense. By assumption, the Zariski-closure of this orbit contains the orbit of~$h_0$ under the action of $\prod_{y_0'\in|Y'_{y_0}|}\Sp(\rH^1_\deR(X_{y_0',\bC}/\bC))$. But this group acts transitively on Lagrangian submodules of $\rH^1_\deR(X_{y_0,\bC}/\bC)=\bigoplus_{y_0'\in|Y'_{y_0}|}\rH^1_\deR(X_{y_0',\bC}/\bC)$, so we are done.
	\end{proof}
\end{lem}

As a consequence of Lemma~\ref{lem:full_monodromy_gives_dense_image}, we obtain a criterion for showing that certain subsets of~$Y(K_v)$ are finite in the case that~$Y$ is a curve.

\begin{cor}[cf.\ {\cite[Lemma~3.3]{LVinventiones}}]\label{cor:non-density_gives_finiteness}
	Suppose that~$Y$ is a smooth curve over a finite extension~$K_v$ of~$\bQ_p$, and that $X\to Y'\to Y$ is an abelian-by-finite family with full monodromy with respect to some complex embedding $\iota\colon K_v\hookrightarrow\bC$. Let~$y_0\in Y(K_v)$ and let $\Nbd_{y_0}\subseteq Y^\an$ be an admissible open neighbourhood of~$y_0$, isomorphic to a closed disc, over which~$\cH^0_\deR(X/Y)^\an$ and~$\cH^1_\deR(X/Y)^\an$ have a full basis of horizontal sections.
	
	Suppose that $C_0\subseteq\Nbd_{y_0}(K_v)$ is a subset such that $\Phi_{y_0}(C_0)$ is not Zariski-dense in~$\dH_{y_0}$. Then~$C_0$ is finite.
\end{cor}

\begin{rmk}
	If $C\subseteq Y(K_v)$ is a subset for which $\Phi_{y_0}(C\cap\Nbd_{y_0}(K_v))$ is not Zariski-dense in~$\dH_{y_0}$ for all choices of point~$y_0$ and neighbourhood~$\Nbd_{y_0}$, then the whole set~$C$ must be finite by Corollary~\ref{cor:non-density_gives_finiteness} and the fact that~$Y(K_v)$ is covered by a finite number of the neighbourhoods $\Nbd_{y_0}(K_v)$ by compactness. This is how we will apply Corollary~\ref{cor:non-density_gives_finiteness} in practice.
\end{rmk}

\begin{proof}[Proof of Corollary~\ref{cor:non-density_gives_finiteness}]
	Let $Z\subsetneq\dH_{y_0}$ be a proper Zariski-closed subvariety containing $\Phi_{y_0}(C_0)$. Then Lemma~\ref{lem:full_monodromy_gives_dense_image} ensures that $Z^\an$ does not contain the image of $\Phi_{y_0}$. This implies that $\Phi_{y_0}^{-1}(Z^\an)\subsetneq\Nbd_{y_0}$ is a closed analytic proper subspace of~$\Nbd_{y_0}$, i.e.\ is the vanishing locus of a non-zero coherent sheaf of ideals in $\cO_{\Nbd_{y_0}}$. But since~$\Nbd_{y_0}$ is a disc, its only closed analytic proper subspaces are finite.
\end{proof}

%\alex{28/5/21}{I'm not sure how much to say about this final line: the only closed analytic proper subspaces of a closed disc are finite. To do this super-carefully, you first say that in fact a closed analytic subspace of a closed disc is the vanishing locus of an ideal of the Tate algebra as in \cite[Definition~4.5.6]{fresnel-van_der_put:rigid_analytic_geometry}, and then use Weierstrass preparation to deduce that the vanishing locus of a non-zero ideal is finite.}
\section{The Lawrence--Venkatesh locus}\label{s:locus}

Now we come to the heart of the paper. Let us fix a smooth projective curve~$Y$ of genus~$g \geq2$ over a number field~$K$ and a $p$-adic place~$v$ of~$K$. Inspired by obstruction theory, we make the following definition.

\begin{defi}\label{def:global_pair}
	Let $X\to Y'\to Y$ be an abelian-by-finite family, and let $p$ be a prime number. We define the \emph{adelic Lawrence--Venkatesh locus (with $p$-adic coefficients)}
	\[
	Y(\bA_K)^\LV_X \subseteq Y(\bA_K)
	\]
	to be the set of adelic points $(y_u)_u\in Y(\bA_K)$ for which there exists a symplectic pair 
	$(A,V)$ in the category of $G_K$-representations  with $\bQ_p$-coefficients such that
	\[
	\rH^{\leq1}_\et(X_{y_u,\Kbar_u},\bQ_p) \cong (A|_{G_u},V|_{G_u})
	\]
	as symplectic pairs in the category of $G_u$-representations for every place~$u$ of~$K$. 
	We say that the pair~$(A,V)$ \emph{interpolates} the local points~$y_u\in Y(K_u)$. For a $p$-adic place $v$ of~$K$, we define
	\[
	Y(K_v)_X^\LV\subseteq Y(K_v)
	\]
	to be the projection of $Y(\bA_K)_X^\LV$ to $Y(K_v)$.
\end{defi}

Although this is the most natural definition of an obstruction locus associated to an abelian-by-finite family, for our purposes it is more convenient to work with a slightly larger locus by relaxing the local conditions imposed on the interpolating pair~$(A,V)$. We recall the definition from the introduction.

\begin{defi}\label{def:s-good}
	Let~$S$ be a finite set of places of~$K$, and~$p$ a prime number. A symplectic pair~$(A,V)$ in the category of $G_K$-representations is called \emph{$S$-good} just when:
	\begin{itemize}
		\item $A$ is unramified outside~$S$;
		\item $V$ is unramified, pure and integral of weight~$1$ outside~$S$; and
		\item $V$ is de Rham at all places over~$p$, with Hodge--Tate weights in~$\{0,1\}$.
	\end{itemize}
	
	Let $X\to Y'\to Y$ be an abelian-by-finite family, and let~$S$ be a finite set of places of~$K$ containing all places dividing~$p\infty$ and all places of bad reduction for~$X\to Y$. Let~$v$ be a finite place of~$K$, lying over the rational prime~$p$. We define the \emph{Lawrence--Venkatesh locus}
	\[
	Y(K_v)_{X,S}^\LV \subseteq Y(K_v)
	\]
	to be the set of points~$y_v\in Y(K_v)$ for which there exists an $S$-good pair~$(A,V)$ such that
	\[
	\rH^{\leq1}_\et(X_{y_v,\Kbar_v},\bQ_p) \cong (A|_{G_v},V|_{G_v})
	\]
	as symplectic pairs in the category of $G_v$-representations. We say that the pair~$(A,V)$ \emph{interpolates} the point~$y_v\in Y(K_v)$.
\end{defi}

\begin{rmk}
	One could also add the restriction that~$A$ is pure and integral of weight~$0$ (i.e.\ an Artin representation) in the above definition. This does not change the locus~$Y(K_v)_{X,S}^\LV$, since if~$(A,V)$ interpolates~$y_v\in Y(K_v)$, then~$A$ is isomorphic to~$\bQ_p^n$ as an algebra and so has finite automorphism group.
\end{rmk}

The relationship between these two loci is as follows.

\begin{lem}\label{lem:properties_of_associated_pairs}
	Let~$X\to Y'\to Y$ be an abelian-by-finite family. Let $p$ be a prime and let~$S$ be a finite set of places of~$K$, containing all places dividing~$p\infty$ and all places of bad reduction\footnote{A finite place of~$K$ is of \emph{good reduction} for~$X\to Y'\to Y$ just when it extends to an abelian-by-finite family over a smooth proper model of~$Y$ over the ring of integers at that place.} for $X\to Y'\to Y$. Suppose that~$y=(y_u)_u\in Y(\bA_K)_X^\LV$ is an adelic point in the Lawrence--Venkatesh locus, interpolated by a symplectic pair~$(A,V)$ in the category of $G_K$-representations. Then the pair~$(A,V)$ is $S$-good in the sense of Definition~\ref{def:intro_locus}.
	
	In particular we have the containment
	\[
	Y(K_v)_X^\LV \subseteq Y(K_v)_{X,S}^\LV
	\]
	for all suitable\footnote{We will use ``suitable'' as a shorthand for ``containing all places dividing~$p\infty$ and all places of bad reduction''.}~$S$, where~$Y(K_v)_{X,S}^\LV$ is as in Definition~\ref{def:intro_locus}.
	\begin{proof}
		This is merely a statement about the local representations~$\rH^i_\et(X_{y_u,\Kbar_u},\bQ_p)$ associated to local points~$y_u\in Y(K_u)$: they are unramified, pure and integral of weight~$i$ whenever~$u\notin S$ (since~$X_{y_u}$ has good reduction), and they are de Rham with Hodge--Tate weights in~$\{0,1,\dots,i\}$ whenever~$u\mid p$.
	\end{proof}
\end{lem}

Our aim in this section is to prove the following, which can be seen as an abstraction and generalisation of the argument in~\cite{LVinventiones}.

\begin{thm}\label{thm:locus_finiteness}
	Suppose that the place~$v$ is self-conjugate. Then there exists an abelian-by-finite family $X\to Y'\to Y$ for which the Lawrence--Venkatesh locus~$Y(K_v)^\LV_{X,S}$ is finite for all suitable~$S$.
\end{thm}

This theorem already implies the Mordell Conjecture, as in \cite{LVinventiones}, once one makes the almost tautological observation that 
\[
Y(K)\subseteq Y(K_v)^\LV_{X,S}
\]
for all~$X\to Y'\to Y$. Indeed, for $y\in Y(K)$, the interpolating pair $(A,V)$
can be taken to be 
\[
\rH^{\leq 1}_{\et}(X_{y,\Kbar},\bQ_p)  = (\rH^0_\et(X_{y,\Kbar},\bQ_p),\rH^1_\et(X_{y,\Kbar},\bQ_p))
\]
with its usual $G_K$-action, which is $S$-good by Lemma~\ref{lem:properties_of_associated_pairs}.

\subsection{The Principal Trichotomy}\label{ss:principal_trichotomy}

The starting point for the argument is a certain restriction on global representations. Before we get to this, let us describe an idea which doesn't work, but nonetheless informs the general approach, as sketched in the introduction.

Suppose we knew that for every pair $(A,V)$ satisfying the conditions of Definition~\ref{def:global_pair}, the $A$-module~$V$ had no non-zero $G_K$-stable isotropic $A$-submodules. One could then show, using the symplectic version of Faltings' Lemma (Lemma~\ref{lem:faltings_symplectic}), that there are only finitely many possibilities for the pair~$(A,V)$ up to isomorphism, once one specifies the dimensions of~$A$ and~$V$. This would imply that there are only finitely many possibilities for the isomorphism class of 
$\rH^{\leq 1}_{\et}(X_{y,\Kbar_v},\bQ_p)$ for~$y\in Y(K_v)^\LV_{X,S}$. To prove finiteness we may therefore restrict to such $y$ with $\rH^{\leq 1}_{\et}(X_{y,\Kbar_v},\bQ_p)$ in a fixed isomorphism class.

Now we restrict attention to points in a suitably small analytic neighbourhood $\Nbd_{y_0}$ of some point~$y_0$, and assume for simplicity of exposition that the fibre~$Y'_{y_0}$ is connected (so isomorphic to $\Spec(L_w)$ for a finite extension $L_w/K_v$ and~$X_{y_0}$ is a polarised abelian variety over~$L_w$). For 
\[
y\in Y(K_v)^\LV_{X,S}\cap\Nbd_{y_0}(K_v),
\]
the fact that~$\rH^{\leq1}_\et(X_{y,\Kbar_v},\bQ_p)$ lies in fixed isomorphism class implies that the image $\Phi_{y_0}(y)\in\dH_{y_0}(K_v)$ under the period map would have to lie in a single orbit under the action of the automorphism group 
\[
\uAut_{(\varphi,N,G_w)}^\GSp(\rH^1_\pst(X_{y_0}/\bQ_p^\nr)) \,.
\]
See \S\ref{sss:phi-N_automorphisms}, especially Remark~\ref{rmk:induced_autos}. 

Since the scalars in $\uAut_{(\varphi,N,G_w)}^\GSp(\rH^1_\pst(X_{y_0}/\bQ_p^\nr))$ act trivially on~$\Res^{K_v}_{\bQ_p}\dH_{y_0}$, the dimension estimate of Proposition~\ref{prop:dim_of_aut}\eqref{proppart:dim_of_aut_symplectic} would show that~$\Phi_{y_0}(y)$ would lie in a closed $\bQ_p$-subvariety of~$\Res^{K_v}_{\bQ_p}\dH$ of $\bQ_p$-dimension $\leq d(2d+1)$, where $d$ is the relative dimension of~$X\to Y$. By the base change--Weil restriction adjunction, this implies that~$\Phi_{y_0}(y)$ would lie in a $K_v$-subvariety of~$\dH$ of $K_v$-dimension at most~$d(2d+1)$. So if $[L_w:K_v]\geq4$, then we would have
\[
d(2d+1) < [L_w:K_v]\cdot\frac{d(d+1)}2 = \dim_{K_v}\dH_{y_0} \,,
\]
and so~$\Phi_{y_0}(y)$ would have to lie in a proper Zariski-closed $K_v$-subvariety of~$\dH_{y_0}$. If moreover our abelian-by-finite family had full monodromy, we would obtain finiteness of $Y(K_v)^\LV_{X,S}$ via Corollary~\ref{cor:non-density_gives_finiteness}.
\smallskip

However, this approach fails as stated, since one has no guarantee that pairs~$(A,V)$ satisfying the conditions of Definition~\ref{def:global_pair} have no non-zero $G_K$-stable isotropic $A$-submodules in~$V$. The key idea in \cite{LVinventiones} is that the numerics of self-conjugate places still impose (rather technical) constraints on the possible representations which can appear, and that these constraints still suffice to prove finiteness results.

This is what we state and prove here, following \cite{LVinventiones}. Before we give a precise statement, we introduce some notation. Suppose that~$A$ is an algebra in the category of $G_K$-representations (resp.\ $G_v$-representations). Since $A$ is an artinian $\bQ_p$-algebra, the set 
\[
\Sigma_A \colonequals \Spec(A)(\bQ_p) = \{\psi : A \to \bQ_p \ ; \ \text{$\bQ_p$-algebra homomorphism}\}
\]
is finite and $G_K$ (resp.\ $G_v$) acts continuously on this finite set from the right. For $\psi \in \Sigma_A$, we write
\[
G_\psi\leq G_K \hspace{0.4cm}(\,\text{resp.}\hspace{0.4cm} G_{w_\psi}\leq G_v \,)
\]
for the stabiliser of~$\psi$. The fixed fields under $G_\psi$ and $G_{w_\psi}$ are denoted by $L_\psi\subset\Kbar$ and $L_{w_\psi}\subset\Kbar_v$, respectively. When~$A$ is an algebra in $G_K$-representations, we may view it as an algebra in $G_v$-representations in a natural way: we then have $G_{w_\psi}=G_v\cap G_\psi$, and $L_{w_\psi} = (L_\psi)_{w_\psi}$ is the completion of~$L_\psi$ at a certain $v$-adic place~$w_\psi$ of~$L_\psi$, namely the restriction of the $v$-adic place on~$\Kbar_v\supseteq\Kbar\supseteq L_\psi$. So the fields~$L_\psi$ and~$L_{w_\psi}$ fit together in the following diagram.
\begin{center}
\begin{tikzcd}
	K_v \arrow[hook]{r} & L_{w_\psi} \arrow[hook]{r} & \Kbar_v \\
	K \arrow[hook]{u}\arrow[hook]{r} & L_\psi \arrow[hook]{u}\arrow[hook]{r} & \Kbar \arrow[hook]{u}
\end{tikzcd}
\end{center}

For an~$A$-module~$V$ in the category of $G_K$- or~$G_v$-representations, we write
\[
V_\psi \colonequals \bQ_p \otimes_{A,\psi} V
\]
for the base-change of~$V$ along $\psi\colon A\to\bQ_p$. This is a symplectic representation of~$G_\psi$ or~$G_{w_\psi}$, respectively.

\begin{rmk}\label{rmk:what_is_Sigma}
	In the particular case of the symplectic pair $\rH^{\leq 1}_\et(X_{y_v,\Kbar_v}, \bQ_p)$ associated to an abelian-by-finite family and a $K_v$-rational point~$y_v\in Y(K_v)$, we have 
	\[
	\rH^0_\et(X_{y_v,\Kbar_v},\bQ_p)=\rH^0_\et(Y'_{y_v,\Kbar_v},\bQ_p) = \prod_{\pi_0(Y'_{y_v,\Kbar_v})} \bQ_p= \prod_{\psi \in Y'_{y_v}(\Kbar_v)} \bQ_p \,.
	\]
Thus there is a canonical $G_v$-equivariant bijection
\[
\{\psi :  \rH^0_\et(X_{y_v,\Kbar_v},\bQ_p) \to \bQ_p \ ; \ \text{$\bQ_p$-algebra homomorphism}\}
 = Y'_{y_v}(\Kbar_v) \,.
\]
So a $\bQ_p$-algebra homomorphism $\psi$ corresponds to a choice of a closed point $y'_v$ of $Y'_{y_v}$ and a $K_v$-embedding $K_v(y'_v)\hookrightarrow\Kbar_v$. The condition that $[G_v:G_{w_\psi}]\geq4$ just means that the 
degree~$[K_v(y'_v):K_v]$ is at least $4$.
\end{rmk}

\begin{prop}[The Principal Trichotomy, cf.\ {\cite[Sublemma on p.~936]{LVinventiones}}]\label{prop:principal_trichotomy}
Suppose that~$v$ is self-conjugate. 	

	Let~$(A,V)$ be a symplectic pair in the category of $G_K$-representations, where~$V$ has rank~$2d>0$ as an $A$-module. Suppose moreover that $V$ is de Rham at places above~$p$, has Hodge--Tate weights in~$\{0,1\}$ at~$v$, and is unramified and pure of weight~$1$ outside a finite set.

	Let~$\Sigma = \Sigma_A$ be the set of $\bQ_p$-algebra homomorphisms $A\to\bQ_p$. Then at least one of the following occurs:
	\begin{enumerate}[label=\alph*),ref=(\alph*)]
		\item\label{proppart:acase} there is a $\psi \in \Sigma$ such that $[G_v:G_{w_\psi}]\geq4$ and $V_\psi$ is $\GSp$-irreducible as a symplectic representation of~$G_\psi\leq G_K$;
		\item\label{proppart:bcase} there is a $\psi \in \Sigma$ such that $[G_v:G_{w_\psi}]\geq4$ and there exists a non-zero isotropic $G_\psi$-subrepresenta\-tion $0\neq W\leq V_\psi$ whose average Hodge--Tate weight at~$w_\psi$ is~$\geq1/2$, i.e.
		\[
		\dim_{L_{w_\psi}}\rF^1\!\sD_{\deR,w_\psi}(W) \geq \frac12\dim_{\bQ_p}(W) \,; \hspace{0.4cm}\text{or}
		\]
		\item\label{proppart:ccase} the number of $\psi \in \Sigma$ satisfying $[G_v:G_{w_\psi}]<4$ is $\geq\frac1{d+1}\dim_{\bQ_p}(A)$.
	\end{enumerate}
\end{prop}

\begin{proof}
	We begin with a preliminary reduction. Let us write~$\Sigma=\coprod_i\Sigma_i$ as the union of its $G_K$-orbits. We write
	\[
	A_i= \prod_{\psi \in \Sigma_i} \bQ_p,
	\]
	and write~$V_i\colonequals A_i\otimes_AV$ for the base-change of $V$ along the $G_K$-equivariant map $A \surj A_i$ whose components are the $\psi \in \Sigma_i = \Sigma_{A_i}$. It is easy to see that $(A,V)$ satisfies condition~\ref{proppart:acase} (respectively~\ref{proppart:bcase}) if and only if some $(A_i,V_i)$ does. Moreover, if all $A_i$ satisfied condition~\ref{proppart:ccase}, then the proportion of each $\Sigma_i$ contained in a $G_v$-orbit of size~$<4$ would be $\geq\frac1{d+1}$, and so the proportion of $\Sigma$ contained in a $G_v$-orbit of size~$<4$ would be $\geq\frac1{d+1}$. So $A$ would also satisfy~\ref{proppart:ccase} because $\#\Sigma \leq \dim_{\bQ_p}(A)$. Hence it suffices to prove the result for the pairs~$(A_i,V_i)$: if any of them satisfy~\ref{proppart:acase} or~\ref{proppart:bcase} then so does~$(A,V)$, and if all of them satisfy~\ref{proppart:ccase}, then so does~$(A,V)$.
	
	\smallskip
	
	We thus restrict attention to the case that~$A=\prod_{\psi\in\Sigma}\bQ_p$ is a $\bQ_p$-split Artinian algebra for a transitive $G_K$-set~$\Sigma$ (which is also the set $\Sigma_A$ of $\bQ_p$-algebra homomorphisms $A\to\bQ_p$). The natural map 
	\[
	V \xrightarrow{\sim} \prod_{\psi\in\Sigma}V_{\psi}
	\]
	 is an isomorphism by the structure theory of modules over Artinian rings. The action of some $\sigma\in G_K$ maps the factor $V_{\psi}$ isomorphically onto the factor $V_{\psi\cdot\sigma}$. In other words, if we fix some $\psi_0\in\Sigma$, then $V=\Ind_{G_{\psi_0}}^{G_K}V_{\psi_0}$ as $G_K$-equivariant symplectic modules over~$A=\Ind_{G_{\psi_0}}^{G_K}\bQ_p$.
	
	Suppose now that neither~\ref{proppart:acase} nor~\ref{proppart:bcase} holds; we will show~\ref{proppart:ccase}. If $\Sigma$ contains no $G_v$-orbit of size~$\geq4$, then~\ref{proppart:ccase} certainly holds, so we may assume without loss of generality that~$[G_v:G_{w_{\psi_0}}]\geq4$. Failure of~\ref{proppart:acase} implies that $V_{\psi_0}$ must fail to be $\GSp$-irreducible, so there exists a non-zero $G_{\psi_0}$-stable isotropic subspace $W_0\leq V_{\psi_0}$. For every~$\sigma\in G_K$, $\sigma^{-1}(W_0)$ is an isotropic subspace of $V_{\psi_0\cdot\sigma}$, stable under the action of $G_{\psi_0\cdot\sigma}=\sigma^{-1}G_{\psi_0}\sigma$. The subrepresentation~$\sigma^{-1}(W_0)$ is automatically de Rham at~$w_{\psi_0\cdot\sigma}$ with Hodge--Tate weights in~$\{0,1\}$, so by failure of~\ref{proppart:bcase}, we must have
	\begin{equation}\label{eq:no_b_bound}\tag{$\ast$}
	\dim_{L_{w_{\psi_0\cdot\sigma}}}\rF^1\!\sD_{\deR,w_{\psi_0\cdot\sigma}}(\sigma^{-1}(W_0)) \leq \frac12(\dim_{\bQ_p}(W_0)-1)
	\end{equation}
	whenever~$\psi_0\cdot\sigma$ is contained in a $G_v$-orbit of size~$\geq4$.
	
	Now the place $w_{\psi_0\cdot\sigma}$ can be viewed as a place of~$L\colonequals L_{\psi_0}$ under the identification $L_{\psi_0\cdot\sigma}\xrightarrow\sim L_{\psi_0}$ given by the action of~$\sigma$. If we let~$\sigma$ run over a set of right coset representatives of~$G_{\psi_0}\leq G_K$, then we obtain every $v$-adic place of~$L$ in this way, with the place~$w$ appearing~$[L_w\colon K_v]$ times. Hence, by summing~\eqref{eq:no_b_bound} and the trivial bound $\dim_{L_w}\rF^1\!\sD_{\deR,w}(W_0)\leq\dim_{\bQ_p}(W_0)$, we obtain the inequality
	\[
	\sum_{w\mid v}[L_w:K_v]\dim_{L_w}\rF^1\!\sD_{\deR,w}(W_0) \leq \frac12(\dim_{\bQ_p}(W_0)-1)\cdot(\#\Sigma-\#\Sigma_{<4})+\dim_{\bQ_p}(W_0)\cdot\#\Sigma_{<4} \,,
	\]
	where~$\Sigma_{<4}$ denotes the subset of $\Sigma$ consisting of the $G_v$-orbits of size~$<4$. But since $W_0$, being a subquotient of~$V$, has Hodge--Tate weights in $\{0,1\}$, Corollary~\ref{cor:friendly_representations} gives us that
	\[
	\sum_{w\mid v}[L_w:K_v]\dim_{L_w}\rF^1\!\sD_{\deR,w}(W_0) = \frac12\dim_{\bQ_p}(W_0)\cdot\#\Sigma \,.
	\]
	Equating and rearranging, we obtain
	\[
	(\dim_{\bQ_p}(W_0)+1)\cdot\#\Sigma_{<4}\geq\#\Sigma=\dim_{\bQ_p}(A) \,.
	\]
	Since~$W_0$ is an isotropic subspace of the $2d$-dimensional symplectic vector space~$V_\psi$, we have the estimate $\dim_{\bQ_p}(W_0)\leq d$ and we have shown~\ref{proppart:ccase}.
\end{proof}

The principal trichotomy allows us to decompose the Lawrence--Venkatesh locus~$Y(K_v)_{X,S}^\LV$ into three sets, according to which of the conditions~\acase, \bcase or~\ccase is satisfied for the interpolating pair~$(A,V)$. Let us make a precise definition.

\begin{defi}\label{def:types}
	Let~$S$ be a finite set of places of~$K$, ~$p$ a prime number, and~$v$ a $p$-adic place of~$K$.
	\begin{enumerate}[label=\alph*),ref=(\alph*)]
		\item We say that an $S$-good symplectic pair~$(A,V)$ in the category of $G_K$-representations is of \emph{type~\acase} just when there exists a $\psi\in\Sigma_A$ such that~$[G_v:G_{w_\psi}]\geq4$ and~$V_\psi$ is $\GSp$-irreducible as a symplectic representation of~$G_\psi\leq G_K$.
		\item We say that a symplectic pair~$(A,V)$ in the category of $G_v$-representations is of \emph{type~\bcase} just when there exists a $\psi\in\Sigma_A$ such that $[G_v:G_{w_\psi}]\geq4$ and there exists a non-zero isotropic $G_{w_\psi}$-subrepresentation $0\neq W\leq V_\psi$ whose average Hodge--Tate weight is~$\geq1/2$.
		\item Let~$d>0$ be a positive integer. We say that an algebra~$A$ in the category of~$G_v$-representations is of \emph{type~\ccase} just when the number of $\psi\in\Sigma_A$ such that $[G_v:G_{w_\psi}]<4$ is~$\geq\frac1{d+1}\dim_{\bQ_p}(A)$.
	\end{enumerate}
	
	If~$X\to Y'\to Y$ is an abelian-by-finite family of constant relative dimension~$d>0$, and if~$S$ is a suitable set of places of~$K$, then we define three subsets
	\[
	Y(K_v)_{X,S,\acase}^\LV , Y(K_v)_{X,\bcase}^\LV , Y(K_v)_{X,\ccase}^\LV \subseteq Y(K_v)
	\]
	as follows.
	\begin{enumerate}[label=\alph*),ref=(\alph*)]
		\item $Y(K_v)_{X,S,\acase}^\LV\subseteq Y(K_v)$ is the set of all points $y_v\in Y(K_v)$ for which there exists an $S$-good pair~$(A,V)$ of type~\acase such that
		\[
		\rH^{\leq1}_\et(X_{y_v,\Kbar_v},\bQ_p) \cong (A|_{G_v},V|_{G_v})
		\]
		as symplectic pairs in the category of~$G_v$-representations.
		\item $Y(K_v)_{X,\bcase}^\LV$ is the set of all points $y_v\in Y(K_v)$ for which $\rH^{\leq1}_\et(X_{y_v,\Kbar_v},\bQ_p)$ is of type~\bcase.
		\item $Y(K_v)_{X,\ccase}^\LV$ is the set of all points $y_v\in Y(K_v)$ for which $\rH^0_\et(X_{y_v,\Kbar_v},\bQ_p)$ is of type~\ccase.
	\end{enumerate}
\end{defi}

With this notation, Proposition~\ref{prop:principal_trichotomy} implies that when~$v$ is self-conjugate, then for any abelian-by-finite family $X \to Y' \to Y$ of constant relative dimension~$d>0$ we have the containment
\[
Y(K_v)_{X,S}^\LV \subseteq Y(K_v)_{X,S,\acase}^\LV \cup Y(K_v)_{X,\bcase}^\LV \cup Y(K_v)_{X,\ccase}^\LV \,.
\]
So, to prove Theorem~\ref{thm:locus_finiteness}, it suffices to study the sets $Y(K_v)_{X,S\acase}^\LV$, $Y(K_v)_{X,\bcase}^\LV$ and $Y(K_v)_{X,\ccase}^\LV$ separately. Specifically, we will prove:
\begin{itemize}
	\item if~$X\to Y'\to Y$ has full monodromy, then $Y(K_v)_{X,S\acase}^\LV$ and $Y(K_v)_{X,\bcase}^\LV$ are finite; and
	\item there exists a choice of $X\to Y'\to Y$ which has full monodromy and for which $Y(K_v)_{X,\ccase}^\LV=\emptyset$.
\end{itemize}
These together suffice to prove Theorem~\ref{thm:locus_finiteness}.

\subsection{Finiteness for points of type~\acase}\label{ss:type_a}

To begin with, we deal with the set~$Y(K_v)_{X,S,\acase}^\LV$ using the theory of period maps. For the duration of \S\ref{ss:type_a}, we fix an abelian-by-finite family $X\to Y'\to Y$ and fix a choice of finite set~$S$ of places of~$K$, containing all places dividing~$p\infty$ and all places of bad reduction for~$X\to Y'\to Y$. Our aim here is to prove the following.

\begin{prop}\label{prop:non-density_a}
	For a point~$y_0\in Y(K_v)$, let
	\[
	\dH_{y_0}(K_v)_{X,S,\acase}^\LV\subseteq\dH_{y_0}(K_v)
	\]
	denote the set of Lagrangian $\rH^0_\deR(X_{y_0}/K_v)$-submodules $\Phi$ of $\rH^1_\deR(X_{y_0}/K_v)$ for which there exists an $S$-good symplectic pair of type~\acase in the category of~$G_K$-representations such that
	\[
	\MM^{\leq 1}(\Phi) \cong (\sD_\pH(A|_{G_v}),\sD_\pH(V|_{G_v}))
	\]
	as symplectic pairs in the category of filtered discrete $(\varphi,N,G_v)$-modules.
	
	Then $\dH_{y_0}(K_v)_{X,S,\acase}^\LV$ is not Zariski-dense in~$\dH_{y_0}$.
\end{prop}

\begin{cor}\label{cor:finiteness_a}
	If~$X\to Y'\to Y$ has full monodromy, then $Y(K_v)_{X,S,\acase}^\LV$ is finite.
	\begin{proof}[Proof of Corollary~\ref{cor:finiteness_a}]
		Choose a point $y_0\in Y(K_v)$ and let~$U_{y_0}\subseteq Y_{K_v}^\an$ be an admissible open neighbourhood isomorphic to a closed disc, over which~$\cH^0_\deR(X_{K_v}/Y_{K_v})^\an$ and~$\cH^1_\deR(X_{K_v}/Y_{K_v})^\an$ have a full basis of horizontal sections. It follows from Proposition~\ref{prop:period_maps_control_abf_reps} that the image of~$Y(K_v)_{X,S,\acase}^\LV\cap U_{y_0}(K_v)$ under the period map~$\Phi_{y_0}$ lies in~$\dH_{y_0}(K_v)_{X,S,\acase}^\LV$, and hence~$Y(K_v)_{X,S,\acase}^\LV\cap U_{y_0}(K_v)$ is finite by Corollary~\ref{cor:non-density_gives_finiteness}. By compactness this implies that~$Y(K_v)_{X,S,\acase}^\LV$ is finite.
	\end{proof}
\end{cor}

For the proof, let us enumerate the closed points of~$Y'_{y_0}$ as $(y_i')_{i\in I}$ as in~\S\ref{sss:period_map_decomposition} and write~$L_{w_i}=K_v(y_i')$. Recall that any Lagrangian~$\rH^0_\deR(X_{y_0}/K_v)$-submodule~$\Phi$ of~$\rH^1_\deR(X_{y_0}/K_v)$ factors as~$\prod_i\Phi_i'$ where each~$\Phi_i'$ is a Lagrangian $L_{w_i}$-subspace of~$\rH^1_\deR(X_{y_i'}/L_{w_i})$, and that we then have by Lemma~\ref{lem:M_decomposition} a decomposition
\[
\MM^1(\Phi) = \prod_i\Ind_{G_{w_i}}^{G_v}\MM^1_i(\Phi_i')
\]
in the category of filtered discrete~$(\varphi,N,G_v)$-modules, compatible with symplectic structures. We examine the possibilities for the isomorphism classes of the~$\MM^1_i(\Phi_i')$.

\begin{defi}\label{def:T_representations_a}
	For an index~$i$, let~$d_i\colonequals\dim_{L_{w_i}}X_{y_i'}$ and write~$\delta\colonequals\deg(Y'\to Y)$. We write $T_i$ for the set of isomorphism classes of symplectic $\unit$-modules of rank~$2d_i$ in the category of filtered discrete $(\varphi,N,G_{w_i})$-modules~$D$ for which there exists:
	\begin{itemize}
		\item a finite extension $L/K$ of degree~$\leq\delta$, unramified outside~$S$;
		\item a $v$-adic place~$w_i$ of~$L$; and
		\item a $\GSp$-irreducible symplectic representation~$V$ of~$G_L$, de Rham at places of~$L$ over~$p$ and unramified, pure and integral of weight~$1$ outside places of~$L$ above~$S$
	\end{itemize}
	such that $L_{w_i}$ is $K_v$-isomorphic to the completion of~$L$ at~$w_i$ and $D\cong\sD_{\pH,w_i}(V|_{G_{w_i}})$ as symplectic filtered discrete~$(\varphi,N,G_{w_i})$-modules.
\end{defi}

As a consequence of Hermite--Minkowski and Faltings' Lemma (Lemma~\ref{lem:faltings_symplectic}), we have the following.

\begin{lem}
	The set~$T_i$ is finite for all~$i$.
\end{lem}

Now as in \S\ref{s:abelian-by-finite} let us write
\[
\dH_{y_i'} \colonequals \LGrass(\rH^1_\deR(X_{y_i'}/L_{w_i}))
\]
for the Lagrangian Grassmannian parametrising Lagrangian $L_{w_i}$-subspaces~$\Phi_i'$ of~$\rH^1_\deR(X_{y_i'}/L_{w_i})$, and write
\[
\dH_{y_i'}(L_{w_i})_\acase\subseteq\dH_{y_i'}(L_{w_i})
\]
for the set of Lagrangian $L_{w_i}$-subspaces~$\Phi_i'$ for which $\MM^1_i(\Phi_i') \in T_i$. The key computation is as follows.

\begin{lem}[cf.\ proof of {\cite[Lemma~6.2]{LVinventiones}}]\label{lem:non-density_componentwise_a}
	If~$[L_{w_i}:K_v]\geq4$, then $\dH_{y_i'}(L_{w_i})_\acase$ is not Zariski-dense in $\Res^{L_{w_i}}_{K_v}\dH_{y_i'}$.
\end{lem}

\begin{rmk}
	It is important in Lemma~\ref{lem:non-density_componentwise_a} that we Weil-restrict down to $K_v$: we certainly make no claim as to whether $\dH_{y_i'}(L_{w_i})_\acase$ is Zariski-dense in~$\dH_{y_i'}$.
\end{rmk}

\begin{proof}[Proof of Lemma~\ref{lem:non-density_componentwise_a}]
	Since~$T_i$ is finite, the set~$\dH_{y_i'}(L_{w_i})_\acase$ is a finite union of orbits under the action of (the $\bQ_p$-points of) $\uAut^\GSp_{(\varphi,N,G_{w_i})}(\rH^1_\pst(X_{y_i'}/\bQ_p^\nr))$ on $\dH_{y_i'}(L_{w_i})$, as in Lemma~\ref{lem:describe orbits}. This action is $\bQ_p$-algebraic, i.e.\ arises from an action of the $\bQ_p$-algebraic group 
	$\uAut^\GSp_{(\varphi,N,G_{w_i})}(\rH^1_\pst(X_{y_i'}/\bQ_p^\nr))$ on $\Res^{L_{w_i}}_{\bQ_p}\dH_{y_i'}$. Since the scalars inside $\uAut^\GSp_{(\varphi,N,G_{w_i})}(\rH^1_\pst(X_{y_i'}/\bQ_p^\nr))$ act trivially, it follows from the dimension estimate of Proposition~\ref{prop:dim_of_aut}\eqref{proppart:dim_of_aut_symplectic} that $\dH_{y_i'}(L_{w_i})_\acase$ is contained in a closed $\bQ_p$-algebraic subvariety of $\Res^{L_{w_i}}_{\bQ_p}\dH_{y_i'}$ of $\bQ_p$-dimension~$\leq d_i(2d_i+1)$. By the base change--Weil restriction adjunction, this implies that $\dH_{y_i'}(L_{w_i})_\acase$ is also contained in a closed $K_v$-algebraic subvariety of $\Res^{L_{w_i}}_{K_v}\dH_{y_i'}$ of $K_v$-dimension~$\leq d_i(2d_i+1)$. Since the $K_v$-dimension of $\Res^{L_{w_i}}_{K_v}\dH_{y_i'}$ is 
	\[
	\dim_{K_v} \Res^{L_{w_i}}_{K_v}\dH_{y_i'} = [L_{w_i}:K_v]\cdot\frac{d_i(d_i+1)}2>d_i(2d_i+1) \,,
	\]
	we have non-density for dimension reasons.
\end{proof}

\begin{proof}[Proof of Proposition~\ref{prop:non-density_a}]
	Consider the decomposition
	\[
	\dH_{y_0} = \prod_i\Res^{L_{w_i}}_{K_v}\dH_{y_i'}
	\]
	from~\eqref{eq:grassmannian_decomposition}. We claim that for every~$\Phi\in\dH_{y_0}(K_v)_{X,S,\acase}^\LV$ there exists an index~$i$ such that~$[G_v:G_{w_i}]\geq4$ and the~$i$th component~$\Phi_i'$ of~$\Phi$ lies in~$\dH_{y_i'}(L_{w_i})_\acase$. The proposition then follows from Lemma~\ref{lem:non-density_componentwise_a}.
	
	Proving the claim is a matter of unwinding definitions. Given some~$\Phi\in\dH_{y_0}(K_v)_{X,S,\acase}^\LV$, there is by definition an $S$-good symplectic pair~$(A,V)$ of type~\acase in the category of $G_K$-representations and an isomorphism
	\begin{equation}\label{eq:interpolating_pair_a}\tag{$\ast$}
	\MM^{\leq1}(\Phi) \cong (\sD_\pH(A|_{G_v}),\sD_\pH(V|_{G_v}))
	\end{equation}
	of symplectic pairs in filtered~$(\varphi,N,G_v)$-modules. Recall that being of type~\acase means that there exists some~$\psi\in\Sigma_A$ such that~$[G_v:G_{w_\psi}]\geq4$ and~$V_\psi$ is $\GSp$-irreducible. Since~$\sD_\pH$ is a fully faithful $\otimes$-functor, the algebra part of~\eqref{eq:interpolating_pair_a} is~$\sD_\pH$ of an isomorphism
	\[
	\rH^0_\et(X_{y_0,\Kbar_v},\bQ_p) \cong A|_{G_v} \,,
	\]
	so the element~$\psi\in\Sigma_A$ corresponds to a $\bQ_p$-algebra homomorphism
	\[
	\psi_i\colon\rH^0_\et(X_{y_0,\Kbar_v},\bQ_p) \to \bQ_p \,.
	\]
	Such a map is automatically the map induced by a $\Kbar_v$-point of $Y'_{y_0}$, i.e.\ a pair of a closed point $y_i'\in|Y'_{y_0}|$ and a $K_v$-embedding $K_v(y_i')\hookrightarrow\Kbar_v$. Replacing~$\psi$ by another element of~$\Sigma_A$ in its~$G_v$-orbit, we may assume without loss of generality that this embedding is the embedding $L_{w_i}\hookrightarrow\Kbar_v$ fixed in~\S\ref{sss:period_map_decomposition}.
	
	For this value of~$i$, we have~$[L_{w_i}:K_v]=[G_v:G_{w_\psi}]\geq4$ by assumption. We also have
	\[
	\sD_{\pH,w_\psi}(V_\psi|_{G_{w_\psi}}) = \sD_{\pH,w_\psi}(\bQ_p\otimes_{A,\psi}V) \cong \unit\otimes_{M^0|_{G_{w_i}},\sD_{\pH,w_i}(\psi_i)}\MM^1(\Phi)|_{G_{w_i}} = \MM^1_i(\Phi_i')
	\]
	as symplectic $\unit$-modules in the category of filtered discrete~$(\varphi,N,G_{w_i})$-modules, using Lemma~\ref{lem:M_decomposition} for the final identification. The fact that~$(A,V)$ is $S$-good of type~\acase implies that the number field~$L_\psi$, place~$w_\psi$ and symplectic representation~$V_\psi$ satisfy the conditions of Definition~\ref{def:T_representations_a}, and hence we have $\MM^1_i(\Phi_i')\in\dH_{y_i'}(L_{w_i})_\acase$. This completes the proof of the claim, and hence of the proposition.
\end{proof}

\subsection{Finiteness for points of type~\bcase}\label{ss:type_b}

Next, we deal with the set~$Y(K_v)_{X,\bcase}^\LV$. Again, this goes via the theory of period maps, but rather than using Faltings' Lemma and a dimension estimate, we instead use a more explicit argument. Throughout \S\ref{ss:type_b}, we again fix an abelian-by-finite family $X\to Y'\to Y$ and a $K_v$-point~$y_0\in Y(K_v)$. We define
\[
\dH_{y_0}(K_v)_{X,\bcase}^\LV\subseteq\dH_{y_0}(K_v)
\]
to be the set of Lagrangian $\rH^0_\deR(X_{y_0}/K_v)$-submodules~$\Phi$ of~$\rH^1_\deR(X_{y_0}/K_v)$ such that
\[
\MM^1(\Phi) \cong (\sD_\pH(A),\sD_\pH(V))
\]
for some symplectic pair~$(A,V)$ in the category of de Rham $G_v$-representations of type~\bcase. Our aim is to prove the following counterpoint to Proposition~\ref{prop:non-density_a}.

\begin{prop}\label{prop:non-density_b}
	$\dH_{y_0}(K_v)_{X,\bcase}^\LV$ is not Zariski-dense in~$\dH_{y_0}$.
\end{prop}

\begin{cor}\label{cor:finiteness_b}
	If~$X\to Y'\to Y$ has full monodromy, then $Y(K_v)_{X,\bcase}^\LV$ is finite.
\end{cor}

Again, the decomposition of the fibre~$X_{y_0}$ reduces this to a simpler computation. For a closed point~$y_i'\in|Y'_{y_0}|$, we write
\[
\dH_{y_i'}(L_{w_i})_\bcase\subseteq\dH_{y_i'}(L_{w_i})
\]
for the set of Lagrangian $L_{w_i}$-subspaces $\Phi_i'$ of~$\rH^1_\deR(X_{y_i'}/L_{w_i})$ for which there exists a non-zero isotropic filtered discrete $(\varphi,N,G_{w_i})$-submodule $D \leq \MM^1_i(\Phi_i')$ such that
\[
\dim_{L_{w_i}}\rF^1\!D_\deR \geq \frac12\dim_{\bQ_p^\nr}D_\pst \,.
\]
We say that~$D$ has average Hodge--Tate weight~$\geq1/2$.

\begin{lem}[cf.\ {\cite[Lemma~6.3]{LVinventiones}}]\label{lem:non-density_componentwise_b}
	If $[L_{w_i}:K_v]\geq4$, then $\dH_{y_i'}(L_{w_i})_\bcase$ is not Zariski-dense in $\Res^{L_{w_i}}_{K_v}\dH_{y_i'}$.
	\begin{proof}
		It suffices to show that $\dH_{y_i'}(L_{w_i})_\bcase$ is not Zariski-dense in the base-change of $\Res^{L_{w_i}}_{K_v}\dH_{y_i'}$ from~$K_v$ to~$\Kbar_v$. We enumerate the $K_v$-embeddings $L_{w_i}\hookrightarrow\Kbar_v$ as $\iota_1,\dots,\iota_r$, with $\iota_1$ the embedding chosen earlier, so that there is an isomorphism
		\[
		\Kbar_v\otimes_{K_v}L_{w_i}\cong \prod_{j=1}^r \Kbar_v
		\]
		of~$\Kbar_v$-algebras whose $j$th component is $1\otimes\iota_j$. For each~$j$, we may choose an element~$\sigma_j\in G_v$ such that $\iota_j=\sigma_j\circ\iota_1$, and we may assume moreover that~$\sigma_j$ acts on~$\bQ_p^\nr$ as an integer power~$\varphi^{m_j}$ of Frobenius. In the case~$j=1$, we may take~$\sigma_1=1$ and~$m_1=0$. Fixing these choices of~$\sigma_j$ and~$m_j$, we have a $\Kbar_v$-linear isomorphism 
		\[
		\Kbar_v\otimes_{\iota_j}\rH^1_\deR(X_{y_i'}/L_{w_i})\cong\Kbar_v\otimes_{\bQ_p^\nr}\rH^1_\pst(X_{y_i'}/\bQ_p^\nr)
		\]
		given as the composite
		\begin{center}
			\begin{tikzcd}
				\Kbar_v\otimes_{\iota_j}\rH^1_\deR(X_{y_i'}/L_{w_i})  \arrow{r}{\sigma_j^{-1}\otimes1}[swap]{\sim}   & \Kbar_v\otimes_{\iota_1}\rH^1_\deR(X_{y_i'}/L_{w_i}) 
				\arrow{d}{c_\BO^{-1}\circ c_\deR^{-1}}[swap]{\sim}     & \\
				&  \Kbar_v\otimes_{\bQ_p^\nr}\rH^1_\pst(X_{y_i'}/\bQ_p^\nr)  
				\arrow{r}{\sigma_j\otimes\varphi^{m_j}} [swap]{\sim}
				& \Kbar_v\otimes_{\bQ_p^\nr} \rH^1_\pst(X_{y_i'}/\bQ_p^\nr) \,.
			\end{tikzcd}
		\end{center}
		Taken together, we obtain an isomorphism
		\begin{equation}\label{eq:base_change_trivialisation}
			\Kbar_v\otimes_{K_v}\rH^1_\deR(X_{y_i'}/L_{w_i}) \cong \prod_{j=1}^r\left(\Kbar_v\otimes_{\iota_j}\rH^1_\deR(X_{y_i'}/L_{w_i})\right) \cong \left(\Kbar_v\otimes_{\bQ_p^\nr}\rH^1_\pst(X_{y_i'}/\bQ_p^\nr)\right)^r
		\end{equation}
		of symplectic modules over $\Kbar_v\otimes_{K_v}L_{w_i}\cong\prod_{j=1}^r \Kbar_v$. In particular, Lagrangian $\Kbar_v\otimes_{K_v}L_{w_i}$-submodules of the left-hand side of~\eqref{eq:base_change_trivialisation} correspond bijectively to $r$-tuples of Lagrangian $\Kbar_v$-subspaces of the symplectic vector space $\Kbar_v\otimes_{\bQ_p^\nr}\rH^1_\pst(X_{y_i'}/\bQ_p^\nr)$, so we have a decomposition
		\[
		\left(\Res^{L_{w_i}}_{K_v}\dH_{y_i'}\right)_{\Kbar_v} \cong \LGrass(\Kbar_v\otimes_{\bQ_p^\nr}\rH^1_\pst(X_{y_i'}/\bQ_p^\nr))^r \,.
		\]
		
		If now $\Phi_i'$ is a Lagrangian $L_{w_i}$-subspace of~$\rH^1_\deR(X_{y_i'}/L_{w_i})$, then $\Kbar_v\otimes_{K_v}\Phi_i'$ corresponds under~\eqref{eq:base_change_trivialisation} to an $r$-tuple of Lagrangian $\Kbar_v$-subspaces $\Phi_{ij}'$ of $\Kbar_v\otimes_{\bQ_p^\nr}\rH^1_\pst(X_{y_i'}/\bQ_p^\nr)$ satisfying 
		\[
		\Phi_{ij}'=(\sigma_j\otimes\varphi^{m_j})(\Phi_{i1}')
		\]
		for all~$i$. If $\Phi_i'\in\dH_{y_i'}(L_{w_i})_\bcase$, then by definition there is a non-zero isotropic $(\varphi,N,G_{w_i})$-submodule $D_\pst\leq\rH^1_\pst(X_{y_i'}/\bQ_p^\nr)$ such that $\dim_{\Kbar_v}(\Kbar_v\otimes_{\bQ_p^\nr}D_\pst)\cap\Phi_{i1}'\geq\frac12\dim_{\bQ_p^\nr}D_\pst$. Since~$D_\pst$ is~$\varphi$-stable, this in fact implies that for all~$j$
		\[
		\dim_{\Kbar_v}(\Kbar_v\otimes_{\bQ_p^\nr}D_\pst)\cap\Phi_{ij}'\geq\frac12\dim_{\bQ_p^\nr}D_\pst \,.
		\]
		Hence non-density of~$\dH_{y_i'}(L_{w_i})_\bcase$ in $\left(\Res^{L_{w_i}}_{K_v}\dH_{y_i'}\right)_{\Kbar_v}$ follows from the following pure linear algebra lemma.
	\end{proof}
\end{lem}

\newcommand{\vbla}{j}
\newcommand{\vblb}{j_1}
\newcommand{\vblc}{j_2}

\begin{sublem}[cf.\ {\cite[Lemma~6.4]{LVinventiones}}]\label{sublem:linear_algebra}
	Let $(V,\omega)$ be a symplectic vector space over a field~$F$, and let~$r\geq4$. Then the set of $r$-tuples $(\Phi_s)_{1\leq s\leq r}$ of Lagrangian $F$-subspaces of $V$ for which there exists a non-zero isotropic $F$-subspace $W\leq V$ satisfying
	\begin{equation}\label{eq:four_large_intersections}\tag{$\ast$}
		\dim(\Phi_\vbla\cap W)\geq\frac12\dim(W) \hspace{0.8cm} \text{for $1\leq \vbla\leq r$}
	\end{equation}
	is contained in a proper Zariski-closed subscheme of $\LGrass(V,\omega)^r$.
	\begin{proof}
		We may assume without loss of generality that~$F$ is algebraically closed, which anyway is the only case we will need. It suffices to prove the result in the case~$r=4$. Let $\IGrass^*(V,\omega)$ denote the Grassmannian of non-zero isotropic subspaces in $V$. The set $Z \subseteq \LGrass(V,\omega)^4$ described in the sublemma is the image under the proper projection 
		\[
		\pr_2 \colon \IGrass^*(V,\omega)\times\LGrass(V,\omega)^4 \to \LGrass(V,\omega)^4
		\]
		of the set $\tilde{Z}$ of all 	$(W,\Phi_1,\Phi_2,\Phi_3,\Phi_4)$ satisfying~\eqref{eq:four_large_intersections}. 
		Now for $\vbla=1,2,3,4$, the function $\dim(\Phi_\vbla\cap W)$  is upper-semicontinuous, and hence the subset $\tilde Z$ is Zariski-closed. The projection $Z = \pr_2(\tilde{Z})$ is then also Zariski-closed. We want to prove that $Z$ is not all of $\LGrass(V,\omega)^4$, for which it suffices to exhibit an $F$-point of $\LGrass(V,\omega)^4$ which doesn't lift to $\tilde Z(F)$.
		
		We construct such a point as follows. Let $\Phi_1 = \langle e_1,e_2,\dots,e_d\rangle$ be a Lagrangian subspace and let $\Phi_2 = \langle f_1,f_2,\dots,f_d\rangle$ be a Lagrangian complement spanned by a dual basis, i.e.\ such that $\omega(e_\alpha,f_\beta) = \delta_{\alpha\beta}$. The graph $\Phi_\psi$ of a linear map $\psi\colon \Phi_1 \to \Phi_2$ is a Lagrangian subspace of $V = \Phi_1 \oplus \Phi_2$ if and only if the pairing on $\Phi_1$ defined by $\langle x,y\rangle_\psi \colonequals \omega(x, \psi(y))$ is symmetric; and $\psi$ is an isomorphism if and only if $\langle -,-\rangle_\psi$ is non-degenerate. We choose $\psi_3,\psi_4\colon\Phi_1\xrightarrow\sim\Phi_2$ as the maps given by $\psi_3(e_\alpha)=f_\alpha$ and $\psi_4(e_\alpha)=\lambda_\alpha f_\alpha$ for a choice of distinct nonzero $\lambda_\alpha \in F^\times$. Clearly the pairings $\langle -, -\rangle_{\psi_\vbla}$ are symmetric for $\vbla = 3,4$. With $\Phi_\vbla = \Phi_{\psi_\vbla}$ for $\vbla = 3,4$ we find
		\[
		V = \Phi_{\vblb} \oplus \Phi_{\vblc} \ ; \ \text{ for all $\vblb<\vblc$, } (\vblb,\vblc) \neq (3,4) \,.
		\]
		
		Suppose now for contradiction that $(\Phi_1,\Phi_2,\Phi_3,\Phi_4)$ lies in the image of a point in $\tilde Z(F)$. There is thus a non-zero isotropic $F$-subspace $W$ of $V$ satisfying~\eqref{eq:four_large_intersections}.  It follows from this that, writing $W_\vbla = \Phi_\vbla \cap W$, we have
		\[
		W = W_{\vblb} \oplus W_{\vblc} \ ; \ \text{ for all $\vblb<\vblc$, } (\vblb,\vblc) \neq (3,4) \,.
		\]
		It follows that $\psi_\vbla(W_1) = W_2$ for $\vbla = 3, 4$ and thus $W_1$ is a $g = \psi_3^{-1} \psi_4$-stable subspace of $\Phi_1$. Since $g(e_\alpha) = \lambda_\alpha e_\alpha$, the map $g$ has $\dim(\Phi_1)$-many different eigenvalues and so $W_1$ must contain one of the vectors~$e_\alpha$. But then $W$ also contains $\psi_3(e_\alpha)=f_\alpha$, contradicting the fact that~$W$ is isotropic as $\omega(e_\alpha,f_\alpha) = 1$.
	\end{proof}
\end{sublem}

\begin{rmk}
	The corresponding result in \cite{LVinventiones}, Lemma~6.4, differs from Sublemma~\ref{sublem:linear_algebra} in that~$W$ is not required to be isotropic, but on the other hand,~$r$ is required to be at least~$ 5$.
\end{rmk}

\begin{proof}[Proof of Proposition~\ref{prop:non-density_b}]
	As in the proof of Proposition~\ref{prop:non-density_a}, our aim is to prove that for every $\Phi\in\dH_{y_0}(K_v)_{X,\bcase}^\LV$ there exists an index~$i$ such that~$[G_v:G_{w_i}]\geq4$ and the $i$th component~$\Phi_i'$ of~$\Phi$ lies in~$\dH_{y_i'}(L_{w_i})_\bcase$. This implies the proposition by Lemma~\ref{lem:non-density_componentwise_b} and the decomposition~\eqref{eq:grassmannian_decomposition} of~$\dH_{y_0}$.
	
	Again, proving this claim is a matter of unwinding definitions. Given some~$\Phi\in\dH_{y_0}(K_v)_{X,\bcase}^\LV$, there is by definition a symplectic pair~$(A,V)$ of type~\bcase in the category of $G_v$-representations and an isomorphism
	\begin{equation}\label{eq:interpolating_pair_b}\tag{$\ast$}
		\MM^1(\Phi) \cong (\sD_\pH(A),\sD_\pH(V))
	\end{equation}
	of symplectic pairs in filtered~$(\varphi,N,G_v)$-modules. Recall that being of type~\bcase means that there exists some~$\psi\in\Sigma_A$ such that~$[G_v:G_{w_\psi}]\geq4$ and~$V_\psi$ possesses a non-zero isotropic $G_{w_\psi}$-subrepresentation $W\leq V_\psi$ whose average Hodge--Tate weight is~$\geq1/2$. As in the proof of Proposition~\ref{prop:non-density_a}, the element~$\psi\in\Sigma_A$ corresponds to the $\bQ_p$-algebra homomorphism
	\[
	\psi_i\colon\rH^0_\et(X_{y_0,\Kbar_v},\bQ_p) \to \bQ_p
	\]
	induced by a pair of a closed point $y_i'\in|Y'_{y_0}|$ and a $K_v$-embedding $K_v(y_i')\hookrightarrow\Kbar_v$; again we may assume without loss of generality that this embedding is the embedding $L_{w_i}\hookrightarrow\Kbar_v$ fixed in~\S\ref{sss:period_map_decomposition}.
	
	For this value of~$i$, we have~$[L_{w_i}:K_v]=[G_v:G_{w_\psi}]\geq4$ by assumption, and
	\[
	\sD_{\pH,w_\psi}(V_\psi|_{G_{w_\psi}}) \cong \MM^1_i(\Phi_i')
	\]
	as symplectic $\unit$-modules in the category of filtered discrete~$(\varphi,N,G_{w_i})$-modules. Hence~$\sD_{\pH,w_\psi}(W)$ corresponds to a non-zero isotropic filtered discrete $(\varphi,N,G_{w_i})$-submodule~$D$ of $\MM^1_i(\Phi_i')$ whose average Hodge--Tate weight is~$\geq1/2$. So~$\MM^1_i(\Phi_i')\in\dH_{y_i'}(L_{w_i})_\bcase$, completing the proof.
\end{proof}

\subsection{The Kodaira--Parshin family and points of type~\ccase}\label{ss:type_ccase}

It remains to deal with the set $Y(K_v)_{X,\ccase}^\LV$. The approach here is different to the approach for $Y(K_v)_{X,S,\acase}^\LV$ and $Y(K_v)_{X,\bcase}^\LV$: rather than prove that $Y(K_v)_{X,\ccase}^\LV$ is finite in any level of generality, we will instead exhibit a particular abelian-by-finite family $X\to Y'\to Y$ for which $Y(K_v)_{X,\ccase}^\LV=\emptyset$. The relevant construction was given in \cite[\S7]{LVinventiones}, where Lawrence and Venkatesh constructed an abelian-by-finite family $X_q\to Y_q'\to Y$ called the \emph{Kodaira--Parshin family} with parameter~$q$ (an odd prime number). The Kodaira--Parshin family has constant relative dimension $d_q=(g-1/2)(q-1)>0$ \cite[\S7.2]{LVinventiones} and always has full monodromy \cite[\S8.2.3]{LVinventiones}. The key point here is that, for a suitable choice of~$q$, we also have $Y(K_v)_{X,\ccase}^\LV=\emptyset$.

\begin{prop}\label{prop:choose_the_right_k-p_family}
	Let $X_q\to Y'_q\to Y$ be the Kodaira--Parshin family with parameter~$q$ an odd prime number, and suppose that~$q$ satisfies the following two conditions:
	\begin{itemize}
		\item $q-1$ is not divisible by~$4$ or any odd prime divisor of $q_v(q_v+1)(q_v^3-1)$, where~$q_v$ is the cardinality of the residue field of~$K_v$; and
		\item $q-1 \geq 9n_v$, where~$n_v$ is the number of extensions of~$K_v$ of degree~$2$ or~$3$ inside~$\Kbar_v$.
	\end{itemize}
	Then $Y(K_v)_{X_q,\ccase}^\LV=\emptyset$.
	\begin{proof}
		We follow a proof similar to~\cite[Theorem~5.4]{LVinventiones}. For a finite non-empty $G_v$-set~$\Sigma$ let us write\footnote{This is the same as \cite[Definition~5.2]{LVinventiones}, except that we count orbits of size~$<4$ rather than~$<8$.} 
		\[
		\size_v(\Sigma) \colonequals \frac{\#\{x\in\Sigma \::\: \text{$x$ is contained in a $G_v$-orbit of size~$<4$}\}}{\#\Sigma} \,.
		\]
		So we wish to prove that
		\[
		\size_v(Y'_{q,y}(\Kbar_v)) < \frac1{d_q+1}
		\]
		for all~$y\in Y(K_v)$.
		
		We denote by $\Aff(q)$ the affine linear group $\bF_q\rtimes\bF_q^\times$ of dimension $1$, and write~$k=(q-1)/2$ for short. The construction of the Kodaira--Parshin family shows that there is a $G_v$-equivariant bijection
		\[
		Y'_{q,y}(\Kbar_v) \cong \Surj^{\out,*}(\pi_1^\et(Y_{\Kbar_v}-y),\Aff(q))\,,
		\]
		where $\Surj^{\out,*}(\pi_1^\et(Y_{\Kbar_v}-y),\Aff(q))$ denotes the set of surjections
		$\pi_1^\et(Y_{\Kbar_v}-y)\twoheadrightarrow\Aff(q)$	which are non-trivial on an inertia generator at $y$, up to conjugation. The projections $\Aff(q)\surj \bF_q^\times\twoheadrightarrow \bZ/k\bZ$ induce $G_v$-equivariant surjections
		\[
		\Surj^{\out,*}(\pi_1^\et(Y_{\Kbar_v}-y),\Aff(q))\twoheadrightarrow\Surj(\pi_1^\et(Y_{\Kbar_v}-y),\bF_q^\times)=\Surj(\pi_1^\et(Y_{\Kbar_v}),\bF_q^\times) \twoheadrightarrow \Surj(\pi_1^\et(Y_{\Kbar_v},\bZ/k\bZ)) \,,
		\]
		each of whose fibres all have the same size, see \cite[Lemma~2.11]{LVinventiones}\footnote{Strictly speaking, \cite[Lemma~2.11]{LVinventiones} only proves the corresponding statement when $\Surj^{\out,*}(\pi_1^\et(Y_{\Kbar_v}-y),\Aff(q))$ is replaced by the set $\Surj^*(\pi_1^\et(Y_{\Kbar_v}-y),\Aff(q))$ of surjections $\pi_1^\et(Y_{\Kbar_v}-y)\twoheadrightarrow\Aff(q)$ which are non-trivial on an inertia generator at~$y$, \emph{not} up to conjugation. But the conjugation action of $\Aff(q)$ on $\Surj^*(\pi_1^\et(Y_{\Kbar_v}-y),\Aff(q))$ is free (since~$\Aff(q)$ has trivial centre), from which we deduce the statement we need.}. We then have the estimate
		\[
		\size_v(Y_{q,y}(\Kbar_v)) = \size_v(\Surj^{\out,*}(\pi_1^\et(Y_{\Kbar_v}-y),\Aff(q))) \leq \size_v(\Surj(\pi_1^\et(Y_{\Kbar_v}),\bZ/k\bZ)) \,,
		\]
		so it suffices to prove that $\size_v(\Surj(\pi_1^\et(Y_{\Kbar_v}),\bZ/k\bZ))<\frac1{d_q+1}$.
		
		On the one hand, the number of surjections $\pi_1^\et(Y_{\Kbar_v})\twoheadrightarrow\bZ/k\bZ$ is equal to the number of $2g$-tuples of elements of $\bZ/k\bZ$ which generate it as an abelian group. The number of such tuples is
		\[
		\# \Surj(\pi_1^\et(Y_{\Kbar_v}),\bZ/k\bZ) = k^{2g}\cdot\prod_r(1-r^{-2g})>\zeta(2g)^{-1}\cdot k^{2g}
		\]
		where $r$ runs through the distinct prime factors\footnote{One can tighten the bound slightly by observing that~$2,3\nmid k$.} of $k$.
		
		On the other hand, the group of homomorphisms $\pi_1^\et(Y_{\Kbar})\rightarrow\bZ/k\bZ$ (not necessarily surjective) is $G_v$-equivariantly isomorphic to the \'etale cohomology group
		\[
		M\colonequals\rH^1_\et(Y_{\Kbar_v},\bZ/k\bZ) \,,
		\]
		which carries a perfect $G_v$-equivariant alternating Weil pairing
		\[
		\langle\cdot,\cdot\rangle\colon M\otimes M\rightarrow\mu_k^\vee\,.
		\]
		We need to estimate the size of the subset $\bigcup_{G_w} M^{G_w}$ of $M$, where $G_w$ ranges over open subgroups of index~$2$ or~$3$. Such a $G_w$ contains an element $T$ acting on the residue field as the $i$th power of arithmetic Frobenius for some $1\leq i\leq3$. If $m_1,m_2\in M^{G_w}$, we then have
		\[
		\langle m_1,m_2\rangle = \langle Tm_1,Tm_2\rangle = T\langle m_1,m_2\rangle = q_v^{-i}\langle m_1,m_2\rangle\,.
		\]
		But the first of our assumptions on $q$ ensures that $\gcd(k,q_v^i-1)=1$, and hence we have $\langle m_1,m_2\rangle=0$, i.e.\ $M^{G_w}$ is isotropic with respect to the Weil pairing. This implies that
		\[
		\# M^{G_w} \leq k^g \,.
		\]
		In total, the number of elements of~$M$ contained in a $G_v$-orbit of size~$<4$ is at most~$n_vk^g$.
		
		Combining this with the previous part, we obtain
		\[
		\size_v(\Surj(\pi_1^\et(Y_{\Kbar_v}),\bZ/k\bZ)) < \frac{\zeta(2g)n_v}{k^g} < \frac{9n_v}{8k^g} \leq \frac1{4k^{g-1}} < \frac1{(2g-1)k+1} = \frac1{d_q+1} \,,
		\]
		which is what we wanted to prove.
	\end{proof}
\end{prop}

\begin{rmk}
	The quantity~$n_v$ appearing in the statement of Proposition~\ref{prop:choose_the_right_k-p_family} is bounded depending only on~$K$. In fact, if~$v\nmid 6$ then we always have $n_v=7$: $K_v$ has three quadratic extensions and four cubic extensions inside~$\Kbar_v$. (If~$q_v\equiv1$ modulo~$3$ then all of the cubic extensions are Galois; if instead~$q_v\equiv-1$ modulo~$3$ then one of the cubic extensions is Galois and the other three are conjugate non-Galois extensions.)
\end{rmk}

\begin{cor}\label{cor:good_choice_of_family}
	For every finite place~$v$ of~$K$, there exists an abelian-by-finite family $X\to Y'\to Y$ of constant relative dimension $>0$, having full monodromy, such that $Y(K_v)_{X_q,\ccase}^\LV=\emptyset$.
	\begin{proof}
		Dirichlet's Theorem ensures that there exists an odd prime~$q$ satisfying the conditions of Proposition~\ref{prop:choose_the_right_k-p_family}. The Kodaira--Parshin family with this parameter~$q$ has full monodromy \cite[\S8.2.3]{LVinventiones} and has $Y(K_v)_{X,\ccase}^\LV=\emptyset$ by Proposition~\ref{prop:choose_the_right_k-p_family}.
	\end{proof}
\end{cor}

\subsection{Proof of Theorem~\ref{thm:locus_finiteness}}

Putting this all together, we obtain the proof of Theorem~\ref{thm:locus_finiteness}. Suppose that~$v$ is self-conjugate, and let~$X_q\to Y'_q\to Y$ be the Kodaira--Parshin family with parameter~$q$, where~$q$ satisfies the conditions of Proposition~\ref{prop:choose_the_right_k-p_family}. The Principal Trichotomy (Proposition~\ref{prop:principal_trichotomy}) implies that
\[
Y(K_v)_{X_q,S}^\LV \subseteq Y(K_v)_{X_q,S,\acase}^\LV\cup Y(K_v)_{X,\bcase}^\LV\cup Y(K_v)_{X_q,\ccase}^\LV \,,
\]
since~$v$ is self-conjugate. Because the Kodaira--Parshin family has full monodromy, the sets $Y(K_v)_{X_q,S,\acase}^\LV$ and $Y(K_v)_{X_q,\bcase}^\LV$ are finite by Corollaries~\ref{cor:finiteness_a} and~\ref{cor:finiteness_b}, and we also have $Y(K_v)_{X_q,\ccase}^\LV=\emptyset$ by Proposition~\ref{prop:choose_the_right_k-p_family}. Hence the Lawrence--Venkatesh locus $Y(K_v)_{X_q,S}^\LV$ is finite.\qed

\begin{rmk}\label{rmk:compare_v,q}
	Compared with the proof in~\cite{LVinventiones}, our proof shows finiteness of the locus $Y(K_v)_{X_q,S}^\LV$ for more pairs~$(v,q)$ of a finite place~$v$ and an odd prime~$q$. The proof in \cite{LVinventiones} establishes finiteness of~$Y(K_v)_{X_q,S}^\LV$ whenever~$(v,q)$ satisfies the following list of conditions:
	\begin{enumerate}[label = (\arabic*$^\circ$), ref = (\arabic*$^\circ$)]
		\item\label{condn:LV_qv_1} $v$ is self-conjugate and unramified over~$\bQ$ (i.e.\ friendly), and does not divide~$2$;
		\item\label{condn:LV_qv_2} $q-1$ is not divisible by~$4$ or any odd prime divisor of~$q_v\cdot\prod_{i=1}^7(q_v^i-1)$;
		\item\label{condn:LV_qv_3} \[
		\frac{4\cdot\zeta(2g)\cdot 2^g}{(q-1)^g} < \frac1{(g-1/2)(q-1)+1} \,;
		\]
		\item\label{condn:LV_qv_4} the Kodaira--Parshin family~$X_q\to Y'_q\to Y$ admits a good model over the ring of integers of~$K_v$ in the sense of \cite[Definition~5.1]{LVinventiones}.
	\end{enumerate}
	By contrast, our proof establishes finiteness of~$Y(K_v)_{X_q,S}^\LV$ whenever~$(v,q)$ satisfies the following less stringent list of conditions:
	\begin{enumerate}[label = (\arabic*), ref = (\arabic*)]
		\item\label{condn:new_qv_1} $v$ is self-conjugate;
		\item\label{condn:new_qv_2} $q-1$ is not divisible by~$4$ or any odd prime divisor of~$q_v\cdot\prod_{i=1}^3(q_v^i-1)$;
		\item\label{condn:new_qv_3} \[
		\frac{n_v\cdot\zeta(2g)\cdot2^g}{(q-1)^g} < \frac1{(g-1/2)(q-1)+1} \,.
		\]
	\end{enumerate}
	(In the case that~$Y$ has good reduction at~$v$, the quantity~$n_v$ can be replaced by~$2$, since the representation~$M$ appearing in the proof of Proposition~\ref{prop:choose_the_right_k-p_family} is unramified, and so one need only consider the two subgroups $G_w$ corresponding to unramified extensions of~$K_v$.)
	\smallskip
	
	We conclude with a miscellany of remarks comparing how easy it is to satisfy the conditions~\ref{condn:LV_qv_1}--\ref{condn:LV_qv_4} compared with~\ref{condn:new_qv_1}--\ref{condn:new_qv_3}.
	
	\begin{enumerate}[label = (\Roman*), ref = (\Roman*)]
		\item\label{rmkpart:100_percent} When~$K$ contains no CM subfield, then for 100\% of places~$v$ (in the natural density) there exists an odd prime~$q$ satisfying~\ref{condn:LV_qv_1}--\ref{condn:LV_qv_4}. Indeed, for any~$q$ satisfying~\ref{condn:LV_qv_3}, conditions~\ref{condn:LV_qv_1} and~\ref{condn:LV_qv_4} are satisfied for all but finitely many places~$v$. And if we choose~$q\equiv3$ modulo~$4$ such that~$\bQ(\zeta_{q-1})$ is linearly disjoint from the Galois closure~$K'$ of~$K$, then $q_v$ is equidistributed among the invertible residue classes modulo~$q-1$, so condition~\ref{condn:LV_qv_2} is satisfied for a proportion of at least
		\begin{equation}\label{eq:original_LV_proportion}\tag{$\ast$}
		\prod_{r\mid q-1}\max\left(1-\frac{18}{r-1},0\right)
		\end{equation}
		of places of~$K$, where~$r$ runs over odd prime factors of~$q$. (Here, $18$ is the maximal number of elements of order~$\leq7$ in a cyclic group.)
		
		We claim that the quantity~\eqref{eq:original_LV_proportion} can be made arbitrarily close to~$1$ for suitably chosen~$q$ (satisfying all other relevant conditions). For this, choose an auxiliary prime number~$r_0\geq100$ strictly greater than all odd prime divisors of~$[K':\bQ]$, and let~$q$ be the least prime number congruent to~$3$ modulo~$4$ and congruent to~$2$ modulo all odd primes less than~$r_0$. Then~\ref{condn:LV_qv_3} is automatically satisfied, $K'$ is linearly disjoint from~$\bQ(\zeta_{q-1})$, and Linnik's Theorem tells us that $\log(q)=O(\pi(r_0)\log(r_0))=O(r_0)$ where~$\pi(r_0)=O(r_0/\log(r_0))$ is the prime counting function. Now every odd prime factor of~$q-1$ is at least~$r_0$ and there are at most $\log(q)/\log(r_0)$ of them, so we obtain
		\[
		\eqref{eq:original_LV_proportion} \geq \left(1-\frac{50}{r_0}\right)^{\log(q)/\log(r_0)} \geq \exp\left(-\frac{100\log(q)}{r_0\log(r_0)}\right) = \exp\left(-\frac{O(1)}{\log(r_0)}\right) \,,
		\]
		and the right-hand bound tends to~$1$ as~$r_0\to\infty$.
		\item On the other hand, there are always places~$v$ for which~\ref{condn:LV_qv_1}--\ref{condn:LV_qv_4} are not satisfied for any~$q$, e.g.\ if~$v$ is ramified over~$\bQ$, if $v\mid2$, or if~$Y$ has bad reduction at~$v$. This is the most significant difference in the context of this paper, since we want to prove a constraint for every place~$v$.
		\item In general, the smallest value of~$q$ satisfying~\ref{condn:new_qv_1}--\ref{condn:new_qv_3} is smaller than the smallest value satisfying~\ref{condn:LV_qv_1}--\ref{condn:LV_qv_4}. For example, if~$Y$ is a curve of genus~$3$, then the smallest value of~$q$ satisfying~\ref{condn:LV_qv_1}--\ref{condn:LV_qv_4} for some~$v$ is~$q=23$, since~\ref{condn:LV_qv_2} implies in particular that~$q-1$ cannot be divisible by $3$, $5$ or~$7$. On the other hand, the smallest value of~$q$ satisfying~\ref{condn:new_qv_1}--\ref{condn:new_qv_3} (with the modification of~\ref{condn:new_qv_3} for places of good reduction) is~$q=11$. As mentioned in the introduction, this extra efficiency may be useful in the context of carrying out the Lawrence--Venkatesh method in practice.
	\end{enumerate}
\end{rmk}
\section{The Lawrence--Venkatesh obstruction for Selmer sections}\label{s:proof}

Having shown finiteness of the Lawrence--Venkatesh locus, we now come to the other half of our proof of Theorem~\ref{thm:finite_selmer_image}. We fix notation as in the preceding section; that is, $K$ is a number field, $Y/K$ is a smooth projective curve of genus~$\geq2$, and $v$ is a $p$-adic place of~$K$ for some rational prime~$p$. Our aim is to prove the following, which proves point~\ref{mainpart:containment} of the introduction.

\begin{thm}\label{thm:selmer_sections_survive_LV}
	For every abelian-by-finite family~$X\to Y'\to Y$, the image of the localisation map
	\[
	\Sec^\Sel(Y/K) \to Y(K_v)
	\]
	is contained in the Lawrence--Venkatesh locus~$Y(K_v)^\LV_X$ (see Definition~\ref{def:global_pair}).
\end{thm}

For the remainder of this section, we fix a choice of abelian-by-finite family $X\to Y'\to Y$ and write $\pi \colon X \to Y$ for the projection.

\subsection{Relative \'etale cohomology of an abelian-by-finite family}\label{ss:relative_etale_cohomology}
The relative \'etale cohomology $\rR^\bullet\!\pi_{\et*}\bQ_p$ of the abelian-by-finite family $\pi\colon X\to Y$ is a $\bQ_p$-local system on~$Y_\et$ whose fibre at any geometric point~$y$ of~$Y$ is $\rH^\bullet_\et(X_y,\bQ_p)$. We see, e.g.\ by looking at any fibre, that the natural map
\[
\bigwedgebullet_{\pi_{\et*}\bQ_p} \rR^1\!\pi_{\et*}\bQ_p\xrightarrow\sim \rR^\bullet\!\pi_{\et*}\bQ_p
\]
is an isomorphism of $\bQ_p$-local systems on~$Y_\et$. Now let~$\lambda$ denote the polarisation on~$X\to Y'$. We define its \emph{first relative \'etale Chern class} 
\[
c_1^\et(\lambda)_{/Y}\in\rH^0(Y,\rR^2\!\pi_{\et*}\bQ_p(1))
\]
as follows. The Kummer sequence induces a boundary map
\[
\rR^1\!\pi_{\et*}\bG_{m} \to \rR^2\!\pi_{\et*}\mu_{p^n} \cong \rR^2\!\pi_{\et*}\bZ/p^n\bZ(1)
\]
for every~$n\geq0$. These maps are compatible in a natural way, so induce a map
\[
\rR^1\!\pi_{\et*}\bG_{m} \to \rR^2\!\pi_{\et*}\bZ_{p}(1)
\]
of inverse systems of abelian sheaves on~$Y_\et$, where the left-hand side is viewed as a constant inverse system. For a line bundle~$\dL$ on~$X$, we define $c_1^\et(\dL)_{/Y}$ to be the image of the class of~$\dL$ in $\Pic(X)=\rH^1(X_\et,\bG_{m})$ under the composite
\[
\rH^1(X_\et,\bG_{m}) \to \rH^0(Y,\rR^1\!\pi_{\et*}\bG_{m}) \to \rH^0\big(Y,\rR^2\!\pi_{\et*}\bZ_{p}(1)\big) \to \rH^0\big(Y,\rR^2\!\pi_{\et*}\bQ_p(1)\big) 
\]
where the first map arises from the Leray spectral sequence for \'etale cohomology, and the final map is tensoring with~$\bQ_p$ over~$\bZ_p$. We set~$c_1^\et(\lambda)_{/Y}=\frac12c_1^\et(\dL)_{/Y}$ for~$\dL=(1,\lambda)^*\dP$ with~$\dP$ the Poincar\'e bundle.

Now it follows from the usual functoriality properties of the Leray spectral sequence that the fibre of $c_1^\et(\lambda)_{/Y}$ at a geometric point~$y$ of~$Y$ is $c_1^\et(\lambda_y)\in\rH^2_\et(X_y,\bQ_p(1))$, where~$\lambda_y$ is the polarisation on the fibre~$X_y$. If we define 
\[
\relH_1^\et(X/Y,\bQ_p)\colonequals\cHom_{\pi_{\et*}\bQ_p}\big(\rR^1\!\pi_{\et*}\bQ_p, \pi_{\et*}\bQ_p\big),
\]
then 
\[
c_1^\et(\lambda)_{/Y}\in\rH^0\!\big(Y,\rR^2\!\pi_{\et*}\bQ_p(1)\big)
=\rH^0\big(Y,\!\!\bigwedgesquare_{\pi_{\et*}\bQ_p}\! \rR^1\!\pi_{\et*}\bQ_p(1)\big)
\]
corresponds to a pairing
\[
\check\omega_\lambda^\et\colon\bigwedgesquare_{\pi_{\et*}\bQ_p}\relH_1^\et(X/Y,\bQ_p)\to
\pi_{\et*}\bQ_p(1) \,.
\]
By considering the fibre of this pairing at a geometric point of~$Y$ and using Lemma~\ref{lem:etale_pairing_perfect}, we see that~$\check\omega_\lambda^\et$ is a perfect pairing: it induces an isomorphism
\[
\relH_1^\et(X/Y,\bQ_p)\xrightarrow\sim \rR^1\!\pi_{\et*}\bQ_p(1)
\]
of $\bQ_p$-local systems on~$Y_\et$. Hence there is a dual pairing
\[
\omega_\lambda\colon\bigwedgesquare_{\pi_{\et*}\bQ_p} 
\rR^1\!\pi_{\et*}\bQ_p
\to
\pi_{\et*}\bQ_p(-1) \,,
\]
making $\rR^1\!\pi_{\et*}\bQ_p$ into a symplectic $\pi_{\et*}\bQ_p$-module in the category of $\bQ_p$-local systems on~$Y_\et$. The fibre of the symplectic pair 
\[
\rR^{\leq 1}\!\pi_{\et*}\bQ_p \colonequals (\pi_{\et*}\bQ_p,\rR^1\!\pi_{\et*}\bQ_p)
\]
at a geometric point~$y$ of~$Y$ is isomorphic to the symplectic pair 
\[
\rH^{\leq1}_\et(X_{y},\bQ_p) = (\rH^0_\et(X_y,\bQ_p),\rH^1_\et(X_y,\bQ_p))
\]
compatibly with the algebra and symplectic module structures and, if~$y$ is defined over a non-algebraically closed field, the action of the Galois group.

\subsection{Attaching representations to Galois sections}

We can use this setup to attach to every Galois section a symplectic pair in the category of $G_K$-representations. Let us fix a geometric point~$\bar\eta$ of~$Y_{\Kbar}$, and write~$A$ and~$V$ for the $\pi_1^\et(Y,\bar\eta)$-representations corresponding to the $\bQ_p$-local systems $\pi_{\et*}\bQ_p$ and $\rR^1\!\pi_{\et*}\bQ_p$, respectively. Thus~$(A,V)$ is a symplectic pair in the category of $\pi_1^\et(Y,\bar\eta)$-representations.

If~$s$ is a section of the fundamental exact sequence
\[
1 \to \pi_1^\et(Y_{\Kbar},\bar\eta) \to \pi_1^\et(Y,\bar\eta) \to G_K \to 1 \,,
\]
then we define $A_s\colonequals A$ and $V_s\colonequals V$, endowed with the action of~$G_K$ given by restriction along $s\colon G_K\to\pi_1^\et(Y,\bar\eta)$. Thus~$(A_s,V_s)$ is a symplectic pair in the category of $G_K$-representations. Moreover, if~$\gamma$ is an element of $\pi_1^\et(Y_{\Kbar},\bar\eta)$ and $s'=\gamma s\gamma^{-1}$ is a conjugate section, then the action of~$\gamma$ gives a $G_K$-equivariant isomorphism $(A_s,V_s)\xrightarrow\sim (A_{s'},V_{s'})$. So the isomorphism class of the symplectic pair $(A_s,V_s)$ depends only on the $\pi_1^\et(Y_{\Kbar},\bar\eta)$-conjugacy class of~$s$.

More generally, if~$\bar\eta'$ is another geometric point of~$Y_{\Kbar}$ and $\gamma\in\pi_1^\et(Y_{\Kbar};\bar\eta,\bar\eta')$ is an \'etale path from~$\bar\eta$ to~$\bar\eta'$, then conjugation by~$\gamma$ induces an isomorphism
\begin{equation}\label{eq:two_fes}\tag{$\dagger$}
\begin{tikzcd}
	1 \arrow[r] & \pi_1^\et(Y_{\Kbar},\bar\eta) \arrow[r]\arrow[d,"\wr"] & \pi_1^\et(Y,\bar\eta) \arrow[r]\arrow[d,"\wr"] & G_K \arrow[r]\arrow[d,equals] & 1 \\
	1 \arrow[r] & \pi_1^\et(Y_{\Kbar},\bar\eta') \arrow[r] & \pi_1^\et(Y,\bar\eta') \arrow[r] & G_K \arrow[r] & 1
\end{tikzcd}
\end{equation}
between the fundamental exact sequences at basepoints~$\bar\eta$ and~$\bar\eta'$, respectively. If~$s'=\gamma s\gamma^{-1}$ is the section of the lower sequence in~\eqref{eq:two_fes} conjugate to~$s$, then the action of~$\gamma$ gives a $G_K$-equivariant isomorphism $(A_s,V_s)\xrightarrow\sim(A'_{s'},V'_{s'})$, where~$A'$ and~$V'$ are the $\pi_1^\et(Y,\bar\eta')$-representations corresponding to the $\bQ_p$-local systems $\pi_{\et*}\bQ_p$ and $\rR^1\!\pi_{\et*}\bQ_p$, respectively. So the isomorphism class of the symplectic pair~$(A_s,V_s)$ is also independent of the choice of geometric basepoint~$\bar\eta$ defining the \'etale fundamental group.

\subsubsection{Selmer sections}

In the particular case that the section~$s$ is Selmer, we can say considerably more about the local behaviour of the representations~$A_s$ and~$V_s$.

\begin{prop}\label{prop:section_localisation}
	Let~$u$ be a place of~$K$. Suppose that the restriction $s|_{G_u}$ of~$s$ to the decomposition group~$G_u\subseteq G_K$ is the local section arising from a point $y_u\in Y(K_u)$. Then
	\[
	(A_s|_{G_u},V_s|_{G_u}) \cong \rH^{\leq1}_\et(X_{y_u,\Kbar_u},\bQ_p)
	\]
	as symplectic pairs in the category of $G_u$-representations.
\end{prop}

\begin{proof}
	By the above discussion, we may suppose that the geometric point~$\bar\eta$ defining the fundamental group is~$\bar y_u$, the $\Kbar_u$-valued geometric point of~$Y_{\Kbar}$ determined by~$y_u$. Moreover, we may suppose that the restricted section~$s|_{G_u}$ is the map
	\[
	G_u=\pi_1^\et(\Spec(K_u),\Spec(\Kbar_u))\xrightarrow{\iota_*}\pi_1^\et(Y_{K_u},\bar y_u)\subseteq\pi_1^\et(Y,\bar y_u)
	\] 
	induced by functoriality from the morphism
	\[
	\iota\colon(\Spec(K_u),\Spec(\Kbar_u))\to(Y_{K_u},\bar y_u)
	\]
	of pointed schemes. By definition of the \'etale fundamental group as a functor, $(A_s|_{G_u},V_s|_{G_u})$ is $G_u$-equivariantly isomorphic to $\iota^*\rR^{\leq1}\!\pi_{\et*}\bQ_p$, which is isomorphic to $\rH_\et^{\leq1}(X_{y_u,\Kbar_u},\bQ_p)$.
\end{proof}

\begin{proof}[Proof of Theorem~\ref{thm:selmer_sections_survive_LV}.] As a consequence of Proposition~\ref{prop:section_localisation}, if~$s$ is a Selmer section with associated adelic point~$y=(y_u)_u\in Y(\bA_K)$, then the symplectic pair~$(A_s,V_s)$ interpolates the points~$y_u$ in the sense of Definition~\ref{def:global_pair}. In particular, we have $y\in Y(\bA_K)_X^\LV$, and $y_v\in Y(K_v)_X^\LV$. Thus we have proved Theorem~\ref{thm:selmer_sections_survive_LV}.
\end{proof}
%%%%%%%%%%%%%%%%%%%%%%%%%%%%%%%%%%%%%%%%%%%
%%%%%%%%%%%%%%%%%%%%%%%%%%%%%%%%%%%%%%%%%%%
\appendix
\section{Filtered derived categories}\label{appx:derived}

In this appendix, we prove a few basic lemmas on filtered derived categories of sheaves used in the proof of Theorem~\ref{thm:pushforward_of_pairs}. Suppose that~$(U,\dA)$ is a small ringed site. One then has the \emph{derived category} $\Derv(\dA)$ of $\dA$-modules, which is given as the localisation of the category$~\Kom(\dA)$ of (unbounded) complexes of $\dA$-modules at the quasi-isomorphisms.

Let $\SKom(\dA)$ denote the category of $\bZ$-indexed decreasing sequences
\[
\dots \to \rF^1\!A \to \rF^0\!A \to \rF^{-1}\!A \to \dots
\]
of complexes of sheaves of $\dA$-modules. The maps $\rF^{i+1}\!A\to\rF^i\!A$ are not necessarily monomorphisms. A morphism $A\to B$ in~$\SKom(\dA)$ is called a \emph{filtered quasi-isomorphism} just when the induced maps~$\rF^i\!A\to\rF^i\!B$ are quasi-isomorphisms for all~$i$. We defined the \emph{filtered derived category} $\FDerv(\dA)$ to be the localisation of~$\SKom(\dA)$ at the filtered quasi-isomorphisms. Equivalently, $\FDerv(\dA)$ is the derived category of the category~$\Seq(\dA)$ of sequences in~$\dA$.

\begin{rmk}
	Any object of~$\SKom(\dA)$ is quasi-isomorphic to an object~$A$ in which all of the maps $\rF^{i+1}\!A\to\rF^i\!A$ are monomorphisms, i.e.\ $A$ is a filtered complex. Thus~$\FDerv(\dA)$ can also be described as the derived category of the category~$\FMod(\dA)$ of filtered $\dA$-modules, hence the name. However, since~$\FMod(\dA)$ is not an abelian category, it is generally clearer to view~$\FDerv(\dA)$ as the derived category of~$\SMod(\dA)$ instead.
\end{rmk}

There is a second important class of morphisms in~$\SKom(\dA)$. If~$A\in\SKom(\dA)$ is a sequence of complexes of~$\dA$-modules, we write~$\gr^i_{\rF}A$ for the cone of the map $\rF^{i+1}\!A\to\rF^i\!A$. A map $A\to B$ in~$\SKom(\dA)$ is called a \emph{graded quasi-isomorphism} just when the induced maps $\gr^i_{\rF}A \to \gr^i_{\rF}B$ are quasi-isomorphisms for all~$i$.

\subsection{Cohomology objects}\label{ss:cohomology_objects}

If~$A\in\SKom(\dA)$ is a sequence of complexes, then it has natural cohomology objects~$\rH^j(A)\in\Seq(\dA)$, namely the sequences
\[
\dots\to\rH^j(\rF^1\!A)\to\rH^j(\rF^0\!A)\to\rH^j(\rF^{-1}\!A)\to\dots \,.
\]
Note that even if~$A\in\FKom(\dA)$ is a filtered complex, its cohomology objects~$\rH^j(A)$ may still not be filtered themselves (meaning the maps $\rH^j(\rF^{i+1}\!A)\to\rH^j(\rF^i\!A)$ may not be injective). We have the following criterion due to Deligne which governs when this happens.

\begin{lem}\label{lem:degeneration_implies_filtered_cohomology}
	Suppose that~$A$ is a filtered complex of sheaves of $\dA$-modules whose filtration is bounded below in every degree. If the spectral sequence associated to~$A$ degenerates at the first page, then~$\rH^j(A)$ is a filtered $\dA$-module for all~$j$.
	\begin{proof}
		It suffices to prove this under the additional assumption that the filtration is degreewise bounded above. In this case, $A$ is a filtered complex with strict differentials by \cite[Proposition~1.3.2]{deligne:hodge_ii}, and it is easy to check that this implies that~$\rH^j(A)$ is filtered.
	\end{proof}
\end{lem}

Let us say that~$A\in\SKom(\dA)$ has \emph{filtered cohomology objects} just when~$\rH^j(A)$ is filtered for all~$j$.

\begin{lem}\label{lem:iso_on_completed_cohomology}
	Suppose that $A\to B$ is a graded quasi-isomorphism in~$\SKom(\dA)$ and~$A$ has filtered cohomology objects.
	\begin{enumerate}
		\item $B$ also has filtered cohomology objects.
		\item The induced map $\rH^j(A)^\wedge\to\rH^j(B)^\wedge$ on the completions of cohomology objects with respect to their filtrations is an isomorphism for all~$j$.
	\end{enumerate}
	\begin{proof}
		For the first point, we note that~$A$ has filtered cohomology objects if and only if the coboundary map $\rH^i(\gr^j_{\rF}\!A)\to\rH^{i+1}(\rF^{j+1}\!A)$ is zero for all~$i$ and~$j$. So if~$A$ has filtered cohomology objects and $A\to B$ is a derived graded quasi-isomorphism, then we have a commuting square
		\begin{center}
			\begin{tikzcd}
				\rH^i(\gr^j_{\rF}\!A) \arrow[r,"0"]\arrow[d,"\wr"] & \rH^{i+1}(\rF^{j+1}\!A) \arrow[d] \\
				\rH^i(\gr^j_{\rF}\!B) \arrow[r] & \rH^{i+1}(\rF^{j+1}\!B)
			\end{tikzcd}
		\end{center}
		for all~$i$ and~$j$, where the top arrow is~$0$ and the left-hand vertical arrow is an isomorphism. So the bottom arrow is~$0$ too and~$B$ has filtered cohomology objects.
		
		For the second point, $A$ having filtered cohomology objects implies that the natural map
		\[
		\rH^i(\rF^j\!A)/\rH^i(\rF^k\!A)\to\rH^i(\rF^j\!A/\rF^k\!A)
		\]
		is an isomorphism for all~$i$ and all~$j\leq k$, where~$\rF^j\!A/\rF^k\!A$ is the cone of the map $\rF^j\!A\to\rF^k\!A$. Since $A\to B$ is a graded quasi-isomorphism, it induces quasi-isomorphisms $\rF^j\!A/\rF^k\!A\xrightarrow\sim\rF^j\!B/\rF^k\!B$ for all~$j\leq k$, and hence the induced maps
		\[
		\rH^i(\rF^j\!A)/\rH^i(\rF^k\!A)\to\rH^i(\rF^j\!B)/\rH^i(\rF^k\!B)
		\]
		are isomorphisms for all~$i$ and all~$j\leq k$. Taking the inverse limit over~$k$ gives the result.
	\end{proof}
\end{lem}

\begin{rmk}
	There are several different definitions of the filtered derived category in the literature; the definition we give above coincides with that used in~\cite{schapira-schneiders:filtered_derived_category} and~\cite{bhatt-morrow-scholze:thh}. If one localises $\SKom(\dA)$ instead at the graded quasi-isomorphisms, one obtains the filtered derived category as defined in e.g.~\cite{gwilliam-pavlov:filtered_derived_category}. The filtered derived categories used in \cite{scholze:relative_p-adic_hodge_theory} is also the localisation of~$\SKom(\dA)$ at graded quasi-isomorphisms: the term ``filtered derived category'' is not defined in \cite{scholze:relative_p-adic_hodge_theory}, but the definition was clarified to us in email correspondence with Peter Scholze. In particular, maps referred to as ``filtered quasi-isomorphisms'' in~\cite{scholze:relative_p-adic_hodge_theory} are what we refer to as graded quasi-isomorphisms above.
\end{rmk}

\begin{rmk}
	The above discussion doesn't strictly apply to the pro-\'etale site~$U_\proet$ of \cite{scholze:relative_p-adic_hodge_theory}, since~$U_\proet$ is large. The necessary modifications were described to us by Peter Scholze. One follows the approach in \cite{scholze:diamonds}, by picking a suitable cardinal (as in \cite[Lemma~4.1]{scholze:diamonds}) and considering only those adic spaces~$U$ which are \emph{$\kappa$-small} in the same sense as \cite[Definition~4.2]{scholze:diamonds}. If~$U$ is $\kappa$-small, then the $\kappa$-small pro-\'etale site~$U_{\proet,\kappa}$ is defined to be the restriction of the pro-\'etale site to the full subcategory of $\kappa$-small objects. This is a small site.
	
	One then checks, from \cite[Proposition~8.2]{scholze:diamonds} that the pullback functor
	\[
	c_{\kappa,\kappa'}^{-1}\colon U_{\proet,\kappa}^\sim \to U_{\proet,\kappa'}^\sim
	\]
	on categories of sheaves is fully faithful for~$\kappa<\kappa'$, and the natural map
	\begin{equation}\label{eq:increase_kappa}
		\dF \to \rR\!c_{\kappa,\kappa',*}c_{\kappa,\kappa'}^{-1}\dF
	\end{equation}
	is an isomorphism for all abelian sheaves~$\dF$ on $U_{\proet,\kappa}$. We call the (large) colimit
	\[
	U_\proet^\sim\colonequals \varinjlim_\kappa U_{\proet,\kappa}^\sim
	\]
	the category of \emph{small} sheaves on~$U_\proet$, which is equivalent to the category of sheaves on~$U_\proet$ which arise via pullback from a sheaf on some~$U_{\proet,\kappa}$. All of the sheaves discussed in \S\ref{ss:relative_p-adic_hodge} are small, and it follows from isomorphy of~\eqref{eq:increase_kappa} that cohomology and derived pushforwards of small sheaves can be computed on the level of sheaves on~$U_{\proet,\kappa}$ for a suitable~$\kappa$. So the discussion in this section can be applied to small sheaves on~$U_\proet$, even though it is large.
\end{rmk}
%%%%%%%%%%%%%%%%%%%%%%%%%%%%%%%%%%%%%%%%%%%
%%%%%%%%%%%%%%%%%%%%%%%%%%%%%%%%%%%%%%%%%%%
%%%%%% Bibliography %%%%%%%%%%%%%%%%%%%%%%%
%%%%%%%%%%%%%%%%%%%%%%%%%%%%%%%%%%%%%%%%%%%

\bibliographystyle{alphaSGA}
\nocite{scholze:relative_p-adic_hodge_theory_erratum}
\bibliography{references}

\end{document}